\documentclass[10pt]{amsart}

\pdfoutput=1

\usepackage[utf8]{inputenc}
\usepackage[T1]{fontenc}
\usepackage[backend=biber,
            isbn=false,
            doi=false,
            maxbibnames=9,
            giveninits=true,
            style=alphabetic,
            citestyle=alphabetic]{biblatex}
\usepackage{url}
\renewbibmacro{in:}{}
\bibliography{paper.bib}

\usepackage{xcolor}
\definecolor{e-mail}{rgb}{0,.40,.80}
\definecolor{reference}{rgb}{.20,.60,.22}
\definecolor{citation}{rgb}{0,.40,.80}
\usepackage{todonotes,setspace}

\usepackage[colorlinks=true,
            linkcolor=reference,
            citecolor=citation,
            urlcolor=e-mail,
            breaklinks=true]{hyperref}
\usepackage[all]{xy}

\usepackage{amssymb}
\usepackage{eucal}
\usepackage{cleveref}
\usepackage{stmaryrd}
\usepackage{verbatim}
\usepackage{fullpage}
\usepackage{mathrsfs}
\usepackage{xifthen}

\usepackage{tikz}
\usetikzlibrary{decorations.pathmorphing, arrows}
\usepackage{tikz-cd}
\usepackage{mathtools}
\tikzset
{
  partial ellipse/.style args={#1:#2:#3}{
    insert path={+ (#1:#3) arc (#1:#2:#3)}
  }
}

\usepackage[bbgreekl]{mathbbol}
\DeclareMathSymbol\bbDelta  \mathord{bbold}{"01}

\newcommand{\cA}{\mathcal{A}}
\newcommand{\cC}{\mathcal{C}}
\newcommand{\cD}{\mathcal{D}}
\newcommand{\cE}{\mathcal{E}}

\newcommand{\cF}{\mathcal{F}}

\newcommand{\cL}{\mathcal{L}}
\newcommand{\cM}{\mathcal{M}}
\newcommand{\cO}{\mathcal{O}}
\newcommand{\cS}{\mathcal{S}}
\newcommand{\cT}{\mathcal{T}}
\newcommand{\cU}{\mathcal{U}}
\newcommand{\cW}{\mathcal{W}}

\newcommand{\fg}{\mathfrak{g}}
\newcommand{\fZ}{\mathfrak{Z}}

\newcommand{\C}{\mathrm{C}}
\renewcommand{\d}{\mathrm{d}}

\renewcommand{\H}{\mathrm{H}}

\renewcommand{\L}{\mathrm{L}}

\newcommand{\R}{\mathrm{R}}
\newcommand{\T}{\mathrm{T}}
\newcommand{\UT}{\mathrm{UT}}

\newcommand{\bA}{\mathbf{A}}
\newcommand{\bD}{\mathbf{D}}
\newcommand{\bE}{\mathbb{E}}
\newcommand{\bF}{\mathbf{F}}
\newcommand{\bK}{\mathbf{K}}

\newcommand{\Q}{\mathbf{Q}}
\newcommand{\bR}{\mathbf{R}}
\newcommand{\bS}{\mathbb{S}}

\newcommand{\Z}{\mathbf{Z}}

\newcommand{\Alg}{\mathrm{Alg}}

\newcommand{\Bord}{\mathrm{Bord}}
\newcommand{\bu}{\mathbf{1}}

\newcommand{\Catperf}{\mathrm{Cat}^{\mathrm{perf}}_\infty}

\newcommand{\ch}{\mathrm{ch}}

\newcommand{\coev}{\mathrm{coev}}

\DeclareMathOperator*{\colim}{colim}
\newcommand{\coCorr}{\mathrm{coCorr}}
\newcommand{\Corr}{\mathrm{Corr}}

\newcommand{\Disc}{\mathcal{D}\mathrm{isc}}
\newcommand{\End}{\mathrm{End}}
\newcommand{\ev}{\mathrm{ev}}

\newcommand{\fd}{\mathrm{fd}}

\newcommand{\flip}{\mathrm{flip}}
\newcommand{\FM}{\mathrm{FM}}
\newcommand{\FE}{\mathrm{Fig}(8)}
\newcommand{\fr}{\mathrm{fr}}
\newcommand{\Fun}{\mathrm{Fun}}

\newcommand{\id}{\mathrm{id}}
\newcommand{\Ind}{\mathrm{Ind}}
\newcommand{\Int}{\mathrm{int}}
\newcommand{\HH}{\mathrm{HH}}
\newcommand{\Ho}{\mathrm{Ho}}
\newcommand{\hofib}{\mathrm{hofib}}
\newcommand{\Hom}{\mathrm{Hom}}
\newcommand{\bHom}{\mathbf{H}\mathrm{om}}

\newcommand{\Lift}{\mathrm{Lift}}

\newcommand{\LocSys}{\mathrm{LocSys}}
\newcommand{\Map}{\mathrm{Map}}
\newcommand{\Man}{\mathcal{M}\mathrm{an}}
\newcommand{\LMod}{\mathrm{LMod}}
\newcommand{\Mod}{\mathrm{Mod}}
\newcommand{\nc}{\mathrm{nc}}
\newcommand{\OC}{\mathcal{OC}}
\newcommand{\op}{\mathrm{op}}
\newcommand{\ori}{\mathrm{or}}
\newcommand{\ost}{\mathring{\mathrm{st}}}
\newcommand{\Perf}{\mathrm{Perf}}

\newcommand{\PrSt}{\mathrm{Pr}^{\mathrm{St}}}
\newcommand{\PSh}{\mathrm{PSh}}
\newcommand{\PT}{\mathrm{PT}}
\newcommand{\pt}{\mathrm{pt}}
\newcommand{\QCoh}{\mathrm{QCoh}}

\newcommand{\SH}{\mathrm{SH}}
\newcommand{\Shv}{\mathrm{Shv}}
\newcommand{\Sing}{\mathrm{Sing}}
\newcommand{\SO}{\mathrm{SO}}
\DeclareMathOperator{\Spec}{Spec}
\newcommand{\Sp}{\mathrm{Sp}}

\newcommand{\Sym}{\mathrm{Sym}}
\newcommand{\Td}{\mathrm{Td}}
\newcommand{\Th}{\mathrm{Th}}
\newcommand{\THH}{\mathrm{THH}}
\newcommand{\Top}{\mathrm{Top}}
\newcommand{\tr}{\mathrm{tr}}

\newcommand{\bWh}{\mathbf{Wh}}
\newcommand{\Wh}{\mathrm{Wh}}
\newcommand{\const}{\mathrm{const}}

\newcommand{\Emb}{\mathrm{Emb}}

\newcommand{\dual}{\mathrm{dual}}
\newcommand{\PrDual}{\mathrm{Pr}^\dual}

\newtheorem{maintheorem}{Theorem}

\newtheorem{maincorollary}[maintheorem]{Corollary}
\newtheorem{thm}{Theorem}[section]

\newtheorem{prop}[thm]{Proposition}
\crefname{prop}{proposition}{propositions}
\newtheorem{cor}[thm]{Corollary}
\newtheorem{lm}[thm]{Lemma}

\theoremstyle{definition}
\newtheorem{defn}[thm]{Definition}
\theoremstyle{remark}
\newtheorem{remark}[thm]{Remark}
\newtheorem*{remarknonum}{Remark}
\newtheorem{example}[thm]{Example}

\newcommand{\defterm}[1]{\textbf{\emph{#1}}}
\newcommand{\adj}[2]{
\xymatrix{
#1 \ar@<.5ex>[r] & #2 \ar@<.5ex>[l]
}
}

\newcommand{\importpiclib}[1][]{
	
	\ifthenelse{\isempty{#1}}{
		
		\newcommand{\newpic}[2]{
			\newcommand{##1}{##2}
		}
		
	}{
		
		\newcommand{\newpic}[2]{
			\newcommand{##1}{##2}
			\string##1 = ##1}
		
	}

	\newpic{\pzetam}{
		\begin{tikzpicture}
			\draw[red, fill=red] (0, 0) circle (2pt);
			\draw[dashed] (-1,0) -- (0,0);
		\end{tikzpicture}
	}
	
	\newpic{\phhcm}{
		\begin{tikzpicture}
			\draw (0, 0) circle (0.2cm);
			\draw[red, fill=red] (0.2, 0) circle (2pt);
			\draw[dashed] (-1.2,0) -- (-0.2,0);
			\draw[fill] (-0.2,0) circle (2pt);
		\end{tikzpicture}
	}
	
	\newpic{\phhm}{
		\begin{tikzpicture}
			\draw (0.0,0) .. controls ++(0:.7) and ++(0: .5) .. (0,.7)
			.. controls ++(0:-.5) and ++(0: -.7) .. (-0.0,0);
			\draw[dashed] (-1,0) -- (0,0);
			\draw[fill] (0,0) circle (2pt);
		\end{tikzpicture}
	}
	
	\newpic{\pfigeightbm}{
		\begin{tikzpicture}
			\draw (0, 0.4cm) circle (0.4cm);
			\draw (0, -0.4cm) circle (0.4cm);
			\draw[dashed, shift={(0,0.4cm)}] (230:0.4cm) -- ++(-1,0);
			\draw[fill, shift={(0,0.4cm)}] (230:0.4cm) circle (2pt);
		\end{tikzpicture}
	}
	
	\newpic{\ptwocircm}{
		\begin{tikzpicture}
			\draw (0, 0.5cm) circle (0.4cm);
			\draw (0, -0.5cm) circle (0.4cm);
			\draw[dashed] (0,0.1) -- ++(-1,0);
			\draw[fill] (0,0.1) circle (2pt);
		\end{tikzpicture}
	}
	
	\newpic{\ppointm}{
		\begin{tikzpicture}
			\draw[dashed] (0,0) -- ++(-1,0);
			\draw[fill] (0,0) circle (2pt);
		\end{tikzpicture}
	}
	
	\newpic{\ppointmun}{
		\begin{tikzpicture}
			\draw[dashed] (0,0.5) -- ++(-1,0);
			\draw[fill] (0,0.5) circle (2pt);
			\draw[fill] (0,-0.5) circle (2pt);
		\end{tikzpicture}
	}
	
	\newpic{\phhc}{
		\begin{tikzpicture}[scale=0.4]
			\draw (0, 0) circle (0.5cm);
			\draw[red, fill=red] (0.5, 0) circle (2pt);
		\end{tikzpicture}
	}
	
	\newpic{\ppoint}{
		\begin{tikzpicture}[scale=0.4]
			\draw[fill] (0, 0) circle (2pt);
		\end{tikzpicture}
	}
	
	\newpic{\pdpoint}{
		\begin{tikzpicture}[scale=0.4]
			\draw[fill] (0, 0.5) circle (2pt);
			\draw[fill] (0, -0.5) circle (2pt);
		\end{tikzpicture}
	}
	
	\newpic{\pcirc}{
		\begin{tikzpicture}[scale=0.4]
			\draw (0, 0) circle (0.5cm);
		\end{tikzpicture}
	}
	
	\newpic{\pcircirc}{
		\begin{tikzpicture}[scale=0.4]
			\draw (0, 0) circle (0.5cm);
			\draw (1, 0) circle (0.5cm);
		\end{tikzpicture}
	}
	
	\newpic{\pcircircirc}{
		\begin{tikzpicture}[scale=0.4]
			\draw (0, 0) circle (0.5cm);
			\draw (0.8, 0.6) circle (0.5cm);
			\draw (1.4, -0.2) circle (0.5cm);
		\end{tikzpicture}
	}
	
	\newpic{\pcircircircsep}{
		\begin{tikzpicture}[scale=0.4]
			\draw (0, 0) circle (0.5cm);
			\draw (0.8, 0.6) circle (0.5cm);
			\draw (1.6, -0.4) circle (0.5cm);
		\end{tikzpicture}
	}
	
	\newpic{\pzeta}{
		\begin{tikzpicture}[scale=0.4]
			\draw[red, fill=red] (0.5, 0) circle (2pt);
		\end{tikzpicture}
	}
	
	\newpic{\pfigeight}{
		\begin{tikzpicture}[scale=0.4]
			\draw (0, 0.5cm) circle (0.5cm);
			\draw (0, -0.5cm) circle (0.5cm);
		\end{tikzpicture}
	}
	
	\newpic{\pohoh}{
		\begin{tikzpicture}[scale=0.4]
			\draw (0, 0.6cm) circle (0.5cm);
			\draw (0, -0.6cm) circle (0.5cm);
		\end{tikzpicture}
	}
	
	\newpic{\povalcirc}{
		\begin{tikzpicture}[scale=0.4]
			\draw (1.6, -0.4) circle (0.5cm);
			\begin{scope}[scale=1.25, rotate=36.8, shift={(-8,-.4)}]
				\draw (8, 0) arc (270:90:0.4cm) -- (8.8, 0.8) arc (90:-90:0.4cm) -- (8, 0);
			\end{scope}
		\end{tikzpicture}
	}
	
	\newpic{\povalcircnsep}{
		\begin{tikzpicture}[scale=0.4]
			\draw (1.4, -0.2) circle (0.5cm);
			\begin{scope}[scale=1.25, rotate=36.8, shift={(-8,-.4)}]
				\draw (8, 0) arc (270:90:0.4cm) -- (8.8, 0.8) arc (90:-90:0.4cm) -- (8, 0);
			\end{scope}
		\end{tikzpicture}
	}
	
	\newpic{\pxfigeightm}{
		\begin{tikzpicture}[scale=0.4]
			\draw (-{sqrt(2)/4},-{sqrt(2)/4}) arc (-225:45:0.5cm);
			\draw ({sqrt(2)/4},{sqrt(2)/4}) arc (-45:225:0.5cm);
			\draw ({sqrt(2)/4},-{sqrt(2)/4}) -- (-{sqrt(2)/4},{sqrt(2)/4});
			\draw (-{sqrt(2)/4},-{sqrt(2)/4}) -- ({sqrt(2)/4},{sqrt(2)/4});
			\draw[red, fill=red] (0, 0) circle (2pt);
		\end{tikzpicture}
	}
	
	\newpic{\ptopeightm}{
		\begin{tikzpicture}[scale=0.4]
			\draw ({sqrt(2)/4},{sqrt(2)/4}) arc (-45:225:0.5cm);
			\draw ({sqrt(2)/4},{sqrt(2)/4}) -- (0,0) -- (-{sqrt(2)/4},{sqrt(2)/4});
			\draw[red, fill=red] (0, 0) circle (2pt);
		\end{tikzpicture}
	}
	
	\newpic{\pxfigeight}{
		\begin{tikzpicture}[scale=0.4]
			\draw (-{sqrt(2)/4},-{sqrt(2)/4}) arc (-225:45:0.5cm);
			\draw ({sqrt(2)/4},{sqrt(2)/4}) arc (-45:225:0.5cm);
			\draw ({sqrt(2)/4},-{sqrt(2)/4}) -- (-{sqrt(2)/4},{sqrt(2)/4});
			\draw (-{sqrt(2)/4},-{sqrt(2)/4}) -- ({sqrt(2)/4},{sqrt(2)/4});
		\end{tikzpicture}
	}
	
	\newpic{\pxfigeightex}{
		\begin{tikzpicture}[scale=0.4]
			\draw (-{sqrt(2)/4},-{sqrt(2)/4}) arc (-225:45:0.5cm);
			\draw ({sqrt(2)/4},{sqrt(2)/4}) arc (-45:225:0.5cm);
			\draw ({sqrt(2)/4},-{sqrt(2)/4}) -- (-{sqrt(2)/4},{sqrt(2)/4});
			\draw (-{sqrt(2)/4},-{sqrt(2)/4}) -- ({sqrt(2)/4},{sqrt(2)/4});
			\draw (0,0) .. controls (1,.75) and (1,-.75) .. (0,0);
		\end{tikzpicture}
	}
	
	\newpic{\ptopeight}{
		\begin{tikzpicture}[scale=0.4]
			\draw ({sqrt(2)/4},{sqrt(2)/4}) arc (-45:225:0.5cm);
			\draw ({sqrt(2)/4},{sqrt(2)/4}) -- (0,0) -- (-{sqrt(2)/4},{sqrt(2)/4});
		\end{tikzpicture}
	}
	
	\newpic{\ptopeightpoint}{
		\begin{tikzpicture}[scale=0.4]
			\draw ({sqrt(2)/4},{sqrt(2)/4}) arc (-45:225:0.5cm);
			\draw ({sqrt(2)/4},{sqrt(2)/4}) -- (0,0) -- (-{sqrt(2)/4},{sqrt(2)/4});
			\draw[fill] (0, -0.5) circle (2pt);
		\end{tikzpicture}
	}
	
	\newpic{\ptopeightex}{
		\begin{tikzpicture}[scale=0.4]
			\draw ({sqrt(2)/4},{sqrt(2)/4}) arc (-45:225:0.5cm);
			\draw ({sqrt(2)/4},{sqrt(2)/4}) -- (0,0) -- (-{sqrt(2)/4},{sqrt(2)/4});
			\draw (0,0) .. controls (1,.75) and (1,-.75) .. (0,0);
		\end{tikzpicture}
	}
	
	\newpic{\ptheta}{
		\begin{tikzpicture}[scale=0.4]
			\draw (0,0) circle (0.5cm);
			\draw (-0.5,0) -- (0.5,0);
		\end{tikzpicture}
	}
}

\begin{document}
\title{Simple homotopy invariance of the loop coproduct}
\address{Trinity College Dublin, Dublin, Ireland}
\email{naeff@tcd.ie}
\author{Florian Naef}
\address{School of Mathematics, University of Edinburgh, Edinburgh, UK}
\email{p.safronov@ed.ac.uk}
\author{Pavel Safronov}
\begin{abstract}
We prove a transformation formula for the Goresky--Hingston loop coproduct in string topology under homotopy equivalences of manifolds. The formula involves the trace of the Whitehead torsion of the homotopy equivalence. In particular, it implies that the loop coproduct is invariant under simple homotopy equivalences. In a sense, our results determine the Dennis trace of the simple homotopy type of a closed manifold from its framed configuration spaces of $\leq 2$ points.
We also explain how the loop coproduct arises as a secondary operation in a 2-dimensional TQFT which elucidates a topological origin of the transformation formula.
\end{abstract}
\maketitle

\section*{Introduction}

\subsection*{String topology}

In his description of the Poisson bracket on the character variety of an oriented (connected) surface $\Sigma$, Goldman \cite{Goldman} introduced a Lie bracket on the abelian group $\fg(\Sigma)$ of homotopy classes of loops $S^1\rightarrow \Sigma$. The Goldman bracket is defined in terms of counting intersection points of two loops. Turaev \cite{TuraevSkein} defined a Lie cobracket on $\fg(\Sigma)$ in terms of counting self-intersection points, so that $\fg(\Sigma)$ becomes a Lie bialgebra. In fact, the Lie cobracket takes values in $\wedge^2\overline{\fg}(\Sigma)$, where $\overline{\fg}(\Sigma)=\fg(\Sigma)/\Z$ is the quotient by the class of contractible loops.

The generalizations of these operations to higher-dimensional manifolds is given by the string topology operations \cite{ChasSullivan1,ChasSullivan2,Sullivan,GoreskyHingston}. Let $M$ be a closed oriented $d$-manifold and $LM=\Map(S^1, M)$ its free loop space. We will primarily be interested in the following two operations:
\begin{itemize}
    \item Loop product $\wedge\colon \H_\bullet(LM)\otimes \H_\bullet(LM)\rightarrow \H_{\bullet-d}(LM)$.
    \item Loop coproduct $\vee\colon \H_{\bullet+d-1}(LM)\rightarrow \H_\bullet(LM\times LM, M\times LM\cup LM\times M)$. For instance, if $\H_\bullet(LM, M)$ has no torsion, the target of the loop coproduct is $\H_\bullet(LM, M)\otimes \H_\bullet(LM, M)$.
\end{itemize}

We refer to \cite{NaefRiveraWahl} for an exposition of different approaches to string topology operations. Along with the BV operator $\Delta\colon \H_\bullet(LM)\rightarrow \H_{\bullet+1}(LM)$ given by the loop rotation, the loop product and coproduct give rise to the string bracket and cobracket on $\H^{S^1}_\bullet(LM)$. For $M=\Sigma$ a surface the string bracket and string cobracket reduce to the Goldman bracket and Turaev cobracket on $\fg(\Sigma)\cong \H^{S^1}_0(L\Sigma)$ and $\overline{\fg}(\Sigma)\cong \H^{S^1}_0(L\Sigma, \Sigma)$.

\subsection*{Homotopy invariance of string topology}

The original construction of the string topology operations involved a subtle infinite-dimensional intersection theory on the loop space $LM$. In particular, these constructions require $M$ to be a smooth manifold. It was realized early on that both the BV operator and the loop product are homotopy invariant \cite{CohenKleinSullivan,Crabb,GruherSalvatore,FelixThomas}. Namely, if $f\colon M\rightarrow N$ is an orientation-preserving homotopy equivalence of closed oriented $d$-manifolds, then the diagram
\[
\xymatrix{
\H_\bullet(LM)\otimes \H_\bullet(LM)\ar^-{\wedge_M}[r] \ar^{Lf\otimes Lf}[d] & \H_{\bullet-d}(LM) \ar^{Lf}[d] \\
\H_\bullet(LN)\otimes \H_\bullet(LN)\ar^-{\wedge_N}[r] & \H_{\bullet-d}(LN)
}
\]
commutes.

However, it was also conjectured by Sullivan, see \cite[Postscript]{SullivanNotes} and \cite{CohenKleinSullivan}, that the full range of string topology operations is not homotopy invariant. The previous results in the direction of this question are as follows. If $M$ is simply-connected, then over the rationals/reals the coproduct was shown to be homotopy invariant in \cite{NaefWillwacher, RiveraWang}. It was also shown to be invariant under homotopy equivalences satisfying certain regularity conditions in \cite{HingstonWahl2}.

The case of non-simply-connected manifolds is more subtle: the first author showed \cite{Naef} that in the case of a homotopy equivalence $f\colon L(7, 1)\rightarrow L(7, 2)$ of 3-dimensional lens spaces, the loop coproduct is not preserved.

The above homotopy equivalence of 3-dimensional lens spaces is the simplest example of a homotopy equivalence $f\colon N\rightarrow M$ with a nonvanishing \emph{Whitehead torsion} $\tau(f)$, an invariant lying in the Whitehead group $\Wh(\pi_1(M)) = K_1(\Z[\pi_1(M)])/(\pm\pi_1(M))$. The Dennis trace defines a map from $K$-theory to Hochschild homology, which in this case produces a map $\tr\colon \Wh(\pi_1(M))\rightarrow \H_1(LM, M)$. Consider the trace of the Whitehead torsion $\tr(\tau(f))\in\H_1(LM, M)$. Denote its image under the antidiagonal map $LM\rightarrow LM\times LM$ given by $\gamma\mapsto (\gamma^{-1}, \gamma)$ by $\overline{\nu}' \otimes \nu'' \in \oplus_{i+j=1} \left(\H_i(LM, M) \otimes \H_j(LM)\right)$ and similarly for $\nu'\otimes \overline{\nu}''\in \oplus_{i+j=1} \left(\H_i(LM) \otimes \H_j(LM, M)\right)$. Our first main result is a complete explanation of the non-homotopy invariance of the loop coproduct. 

\begin{maintheorem}[{\Cref{cor:stringcoproducthomotopyinvariance}}] \label{maintheorem:stringcoproducterror}
Let $f\colon M\rightarrow N$ be an orientation-preserving homotopy equivalence of closed oriented $d$-manifolds. For every $\alpha\in \H_{n+d-1}(LM, M)$ we have
\[
\vee_M(f(\alpha)) - f(\vee_N(\alpha)) = (-1)^{n+1}\overline{\nu}'\otimes (\nu''\wedge_N f(\alpha)) - (f(\alpha)\wedge_N \nu')\otimes \overline{\nu}''.
\]
\end{maintheorem}

\begin{remarknonum}
    The same result also holds for any other homology theory over which $M$ and $N$ are oriented and $f$ is orientation-preserving. On the level of spectra the loop coproduct is a map
    \[
    \Sigma(LM^{-\T M}) \to \Sigma^\infty LM/M \wedge \Sigma^\infty  LM/M,
    \]
    where $LM^{-\T M}$ is the corresponding Thom spectrum. The same transformation formula holds for this operation as well, now for arbitrary (i.e. not necessarily orientation-preserving) homotopy equivalences.
\end{remarknonum}

Homotopy equivalences with vanishing Whitehead torsion are known as \emph{simple homotopy equivalences}, so we obtain that the loop coproduct is invariant under orientation-preserving simple homotopy equivalences. This is the first result to our knowledge which relates string topology operations to $K$-theoretic invariants such as the Whitehead torsion.

\begin{remarknonum}
While preparing this manuscript, we were informed that Lea Kenigsberg and Noah Porcelli have obtained an independent proof of a variant of Theorem A. The details will appear in their upcoming preprint \cite{KenigsbergPorcelli}.
\end{remarknonum}

To explain the appearance of the Whitehead torsion, let us recall from \cite{Crabb} that the loop product can be viewed as a fiberwise version of the intersection product $\H_\bullet(M)\otimes \H_\bullet(M)\rightarrow \H_{\bullet-d}(M)$ on the manifold. Namely, given a space $E\rightarrow M\times M$ the fiberwise version of the intersection product is $\H_\bullet(E)\rightarrow \H_{\bullet-d}(E\times_{M\times M} M)$. Setting $E = LM\times LM\xrightarrow{\ev\times \ev} M\times M$, where $\ev\colon LM\rightarrow M$ is the evaluation of a loop at the basepoint, by \cref{prop:stringproductintersection} we get the loop product as the composite
\[\H_\bullet(LM)\otimes \H_\bullet(LM)\longrightarrow \H_{\bullet-d}(LM\times_M LM)\longrightarrow \H_{\bullet-d}(LM),\]
where the last map is given by the composition of loops. The first map is a fiberwise intersection product and may be constructed from a pairing
\[\epsilon_p\colon \Z_M\boxtimes \Z_M\rightarrow \Delta_\sharp \Z_M[d]\]
in the $\infty$-category of local systems of chain complexes on $M\times M$ which is invariant under orientation-preserving homotopy equivalences.

To present a similar point of view on the loop coproduct, we define the \emph{relative intersection product} in \cref{sect:relativeintersection} following \cite[Section 4.1]{NaefRiveraWahl}. To define it we consider a commutative diagram
\[
\xymatrix{
F \ar[r] \ar[d] & E \ar[d] \\
M \ar[r] & M\times M
}
\]
The \emph{relative intersection product}, see \cref{def:relint}, lifts the intersection product to relative homology:
\[\H_\bullet(E, F)\longrightarrow \H_{\bullet-d}(E\times_{M\times M} M, F).\]
We have a commutative diagram
\[
\xymatrix{
LM\sqcup LM \ar^-{J_0\sqcup J_1}[r] \ar[d] & LM \ar^{(\ev_0, \ev_{1/2})}[d] \\
M \ar^-{\Delta}[r] & M\times M
}
\]
where $\ev_t\colon LM\rightarrow M$ are the evaluation maps at time $t\in [0, 1]$ and $J_i\colon LM\rightarrow LM$ are the reparametrization maps which make the loop constant either for times $[0, 1/2]$ or $[1/2, 1]$. Then the loop coproduct $\H_{\bullet+d-1}(LM)\rightarrow \H_\bullet(LM\times LM, M\times LM\cup LM\times M)$ is defined in terms of the relative intersection product with respect to the above commutative diagram, see \cref{def:stringcoproduct}.

The relative intersection product consists of two essential ingredients: the pairing $\epsilon_p$ as well as a trivialization of the image of a certain canonical element, the Hochschild homology Euler characteristic, $\chi_{\HH}(M)\in\C_0(LM)$ in $\C_\bullet(LM, M)$. In other words, the second piece of data consists of a lift of $\chi_{\HH}(M)$ along the inclusion of the constant loops $\C_\bullet(M)\rightarrow \C_\bullet(LM)$.

The Hochschild homology Euler characteristic $\chi_{\HH}(M)\in\C_0(LM)$ is well-defined for any finitely dominated space $M$ and can be understood as the class of the constant local system $\Z_M\in\LocSys(M)$ (which is a compact object since $M$ is finitely dominated) in Hochschild homology $\HH_\bullet(\LocSys(M)^\omega)\cong \dim(\LocSys(M))$. Moreover, the image of $\chi_{\HH}(M)$ in $\C_0(\pt)=\Z$ under the projection $LM\rightarrow \pt$ gives the Euler characteristic of $M$.

Let us assume for simplicity that $M$ is connected. Then we may define a version of $\chi_{\HH}(M)$ over the sphere spectrum (see \cref{sect:assembly} for details):
\begin{itemize}
    \item There is an element
    \[\chi_{\THH}(M)\in\Omega^\infty\Sigma^\infty_+ LM\cong \THH(\Sigma^\infty_+\Omega M),\]
    \emph{the $\THH$ Euler characteristic}. It coincides with the free loop transfer along $M\rightarrow \pt$ in the sense of \cite{LindMalkiewich}.
    \item There is an element
    \[\chi_{A}(M)\in A(M) = K(\Sigma^\infty_+ \Omega M)\]
    in Waldhausen's $A$-theory, \emph{the $A$-theoretic Euler characteristic}. It coincides with the homotopy-invariant Euler characteristic from \cite[Section 6]{DwyerWeissWilliams}. Under the Dennis trace map
    \[\tr\colon A(M)\rightarrow \Omega^\infty\Sigma^\infty_+ LM\]
    we have $\chi_A(M)\mapsto \chi_{\THH}(M)$.
\end{itemize}

The inclusion of constant loops $\C_\bullet(M)\rightarrow \C_\bullet(LM)$ is an instance of the assembly map \cite{WeissWilliams}: it is a universal colimit-preserving approximation to the functor $M\mapsto \C_\bullet(LM)$. There are similar assembly maps
\[\alpha\colon \Omega^\infty_+ M\otimes \bA(\pt)\longrightarrow \bA(M),\qquad \alpha\colon \Sigma^\infty_+ M\longrightarrow \Sigma^\infty_+ LM\]
in $A$-theory and $\THH$ related by the Dennis trace map.

The study of lifts of the $A$-theoretic Euler characteristic $\chi_A(M)$ along the assembly map is at the heart of simple homotopy theory. Namely, the fiber of the assembly map at $\chi_A(M)$ is the space of the structures of a simple homotopy type on a given finitely dominated homotopy type. An explicit description of the fiber of the assembly map in terms of the space of PL h-cobordisms is given by the parametrized h-cobordism theorem \cite{JahrenRognesWaldhausen}. We use the following two basic facts from simple homotopy theory:
\begin{enumerate}
    \item If $M$ is a finite polyhedron (equivalently, a finite CW complex), there is a lift $\lambda_{\Wh}(M)$ of $\chi_A(M)$ along the assembly map, see \cref{def:Whiteheadlift}.
    \item If $f\colon M_1\rightarrow M_2$ is a homotopy equivalence of finite polyhedra, the difference of the lifts defines an element $\tau(f)\in \Wh(\pi_1(M_1))$ in the Whitehead group given by the Whitehead torsion, see \cref{prop:whiteheadtorsion}.
\end{enumerate}

The geometric input to the construction of the relative intersection product is a model of the Thom collapse for the diagonal $\epsilon_p$ which fixes the diagonal up to homotopy; namely, we require the composite $\Delta_\sharp \Z_M\rightarrow \Z_M\boxtimes \Z_M\xrightarrow{\epsilon_p} \Delta_\sharp \Z_M[d]$ to be the pushforward of a map $\Z_M\rightarrow \Z_M[d]$ (the Euler class in $\C^d(M)$) along the diagonal. Geometrically, for $M$ a closed manifold we use a model of the Thom collapse given by the commutative diagram
\begin{equation}
\label{diag:poincdiagonal}    
\xymatrix{
\UT M \ar[r] \ar[d] & \FM_2(M) \ar[d] \\
M \ar[r] & M\times M,
}
\end{equation}
where $\FM_2(M)$ is the Fulton--MacPherson compactification of the configuration space of two points $M\times M\setminus \Delta$ and $\UT M\rightarrow \FM_2(M)$ is the inclusion of its boundary, which is the unit tangent bundle of $M$. This diagram is a pushout of spaces; so, passing to relative suspension spectra, we obtain a commutative triangle
\[
\xymatrix{
\Delta_\sharp \bS_M\ar[r] \ar_{\Delta_\sharp e(M)}[d] & \bS_{M\times M} \ar^{\epsilon_p}[dl] \\
\Delta_\sharp \bS^{\T M} &
}
\]
of parametrized spectra over $M\times M$. We explain in \cref{sect:PTlift} that this diagram gives rise (and is in fact equivalent) to a lift $\lambda_{\PT}(M)$ of the $\THH$ Euler characteristic $\chi_{\THH}(M)$ along the assembly map.

So, at this point we have two lifts of $\chi_{\THH}(M)$: one, $\lambda_{\PT}(M)$, constructed using the geometric model of the Pontryagin--Thom collapse along the diagonal and another one, $\tr(\lambda_{\Wh}(M))$, constructed using a triangulation of $M$. The proof of \cref{maintheorem:stringcoproducterror} then reduces to the following statement.

\begin{maintheorem}[{\Cref{thm:florian}}] \label{maintheorem:florian}
Let $M$ be a closed $d$-manifold. Then the lifts $\lambda_{\PT}(M)$ and $\tr(\lambda_{\Wh}(M))$ of the $\THH$ Euler characteristic $\chi_{\THH}(M)$ are equivalent.
\end{maintheorem}

To prove \cref{maintheorem:florian}, we show that the lift $\lambda_{\PT}(M)$ is compatible with triangulations. As the assembly map is an equivalence when $M$ is contractible and $\lambda_{\Wh}(M)$ is constructed by covering the manifold by open stars of simplices of the triangulation, this gives the result.

\subsection*{Connection to embedding calculus}

\Cref{maintheorem:florian} has the following implication for embedding calculus \cite{WeissEmbeddings}. Let $\Man_d$ be the $\infty$-category of smooth manifolds and embeddings. Let $\Disc_d\subset \Man_d$ be the full subcategory consisting of finite disjoint unions of $\bR^d$. Following the perspective of \cite{deBritoWeiss, KrannichKupers} we understand embedding calculus as the restricted Yoneda embedding $\Man_d\rightarrow \PSh(\Disc_d)=\Fun(\Disc_d^{\op}, \cS)$ given by sending a manifold $M$ to $E_M := \Emb(-, M)$, the collection of spaces of embeddings of Euclidean spaces into $M$. For instance, for a $\Disc_d$-algebra $A$ its factorization homology over $M$ as in \cite{AyalaFrancisFactorization} is completely determined by $E_M$:
\[\int_M A \cong E_M\otimes_{\Disc_d} A.\]

On the level of homotopy types, our diagram \eqref{diag:poincdiagonal} can be written in terms of the $\Disc_d$-presheaf $E_M$ as follows:
\[
\xymatrix{
E_M(\bR^d) \times_{O(d)} E_{\bR^d}(\bR^d \sqcup \bR^d)/O(d)^{\times 2} \ar[r] \ar[d] & E_M(\bR^d \sqcup \bR^d) /O(d)^{\times 2} \ar[d] \\
M \ar[r] & M\times M
}
\]
so that the lift $\lambda_{\PT}(M)$ is an invariant of $E_M$. We thus obtain the following consequence.

\begin{maincorollary}
Let $f \colon N \to M$ be a homotopy equivalence between closed $d$-manifolds with Whitehead torsion $\tau(f)\in\Wh(\pi_1(M))$. If $\tr(\tau(f)) \neq 0\in \H_1(LM, M)$, then $f$ cannot be extended to an equivalence $E_N\rightarrow E_M$ of $\Disc_d$-presheaves.
\end{maincorollary}

As explained above, $\tr(\tau(f))$ only depends on the corresponding diagrams \eqref{diag:poincdiagonal}. Diagrams of that shape (satisfying a certain non-degeneracy condition) are also studied in \cite{klein2008poincare} (see also \cite{klein2022poincare} for a comparison to our setting) under the name of \emph{Poincar\'e diagonals}. It is furthermore suggested there that a $C_2$-refinement of $\tr(\lambda_{\PT}(M))$ is a complete invariant for $d$ sufficiently large.

\subsection*{String topology TQFT}

Cohen and Godin \cite{CohenGodin} described a 2-dimensional TQFT $Z$ whose state space $Z(S^1)$ is isomorphic to $\H_\bullet(LM)$ and whose pair-of-pants product is given by the loop product. There are two subtleties related to the definition of this TQFT. First, the TQFT suffers an orientation anomaly: the operation associated to a surface $\Sigma$ shifts the homological degree by $\chi(\Sigma)\dim(M)$. We instead consider $Z$ as a framed 2d TQFT, so that, for instance, the boundary circles are equipped with 2-framings; we denote by $S^1_n$ the circle equipped with a 2-framing with $n$ twists. The second issue is that the operation associated to a disk with an incoming circle is not defined; in other words, the 2d TQFT is a \emph{positive boundary} one (in the terminology of \cite{CohenGodin}) or a \emph{non-compact} one (in the terminology of \cite{LurieCobordism}). The construction of the (closed) 2-dimensional TQFT of Cohen and Godin was extended to an open-closed TQFT (in the sense of Moore and Segal \cite[Chapter 2]{Dbranes}) in the works \cite{SullivanOpenclosed,Harrelson,Ramirez} with branes corresponding to submanifolds of $M$.

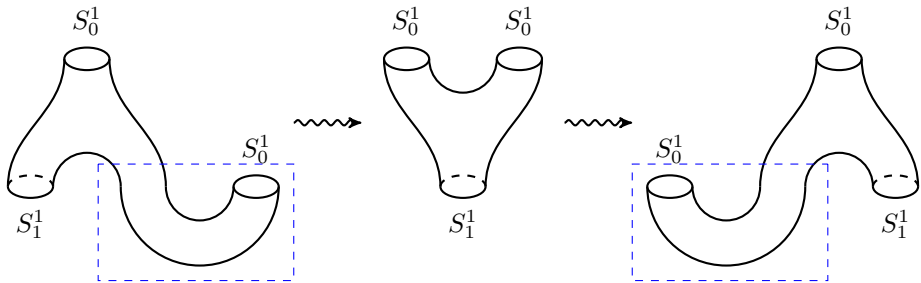
\begin{figure}[ht]
	\begin{tikzpicture}[thick]
		\draw (-5, 0) [partial ellipse = -180:0:0.3 and 0.15];
		\draw (-5, 0) [dashed, partial ellipse = 180:0:0.3 and 0.15];
		\draw (-4.25, 1.7) ellipse (0.3 and 0.15);
		\draw (-2, 0) ellipse (0.3 and 0.15);
		\draw (-5.3, 0) to[out=90, in=-90] (-4.55, 1.7);
		\draw (-4.7, 0) arc (180:0:0.45);
		\draw (-3.8, 0) arc (-180:0:1.05);
		\draw (-2.3, 0) arc (0:-180:0.45);
		\draw (-3.2, 0) to[out=90, in=-90] (-3.95, 1.7);
		\draw[thin, blue, dashed] (-4.1, 0.3) rectangle (-1.5, -1.25);
		\draw (-5, -0.5) node {$S^1_1$};
		\draw (-4.25, 2.2) node {$S^1_0$};
		\draw (-2, 0.5) node {$S^1_0$};
		
		\draw[->, decorate, decoration={snake,amplitude=1pt, segment length=5pt}, -stealth'] (-1.5, 0.85) -- (-0.6, 0.85);
		\draw[->, decorate, decoration={snake,amplitude=1pt, segment length=5pt}, -stealth'] (2.1, 0.85) -- (3, 0.85);
		
		\draw (0.75, 0) [partial ellipse = -180:0:0.3 and 0.15];
		\draw (0.75, 0) [dashed, partial ellipse = 180:0:0.3 and 0.15];
		\draw (0, 1.7) ellipse (0.3 and 0.15);
		\draw (1.5, 1.7) ellipse (0.3 and 0.15);
		\draw (0.45, 0) to[out=90, in=-90] (-0.3, 1.7);
		\draw (1.05, 0) to[out=90, in=-90] (1.8, 1.7);
		\draw (0.3, 1.7) arc (-180:0:0.45);
		\draw (0.75, -0.5) node {$S^1_1$};
		\draw (0, 2.2) node {$S^1_0$};
		\draw (1.5, 2.2) node {$S^1_0$};
		
		\draw (3.5, 0) ellipse (0.3 and 0.15);
		\draw (5.75, 1.7) ellipse (0.3 and 0.15);
		\draw (6.5, 0) [partial ellipse = -180:0:0.3 and 0.15];
		\draw (6.5, 0) [dashed, partial ellipse = 180:0:0.3 and 0.15];
		\draw (6.8, 0) to[out=90, in=-90] (6.05, 1.7);
		\draw (6.2, 0) arc (0:180:0.45);
		\draw (5.3, 0) arc (0:-180:1.05);
		\draw (3.8, 0) arc (-180:0:0.45);
		\draw (4.7, 0) to[out=90, in=-90] (5.45, 1.7);
		\draw[thin, blue, dashed] (5.6, 0.3) rectangle (3, -1.25);
		\draw (3.5, 0.5) node {$S^1_0$};
		\draw (5.75, 2.2) node {$S^1_0$};
		\draw (6.5, -0.5) node {$S^1_1$};
	\end{tikzpicture}
	\caption{A homotopy $\vee'$.}
	\label{fig:TFThomotopyIntro}
\end{figure}

It was observed by Tamanoi \cite{Tamanoi} that while the pair-of-pants product in the Cohen--Godin TQFT is interesting (it is the loop product), the pair-of-pants coproduct is not interesting: it can be expressed purely in terms of the Euler characteristic of $M$. The proof of this fact can be given by using a Frobenius-type relation illustrated in \cref{fig:TFThomotopyIntro} as well as the computation of the value $\bbDelta(1)\in Z(S^1)\otimes Z(S^1)\cong \H_\bullet(LM)\otimes \H_\bullet(LM)$ of $Z$ on the cylinder viewed as a cobordism with an empty incoming boundary. The key observation is that the element $\bbDelta(1)\in\H_0(LM)\otimes \H_0(LM)$ is the image of the homological Euler class $e(M)\in\H_0(M)$ under
\[\H_0(M)\xrightarrow{\Delta}\H_0(M)\otimes \H_0(M)\longrightarrow \H_0(LM)\otimes \H_0(LM).\]
In turn, if $M$ is connected with a basepoint $x$, the homological Euler class is equal to $\chi(M)[x]$.

In \cite{LurieCobordism} it was suggested how to extend the string topology TQFT to a fully extended (non-compact) framed 2d TQFT in the fully homotopical context. Namely, by the cobordism hypothesis such a TQFT $Z$ is determined by a smooth dg category $Z(\pt)$ which one can take to be $\LocSys(M)$. In a sense, it corresponds to an open-closed TQFT with a maximal set of branes; the branes corresponding to submanifolds $i\colon N\hookrightarrow M$ in this picture correspond to the pushforward local systems $i_\sharp \Z_N$. We show in \cref{thm:TFTstringproduct} that the pair-of-pants product in this TQFT coincides with the loop product; so, this 2d TQFT $Z$ is indeed a homotopical and fully extended lift of the Cohen--Godin TQFT.

The homotopy of cobordisms shown in \cref{fig:TFThomotopyIntro}, which realizes the proof of the triviality of the pair-of-pants coproduct $Z(S^1_1)\rightarrow Z(S^1_0)\otimes Z(S^1_0)$, gives rise to a secondary operation on $Z(S^1_1)$. Namely, given a trivialization of $\bbDelta(1)\in\C_0(LM\times LM)$, which is marked in \cref{fig:TFThomotopyIntro} in blue, we obtain a secondary coproduct
\[\vee\colon Z(S^1_1)\longrightarrow Z(S^1_0)\otimes Z(S^1_0)[-1].\]
For a general closed oriented manifold $M$, the lift $\lambda_{\PT}(M)$ provides a trivialization of the image of $\bbDelta(1)$ in $\C_0(LM\times LM, M\times M)$. Thus, using the lift $\lambda_{\PT}(M)$ we obtain the coproduct
\begin{equation}\label{eq:secondaryTQFTcoproduct}
\vee\colon \C_\bullet(LM)[-d]\longrightarrow \C_\bullet(LM, M)\otimes \C_\bullet(LM, M)[-1].
\end{equation}

\begin{maintheorem}[{\Cref{thm:TFTHWequivalence}}]
\label{maintheorem:tqftcoproduct}
The secondary TQFT coproduct \eqref{eq:secondaryTQFTcoproduct} coincides with the loop coproduct.
\end{maintheorem}

This theorem provides a conceptual explanation for the failure of homotopy invariance of the loop coproduct: while the TQFT $Z$ is homotopy invariant, the loop coproduct is a secondary operation which requires a trivialization of $\bbDelta(1)$ and this extra piece of information is not homotopy invariant. More precisely, any two trivializations differ by an element $\tau \in Z(S^1_0) \otimes Z(S^1_0)[-1]$. The corresponding secondary TQFT coproduct then differ by the sum of the two sides of \cref{fig:TFThomotopyIntro}, where $\bbDelta(1)$ is replaced by $\tau$. This is exactly the formula in \cref{maintheorem:stringcoproducterror}.

This theorem also suggests how to generalize the loop coproduct for smooth dg categories other than $\cC=\LocSys(M)$. Namely, any such smooth dg category defines a non-compact framed 2d TQFT $Z_\cC$. In this case $\bbDelta(1)$ is known as the Shklyarov copairing \cite{Shklyarov} (its analog for proper dg categories is the Mukai pairing \cite{Caldararu1}). Given a trivialization of $\bbDelta(1)$ (or its partial trivialization like in the case of manifolds with non-vanishing Euler characteristic) we obtain a secondary coproduct $\vee\colon Z(S^1_1)\rightarrow Z(S^1_0)\otimes Z(S^1_0)[-1]$. In \cref{sect:Shklyarovexamples} we describe several examples of smooth dg categories, where the Shklyarov copairing admits a trivialization (or a partial trivialization); in some of these cases we expect that the secondary coproduct coincides with the already known constructions of such coproducts:
\begin{itemize}
    \item For
    oriented manifolds $M$ together with a non-vanishing vector field $f$ (or equivalently a combinatorial Euler structure in the sense of \cite{TuraevEuler}) one can fully trivialize $\bbDelta(1)$ so that we obtain a loop operation
    \[
    \vee_f \colon \C_\bullet(LM)[-d] \to \C_\bullet(LM) \otimes \C_\bullet(LM)[-1].
    \]
    We expect this to coincide with the constructions in \cite{cieliebak2020poincar}, \cite[Theorem C]{latschev2024bv} and \cite[Section 3.4]{NaefWillwacher}. 
    We note that this coproduct depends on the non-vanishing vector field and is not preserved by general diffeomorphisms (not preserving the vector field). Instead, there is a transformation formula analogous to \cref{maintheorem:stringcoproducterror}. 
    \item If $M$ is a nondegenerate Liouville manifold of dimension $2n$, the wrapped Fukaya category $\cW(M)$ is smooth and there is an isomorphism $\HH_\bullet(\cW(M))\cong \SH^\bullet(M)$ to the symplectic cohomology \cite{Ganatra}. By \cite{Ritter} the Shklyarov copairing $\bbDelta(1)$ vanishes when projected to the positive-energy symplectic cohomology $\SH^\bullet_{>0}(M)$, so we obtain a secondary coproduct
    \[\SH^\bullet(M)\longrightarrow \SH^\bullet_{>0}(M)\otimes \SH^\bullet_{>0}(M)[n-1].\]
    We expect that it coincides with the loop coproduct defined geometrically \cite{AbbondandoloSchwarz,CieliebakHingstonOancea,CieliebakOancea,Kenigsberg}. We refer to \cref{sect:exsymplecticgeometry} for more on this.
    \item Let $A = k\langle x_1, \dots, x_n\rangle$ be the free algebra. It is smooth and by \cref{lm:finitecellresolution} the Shklyarov copairing $\bbDelta(1)$ vanishes when projected to the reduced Hochschild homology $\overline{\HH}_\bullet(A)$ defined as the quotient of $\HH_\bullet(A)$ by the span of the class $[A]$ of the free module, so we obtain a secondary coproduct
    \[\HH^1(A)\longrightarrow \overline{\HH}_0(A) \otimes \overline{\HH}_0(A^{\op}).\]
    We expect that it coincides with the divergence map from \cite[Proposition 3.1]{AKKN}. In the case of the path algebra of a doubled quiver, the composite of the divergence map and the Connes operator coincides with the Lie cobracket on the necklace Lie bialgebra \cite[Equation (2.8)]{Schedler}.
\end{itemize}

Related interpretations of the loop coproduct have appeared in the literature previously:
\begin{itemize}
    \item For a smooth dg category $\cC$ the work of Rivera, Takeda and Wang (see \cite[Section 3.2]{RiveraTakedaWang}) constructs an explicit homotopy $G$ between Hochschild complexes of $\cC$ which we expect to be a chain-level model of the homotopy $\vee'$ from \cref{fig:TFThomotopyIntro}. They conjecture that the secondary coproduct coincides with the loop coproduct for an appropriate choice of the lift of $\bbDelta(1)$ (which is denoted by $E$ in their paper). \Cref{maintheorem:tqftcoproduct} states that an appropriate choice is the Pontryagin--Thom lift $\lambda_{\PT}(M)$, which by \cref{maintheorem:florian} is the same as the Whitehead lift $\tr(\lambda_{\Wh}(M))$. Moreover, they show in \cite[Theorem 6.13]{RiveraTakedaWang} that the corresponding secondary coproduct is coassociative for $\dim M \geq 3$.
    \item For a closed oriented Riemannian manifold $M$ the work of Cieliebak, Hingston and Oancea (see \cite[Section 3.2]{CieliebakHingstonOancea}) constructs a homotopy $\lambda^F$ between certain Floer chain complexes of the unit disk cotangent bundles $D^* M$ (following Abbondandolo and Schwarz \cite{AbbondandoloSchwarz}). The analog of $\bbDelta(1)$ in that case is the secondary continuation quadratic vector $Q_0^F$. Moreover, they show in \cite[Theorem 6.1]{CieliebakHingstonOancea} that the resulting secondary coproduct on the reduced symplectic homology of $D^* M$ coincides with the loop coproduct under the Viterbo--Abbondandolo--Schwarz isomorphism $\SH^\bullet(D^* M)\cong \H_{-\bullet}(LM)$.
\end{itemize}

\subsection*{Acknowledgements}

We would like to thank Manuel Rivera and Nathalie Wahl for useful discussions.

\section{Duality and dimensions}

\subsection{Categorical background}
\label{sect:background}
	
Throughout the paper we use the language of $\infty$-categories. We denote by $\cS$ the $\infty$-category of spaces (homotopy types or $\infty$-groupoids). Given a topological space $X$, we denote the corresponding homotopy type by $\Sing(X)\in\cS$. We denote by $\Sp$ the $\infty$-category of spectra. As a convention, we denote by smash product in $\Sp$ by $\otimes$ (more classically, it is denoted by $\wedge$). Similarly, the direct sum in $\Sp$ is denoted by $\oplus$ (more classically, it is denoted by $\vee$). For an $\infty$-category $\cC$ we denote by $\cC^\omega\subset \cC$ the subcategory of compact objects.

We denote by $\Catperf$ the symmetric monoidal $(\infty, 2)$-category whose objects are idempotent-complete small stable $\infty$-categories and 1-morphisms exact functors. The unit object $\Sp^\omega\in\Catperf$ is the $\infty$-category of finite spectra.

We denote by $\PrSt$ the symmetric monoidal $(\infty, 2)$-category whose objects are stable presentable $\infty$-categories and 1-morphisms colimit-preserving functors. The unit object $\Sp\in\PrSt$ is the $\infty$-category of spectra. For an $\infty$-category $\cC$ we denote by $\Hom_\cC(-, -)\in\cS$ the mapping space. If $\cC\in\PrSt$ we denote by $\bHom_\cC(-, -)\in\Sp$ the mapping spectrum.

        \subsection{Dimensions}
        In this section we recall the formalism of traces in $\infty$-categories explained in \cite{ToenVezzosi,HSS,BZNNonlinear,CCRY}. We refer to \cite{DoldPuppe,PontoShulman} for a 1-categorical treatment. We denote by $\sigma$ the braiding on a symmetric monoidal $\infty$-category.

	\begin{defn}
		Let $\cA$ be a symmetric monoidal category. An object $x\in\cA$ is \defterm{dualizable} if there is an object $x^\vee\in\cA$ and a pair of morphisms $\ev\colon x\otimes x^\vee\rightarrow \bu_\cA$ and $\coev\colon \bu_\cC\rightarrow x^\vee\otimes x$ such that the composites
		\begin{align*}
			x &\xrightarrow{\id\otimes \coev} x\otimes x^\vee\otimes x\xrightarrow{\ev\otimes \id} x \\
			x^\vee &\xrightarrow{\coev\otimes \id} x^\vee\otimes x\otimes x^\vee\xrightarrow{\id\otimes \ev} x^\vee
		\end{align*}
		are the identities. An object of a symmetric monoidal $\infty$-category is dualizable if it is dualizable in the underlying homotopy 1-category.
		\label{def:dualizable}
	\end{defn}

    \begin{remark}
 Explicitly, an object $x\in\cA$ of a symmetric monoidal $\infty$-category is dualizable if it has a dual $x^\vee$, evaluation and coevaluation morphisms
	\[\ev\colon x\otimes x^\vee\rightarrow \bu_\cA,\qquad \coev\colon \bu_\cA\rightarrow x^\vee\otimes x\]
	and cusp 2-isomorphisms
	\[\alpha\colon (\ev\otimes \id_x)\circ (\id_x\otimes \coev)\rightarrow \id_x,\qquad \beta\colon (\id_{x^\vee}\otimes \ev)\circ (\coev\otimes \id_{x^\vee})\rightarrow \id_{x^\vee}.\]
    \end{remark}
    
    \begin{remark}
	We have the following uniqueness statements for duality data (see \cite[Lemma 4.6.1.10]{LurieHA} and \cite{RiehlVerityAdj}):
	\begin{itemize}
		\item For a given $x$, the space $\{x^\vee, \ev\}$ of objects $x^\vee$ and evaluation morphisms $\ev$, for which there exist some compatible triple \{$\coev,\alpha,\beta$\}, is contractible.
		\item For a given $x$, the space $\{x^\vee, \ev,\coev,\alpha\}$ of objects $x^\vee$, evaluation morphisms $\ev$, coevaluation morphisms $\coev$ and the cusp isomorphism $\alpha$, for which there exists some compatible $\beta$, is contractible.
	\end{itemize}
	In other words, we are allowed not to worry about the choices of the duality data as long as we choose the ``right'' amount of information.
    \label{rmk:ddat}
    \end{remark}

    Given dualizable objects $x, y\in\cA$ the object $x\otimes y$ is dualizable with the dual $y^\vee\otimes x^\vee$, the evaluation morphism
    \[x\otimes y\otimes y^\vee\otimes x^\vee\xrightarrow{\id\otimes \ev_y\otimes \id}x\otimes x^\vee\xrightarrow{\ev_x} \bu_\cA\]
    and the coevaluation morphism
    \[\bu_\cA\xrightarrow{\coev_y} y^\vee\otimes y\xrightarrow{\id\otimes \coev_x\otimes \id} y^\vee\otimes x^\vee\otimes x\otimes y.\]

    Given a morphism of dualizable objects, we can pass to its dual.

        \begin{defn}
		Let $\cA$ be a symmetric monoidal category, $x,y\in\cA$ are dualizable objects and $f\colon x\rightarrow y$ is a morphism. The \defterm{dual morphism} $f^\vee\colon y^\vee\rightarrow x^\vee$ is the composite
		\[y^\vee\xrightarrow{\coev_x\otimes\id}x^\vee\otimes x\otimes y^\vee\xrightarrow{\id\otimes f\otimes \id} x^\vee\otimes y\otimes y^\vee\xrightarrow{\id\otimes\ev_y} x^\vee.\]
	\end{defn}

         In particular, we have canonical commutative diagrams (see \cite[Section 3.2]{BZNNonlinear})
	\begin{equation}
		\xymatrix{
			& x\otimes x^\vee \ar^{f\otimes \id}[dr]&  \\
			\bu_\cC \ar^{\coev_x}[ur] \ar_{\coev_y}[dr] && y\otimes x^\vee \\
			& y\otimes y^\vee \ar_{\id\otimes f^\vee}[ur] &
		}
		\qquad
		\xymatrix{
			& y\otimes y^\vee \ar^{\ev_x}[dr] & \\
			x\otimes y^\vee \ar_{\id\otimes f^\vee}[dr] \ar^{f\otimes \id}[ur] && \bu_\cC \\
			& x\otimes x^\vee \ar_{\ev_y}[ur] &
		}
		\label{eq:coevevcompatibility}
	\end{equation}

        \begin{lm}
        Let $\cA$ be a symmetric monoidal $\infty$-category and $x\in\cA$ a dualizable object with evaluation $\ev\colon x\otimes x^\vee\rightarrow \bu_\cA$ and coevaluation $\coev\colon \bu_\cA\rightarrow x^\vee\otimes x$. Then there is a canonical 2-isomorphism
        \begin{equation}\label{eq:evdual}
        (\ev)^\vee\cong \sigma\circ \coev.
        \end{equation}
        \end{lm}
        \begin{proof}
        By definition the dual $(\ev)^\vee$ is given by the composite
        \[\bu_\cA\xrightarrow{(\sigma\circ \coev)} x\otimes x^\vee\xrightarrow{\id_x\otimes \coev\otimes \id_{x^\vee}} x\otimes x^\vee\otimes x\otimes x^\vee\xrightarrow{\id_{x\otimes x^\vee}\otimes \ev} x\otimes x^\vee.\]
        Applying the cusp 2-isomorphism $\beta_x$ we obtain the 2-isomorphism $(\ev)^\vee\rightarrow \coev$.
        \end{proof}

        Dualizable objects have dimensions defined as follows.

 	\begin{defn}
		Let $\cA$ be a symmetric monoidal $\infty$-category and $x\in\cA$ a dualizable object. Its \defterm{dimension} $\dim(x)\in\End_\cA(\bu_\cA)$ is the composite
		\[\dim(x)\colon \bu_\cA\xrightarrow{\coev} x^\vee\otimes x\xrightarrow{\sigma} x\otimes x^\vee\xrightarrow{\ev}\bu_\cA.\]
            More generally, for an endomorphism $f\colon x\rightarrow x$ of a dualizable object we define the \defterm{trace} $\tr(f)\in\End_\cA(\bu_\cA)$ to be the composite
            \[\tr(f)\colon \bu_\cA\xrightarrow{\coev} x^\vee\otimes x\xrightarrow{\id\otimes f} x^\vee\otimes x\xrightarrow{\sigma}x\otimes x^\vee\xrightarrow{\ev}\bu_\cA.\]
		\label{def:dimension}
	\end{defn}

 \begin{remark}
 To simplify the notation, we will often omit the braiding, implicitly identifying coevaluation maps $\bu_\cA\rightarrow x^\vee\otimes x$ with maps $\bu_\cA\rightarrow x\otimes x^\vee$ using the braiding.
 \end{remark}
 
 \Cref{rmk:ddat} makes it clear how to extend $\dim$ to a map $\cA^\dual \to \End_\cA(\bu)$ where $\cA^\dual$ is the groupoid core of dualizable objects in $\cA$ (or, equivalently, the $\infty$-groupoid of tuples $(x,x^\vee, \ev, \coev, \alpha)$ satisfying the above conditions).

   Now let $\cA$ be a symmetric monoidal $(\infty, 2)$-category. In this case the construction of the dimension $\dim$ extends as follows. By definition, an adjunction in $\cA$ is the same as an adjunction in the corresponding homotopy 2-category. In other words, it is specified by the condition that there exist two 1-morphisms $f\colon x\rightarrow y$ and $f^\R\colon y\rightarrow x$ together with two 2-morphisms $\eta\colon \id_x\Rightarrow f^\R f$ and $\epsilon\colon ff^\R\Rightarrow \id_y$ together with 3-isomorphisms witnessing the triangle equalities. We refer to \cite[Appendix F.5]{RiehlVerity} for a proof that in $\PrSt$ this notion of adjunction is equivalent to the one given in \cite[Definition 5.2.2.1]{LurieHTT}.

        \begin{defn}
        Let $\cA$ be a symmetric monoidal $(\infty, 2)$-category, $f\colon x\rightarrow y$ is a morphism of dualizable objects which admits a right adjoint $f^\R\colon y\rightarrow x$. The \defterm{transfer map}
        \[\dim(f)\colon \dim(x)\longrightarrow \dim(y)\]
        is given by traversing the diagram
		\[
		\xymatrix@C=1cm{
			& x\otimes x^\vee \ar_{f\otimes \id}[dr] \ar@{=}[rr] & \ar@{}[d]^(.2){}="a"^(.63){}="b" \ar^{\eta_f}@{=>} "a";"b" & x\otimes x^\vee \ar^{\ev_x}[dr] & \\
			\bu_\cA \ar^{\coev_x}[ur] \ar_{\coev_y}[dr] && y\otimes x^\vee \ar_{f^\R\otimes \id}[ur] \ar^{\id\otimes (f^\R)^\vee}[dr] \ar@{}[d]^(.2){}="c"^(.63){}="d" \ar^{\epsilon_f}@{=>} "c";"d" && \bu_\cA \\
			& y\otimes y^\vee \ar^-{\id\otimes f^\vee}[ur] \ar@{=}[rr] &&  y\otimes y^\vee \ar_{\ev_y}[ur]
		}
		\]
		where the two squares commute using \eqref{eq:coevevcompatibility}.
        \label{def:transfermap}
        \end{defn}

    \begin{remark}
    In formulas, the transfer map is given by the composite
    \begin{align*}
        \dim(x)&=\ev_x \circ \coev_x\\
        &\xrightarrow{\eta_f} \ev_x \circ (f^\R f \otimes \id)\circ \coev_x\\
        &\cong \ev_x \circ (f^\R \otimes f^\vee)\circ \coev_y\\
        &\cong \ev_y \circ (\id\otimes (f^\R)^\vee\circ f^\vee)\circ\coev_y \\
        &\longrightarrow \ev_y \circ \coev_y=\dim(y).
    \end{align*}
    \end{remark}

    We are mainly interested in the example where $\cA = \PrSt$. The following statement gives a description of some of the objects in $\cA^\dual$ (a characterization of all dualizable objects in $\PrSt$ is given in \cite[Proposition D.7.3.1]{LurieSAG}). 

	\begin{prop}[{\cite[Proposition D.7.2.3]{LurieSAG}}]
		Suppose $\cC\in\PrSt$ is compactly generated, i.e. $\cC=\Ind(\cC^\omega)$. Then $\cC$ is dualizable. The dual is $\cC^\vee = \Ind(\cC^{\omega, \op})$ and the evaluation pairing $\ev\colon \cC\otimes \cC^\vee\rightarrow \Sp$ is given by ind-extending the mapping spectrum functor $\bHom_\cC(-, -)\colon \cC^\omega\times \cC^{\omega, \op}\rightarrow \Sp$.
		\label{prop:Stdualizable}
	\end{prop}

        \begin{example}
        Let $I$ be a small $\infty$-category. Then $\Sp^I = \Fun(I, \Sp)$ is dualizable with dual given by $\Sp^{I^{\op}}$. The evaluation pairing $\ev\colon \Sp^I\otimes \Sp^{I^{\op}}\rightarrow \Sp$ is
        \[\ev(F, G) = \int^{i\in I} F(i)\otimes G(i)\]
        and the coevaluation pairing $\coev\colon \Sp\rightarrow \Sp^{I^{\op}}\otimes \Sp^I\cong \Sp^{I^{\op}\times I}$ sends $\bS$ to the functor $i, j\mapsto \Sigma^\infty_+ \Hom_I(i, j)$. In particular,
        \[\dim(\Sp^I)\cong \int^{i\in I} \Sigma^\infty_+ \Hom_I(i, i)\]
        and for a functor $f\colon I\rightarrow J$ the transfer
        \[\dim(f_\sharp)\colon \dim(\Sp^I)\longrightarrow \dim(\Sp^J)\]
        is given by the composite
        \[\int^{i\in I} \Sigma^\infty_+\Hom_I(i, i)\longrightarrow \int^{i\in I}\Sigma^\infty_+\Hom_J(f(i), f(i))\longrightarrow \int^{j\in J}\Sigma^\infty_+\Hom_J(j, j).\]
        \label{ex:presheafdual}
        \end{example}
	
	\begin{example}
		Suppose $\cC,\cD\in\PrSt$ are compactly generated and $F\colon \cC\rightarrow \cD$ is a functor preserving compact objects. Then it admits a colimit-preserving right adjoint; in other words, it is a right-adjointable 1-morphism in $\PrSt$. By \cite[Proposition 4.24]{HSS} we have $\dim(\cC)\cong \THH(\cC^\omega)$ and the transfer map fits into the commutative diagram
		\[
		\xymatrix@C=2cm{
			\dim(\cC) \ar^{\dim(F)}[r] \ar^{\sim}[d] & \dim(\cD) \ar^{\sim}[d] \\
			\THH(\cC^\omega) \ar^{\THH(F)}[r] & \THH(\cD^\omega).
		}
		\]
	\end{example}
	
	\begin{example}
		Suppose $\cC\in\PrSt$ is compactly generated and $x\in\cC$ is a compact object. Consider the functor $F_x\colon \Sp\rightarrow \cC$ given by $V\mapsto V\otimes x$ which admits a colimit-preserving right adjoint. By the previous example the class $[x]\in \Omega^\infty\THH(\cC^\omega)$ is computed as the image of the transfer map
		\[\bS\cong \dim(\Sp)\xrightarrow{\dim(F_x)} \dim(\cC).\]
        This is the \defterm{Chern character} of the object $x$ \cite{McCarthy,Keller,BZNNonlinear}.
	\end{example}

   The construction of transfer maps can be formulated coherently as follows. The authors of \cite{HSS} define an $(\infty,1)$-category $\cA^\dual$ with the following informal description:
    \begin{enumerate}
        \item Objects of $\cA^\dual$ are dualizable objects in $\cA$.
        \item For $x, y \in \cA^\dual$ morphisms $x\rightarrow y$ in $\cA^\dual$ are the same as morphisms $x\rightarrow y$ in $\cA$ admitting a right adjoint.
    \end{enumerate}

    The $\infty$-category $\cA^\dual$ allows one to formulate the functoriality of dimensions as follows (see \cite[Definition 2.9]{HSS}).

    \begin{thm}
        There exists a symmetric monoidal functor
        \[
        \dim(-) \colon \cA^\dual \to \End_\cA(\bu),
        \]
        given on objects by $x\mapsto \dim(x)$ and on morphisms by $(f\colon x\rightarrow y)\mapsto (\dim(f)\colon \dim(x)\rightarrow \dim(y))$.
        \label{thm:ExistenceOfDim}
    \end{thm}

    Using the functoriality of the dimensions given by \cref{thm:ExistenceOfDim}, we see that the Chern character $\dim(F_x)\colon \bS\rightarrow \dim(\cC)$ when $x\in\cC^\omega$ ranges over compact objects of $\cC$ defines a map
    \begin{equation}
    (\cC^\omega)^\sim\longrightarrow \Omega^\infty\dim(\cC)
    \label{eq:cherncharacters}
    \end{equation}
    from the groupoid core of compact objects in $\cC$.

    \subsection{Relative dualizability}
	
	We will now specialize the setting of the previous section to the case $\cA=\PrSt$. We begin with the following observation.
 
    \begin{lm}
    \label{lm:reldualiscompact}
    Let $\cC\in\PrSt$ be a stable presentable $\infty$-category and $x\in\cC$ an object. Then $x$ is compact if, and only if, the functor $\Sp \to \cC$ given by $V \mapsto V \otimes x$ admits a right adjoint in $\PrSt$.
	\end{lm}
	\begin{proof}
        By definition, the functor $V\mapsto V\otimes x$ admits a right adjoint given by the mapping spectrum functor $\bHom_\cC(x, -)\colon \cC\rightarrow \Sp$. It preserves finite colimits since $\cC$ is stable; therefore, it defines a right adjoint in $\PrSt$ if it preserves filtered colimits.

        Suppose $\bHom_\cC(x, -)\colon \cC\rightarrow \Sp$ preserves filtered colimits. Then the composite \[\Hom_\cC(x, -)\colon \cC\xrightarrow{\bHom_\cC(x, -)} \Sp\xrightarrow{\Omega^\infty} \cS\]
        preserves filtered colimits. Therefore, $x\in\cC$ is compact.

        Conversely, suppose $x\in\cC$ is compact. Then $x[n]\in\cC$ is compact for every $n\in\Z$. Therefore, $\Omega^{\infty+n}\bHom_\cC(x, -)\colon \cC\rightarrow \cS$ preserves filtered colimits. As the functors $\{\Omega^{\infty+n}\colon \Sp\rightarrow \cS\}_{n\in\Z}$ are jointly conservative and preserve filtered colimits, this implies that $\bHom_\cC(x, -)\colon \Sp\rightarrow \cS$ preserves filtered colimits.
	\end{proof}
 The previous statement rephrases the property of an object of a stable presentable $\infty$-category being compact internally to the 2-category $\PrSt$. We are now going to use this characterization to introduce a duality for compact objects in dualizable (for example, compactly generated) stable presentable $\infty$-categories.
    
	\begin{defn}
 Let $\cC\in\PrSt$ be a dualizable stable presentable $\infty$-category with the dual $\cC^\vee\in\PrSt$, evaluation functor $\ev\colon \cC\otimes \cC^\vee\rightarrow \Sp$ and the coevaluation functor $\coev\colon \Sp\rightarrow \cC^\vee\otimes \cC$. We say an object $x\in\cC$ is \defterm{relatively (right) dualizable} if there is an object $x^\vee\in\cC^\vee$ together with morphisms $\epsilon\colon x^\vee\boxtimes x\rightarrow \coev(\bS)$ and $\eta\colon \bS\rightarrow \ev(x, x^\vee)$ satisfying the obvious analogs of the duality axioms in \cref{def:dualizable}.
	\end{defn}
	
	\begin{prop}
 Let $\cC\in\PrSt$ be a dualizable stable presentable $\infty$-category and $x\in\cC$. The following are equivalent:
		\begin{enumerate}
            \item $x\in\cC$ has a relative dual $x^\vee\in\cC^\vee$ with evaluation pairing $\epsilon\colon x^\vee\boxtimes x\rightarrow \coev(\bS)$.
            \item The functor $\Sp\rightarrow \cC$ given by $\bS\mapsto x$ admits a colimit-preserving right adjoint $H_x\colon \cC\rightarrow \Sp$ such that the induced map $x^\vee\rightarrow (\id\otimes H_x)\coev(\bS)$ is an isomorphism, i.e. the functor $\Sp\rightarrow \cC^\vee$ given by $\bS\mapsto x^\vee$ is dual to $H_x$.
            \item $x$ is compact.
		\end{enumerate}
		\label{prop:CWduality}
	\end{prop}
	\begin{proof}
		Assume $x\in\cC$ has a relative dual $x^\vee\in\cC^\vee$. The relative dualizability data is provided by maps
		\begin{align*}
			\epsilon&\colon x^\vee\boxtimes x\longrightarrow \coev(\bS) \\
			\eta&\colon \bS\longrightarrow \ev(x, x^\vee)
		\end{align*}
		where $\epsilon$ is a morphism in $\cC^\vee\otimes \cC$ and $\eta$ is a morphism in $\Sp$. Identifying $\cC^\vee\otimes \cC\cong \Fun(\cC, \cC)$ via $y\boxtimes x\mapsto \ev(-, y)\otimes x$ the data of $\epsilon$ is equivalent to the data of the morphism
		\[\overline{\epsilon}\colon \ev(-, x^\vee)\otimes x\longrightarrow \id\]
		in $\Fun(\cC, \cC)$.
		
		We see that the relative dualizability data $(\epsilon, \eta)$ for $(x, x^\vee)$ is equivalent to the data witnessing the functors $(\bS\mapsto x, z\mapsto \ev(z, x^\vee))$ as being adjoint. In turn, the functor $\cC\rightarrow \Sp$ given by $z\mapsto \ev(z, x^\vee)$ is dual to the functor $\Sp\rightarrow \cC^\vee$ given by $\bS\mapsto x^\vee$ which proves the equivalence of the first two statements.

  The equivalence of the last two statements is the content of \cref{lm:reldualiscompact}.
	\end{proof}
 
 This perspective on compact objects allows us to give a formula for the Chern character.

	\begin{prop}
		Suppose $\cC\in\PrSt$ is dualizable and $x\in\cC$ is compact. Then its Chern character $\dim(F_x)\colon \bS\rightarrow \dim(\cC) = \ev(\coev(\bS))$ is given by the composite
		\[\bS\xrightarrow{\eta}\ev(x, x^\vee)\xrightarrow{\epsilon} \ev(\coev(\bS)).\]
		\label{cor:ChernCW}
	\end{prop}
    \begin{proof}
        The statement is obtained by unpacking \cref{def:transfermap}.
    \end{proof}

    \subsection{\texorpdfstring{$K$}{K}-theory}

    Given a small stable $\infty$-category $\cC$ we have the connective $K$-theory spectrum $\bK(\cC)$ defined as in \cite[Section 7]{BlumbergGepnerTabuadaK}. We denote by $K(\cC)=\Omega^\infty \bK(\cC)$ the underlying space. By construction, there is a natural map $\cC^\sim\rightarrow K(\cC)$ from the groupoid core of objects in $\cC$ to $K(\cC)$; so, an object $x\in\cC$ defines a point $[x]\in K(\cC)$.

    There is a \defterm{Dennis trace map}
    \[\tr\colon \bK(\cC)\longrightarrow \THH(\cC)\]
    defined, for instance, by the universal property in \cite[Theorem 10.6]{BlumbergGepnerTabuadaK}. As explained in \cite[Remark 6.12]{HSS}, it is compatible with the dimension map defined previously as follows.

    \begin{prop}
    Let $\cC$ be a stable compactly generated $\infty$-category. Then there is a commutative diagram
    \[
    \xymatrix{
    (\cC^\omega)^\sim \ar^-{\eqref{eq:cherncharacters}}[r] & \Omega^\infty\dim(\cC) \\
    K(\cC^\omega) \ar[u] \ar^-{\tr}[r] & \THH(\cC^\omega) \ar^{\sim}[u]
    }
    \]
    \end{prop}

    Recall from \cite[Definition 6.1]{BlumbergGepnerTabuadaK} that an additive invariant is a functor $F\colon \Catperf\rightarrow \Sp$ which preserves filtered colimits, sends Morita equivalences to isomorphisms and sends split-exact sequences to exact sequences of spectra. We will use that algebraic $K$-theory $\bK\colon \Catperf\rightarrow \Sp$ as well as topological Hochschild homology $\THH\colon \Catperf\rightarrow \Sp$ are additive functors, see \cite[Propositions 7.10 and 10.2]{BlumbergGepnerTabuadaK}.

        \subsection{Smooth objects}
        Throughout this section $\cA$ denotes a symmetric monoidal $(\infty, 2)$-category.

        \begin{defn}\label{def:smoothproper}
		Let $x\in\cA$ be a dualizable object.
  \begin{itemize}
      \item $x$ is \defterm{smooth} if $\ev$ admits a left adjoint $\ev^\L$.
      \item $x$ is \defterm{proper} if $\ev$ admits a right adjoint $\ev^\R$.
  \end{itemize}
	\end{defn}
	
	\begin{remark}
		For any dualizable object the evaluation map $\ev$ is dual to the coevaluation map $\coev$, so $\ev$ admits a left adjoint if, and only if, $\coev$ admits a right adjoint $\coev^\R$ and vice versa. In particular, for $\cC \in \PrSt$ we obtain that $\cC$ is smooth if, and only if, $\coev(1) \in \cC \otimes \cC^\vee$ is compact.
	\end{remark}

        \begin{example}
        Suppose $\cA=\PrSt$ and $\cC\in\PrSt$ is compactly generated. Using the duality data from \cref{prop:Stdualizable} we see that $\cC$ is proper if, and only if, for every compact objects $x,y\in\cC$ the mapping spectrum $\bHom_\cC(x, y)$ is finite.
        \end{example}

        If $\cC\in\PrSt$ is smooth, by adjunction we obtain an equivalence
        \[\dim(\cC)\cong \bHom_{\cC\otimes \cC^\vee}(\ev^\L(\bS), \coev(\bS)).\]

        In the smooth case the data of a relative dualizability of an object $x\in\cC$ can be phrased in terms of the evaluation $\epsilon\colon x\boxtimes x^\vee\rightarrow \coev(\bS)$ and coevaluation $\eta\colon \ev^\L(\bS)\rightarrow x\boxtimes x^\vee$ satisfying the duality axioms. In this case the Chern character of a relatively dualizable object $x\in\cC$ in a smooth $\infty$-category can be written as the composite
	\begin{equation}
	\ev^\L(\bS)\xrightarrow{\eta} x\boxtimes x^\vee\xrightarrow{\epsilon} \coev(\bS)
        \label{eqn:cherncharevL}
	\end{equation}

        If $\cD$ is a monoidal $\infty$-category and $\cM$ a $\cD$-module $\infty$-category equipped with an object $M\in\cM$, the \defterm{center} of $M$ is an object $\fZ(M)\in\cD$ satisfying the universal property
        \[\Hom_{\cD}(D, \fZ(M))\cong \Hom_{\cM}(D\otimes M, M)\]
        for every $D\in\cD$, see \cite[Definition 5.3.1.6]{LurieHA} for more details. We will use the center in the following way.

        \begin{defn}
        Let $x\in\cA$ be an object. The \defterm{center} of $x$ is an object $\fZ(x)\in\End_\cA(\bu)$ which is the center of $\id_x\in\End_\cA(x)$, where we view $\End_\cA(x)$ as an $\End_\cA(\bu)$-module $\infty$-category.
        \end{defn}

        \begin{remark}
        As $\id_x$ is an $\bE_1$-algebra in $\End_\cA(x)$, one can show that if the center $\fZ(x)\in\End_\cA(\bu)$ exists, then $\fZ(x)$ upgrades to an $\bE_2$-algebra in $\End_\cA(\bu)$.
        \end{remark}

        \begin{remark}
        Considering $\cA$ as the bicategory of small categories, we see that for a category $\cC\in\cA$ its center $\fZ(\cC)$ is identified with the set of natural endomorphisms of the identity functor.
        \end{remark}

        \begin{prop}
        Let $x\in\cA$ be a smooth object. Then $\coev^\R\circ \coev$ is the center $\fZ(x)$.
        \end{prop}
        \begin{proof}
        Since $x$ is dualizable, we may identify $\End_\cA(x)\cong \Hom_\cA(\bu, x^\vee\otimes x)$ and under this identification $\id_x\mapsto \coev$. Moreover, since $x$ is smooth, $\coev$ has a right adjoint $\coev^\R$. For $D\in\End_\cA(\bu)$ we have natural isomorphisms
        \[\Hom_{\End_\cA(x)}(D\otimes \id, \id)\cong \Hom_{\Hom_\cA(\bu, x^\vee\otimes x)}(\coev\circ D, \coev)\cong \Hom_{\End_\cA(\bu)}(D, \coev^\R\circ \coev).\]
        \end{proof}

\section{Parametrized spectra}

In this section we describe $\infty$-categorical functoriality of the category of parametrized spectra introduced in \cite{MaySigurdsson} following \cite{AndoBlumbergGepner,HaugsengBG,HopkinsLurie}.

\subsection{Functor categories}

Let $I$ be a small $\infty$-category and $\cC$ a stable presentable $\infty$-category. Consider
\[\Fun(I, \cC)\in\PrSt.\]
As it has small colimits and limits, it is tensored and cotensored over spaces. For a functor $f\colon I\rightarrow J$ of small $\infty$-categories we have the induced functors
\[
\xymatrix@C=2cm{
\Fun(I, \cC) \ar@/^2.0pc/^{f_\sharp}[r] \ar@/_2.0pc/^{f_*}[r] & \Fun(J, \cC) \ar_{f^*}[l]
}
\]
where $(f^* G)(i) = G(f(i))$ is the restriction functor and its left and right adjoints are the left and right Kan extensions given by the (co)ends
\begin{equation}
(f_\sharp F)(j) = \int^{i\in I} \Hom_J(f(i), j)\otimes F(i),\qquad (f_* F)(j) = \int_{i\in I} \Hom(\Hom_J(j, f(i)), F(i)),
\label{eq:Kanextensions}
\end{equation}
see \cite[Lemma A.11]{Ariotta}. We have $f_\sharp\dashv f^* \dashv f_*$. In particular, $f_\sharp$ and $f^*$ are colimit-preserving functors, but note that $f_*$ is not, in general, colimit-preserving. Let us now state some basic properties of functor categories that we will use.

Recall the notion of a smooth functor $f\colon I\rightarrow J$ from \cite[Definition 4.4.15]{Cisinski} (see also \cite[Defnition 4.1.2.9]{LurieHTT}). By \cite[Lemma 4.29]{Barkan} an equivalent way to define it is as follows.

\begin{defn}
A functor $f\colon I\rightarrow J$ is \defterm{proper} if for every $j\in J$ the functor $\{j\}\times_J I\rightarrow I_{/j}$ is cofinal. A functor $f$ is \defterm{smooth} if $f^{\op}\colon I^{\op}\rightarrow J^{\op}$ is proper.
\end{defn}

We will encounter smooth functors coming from the following two examples:
\begin{itemize}
    \item If $f\colon X\rightarrow Y$ is a morphism of $\infty$-groupoids, then it is smooth by \cite[Proposition 4.4.11]{Cisinski}.
    \item The projection $p\colon I\rightarrow \pt$ is smooth by \cite[Proposition 4.4.12]{Cisinski} for any $I$.
\end{itemize}

\begin{prop}[Base change formula]
Let $\cC\in\PrSt$ be a stable presentable $\infty$-category. Suppose that
\[
\xymatrix{
\tilde{I} \ar^{\tilde{g}}[r] \ar^{\tilde{f}}[d] & I \ar^{f}[d] \\
\tilde{J} \ar^{g}[r] & J
}
\]
is a Cartesian diagram of small $\infty$-categories, where $g$ is smooth. Then the natural transformation
\[\tilde{f}_\sharp\tilde{g}^*\Rightarrow g^* f_\sharp\]
of functors $\Fun(I, \cC)\rightarrow \Fun(\tilde{J}, \cC)$ is an equivalence.
\label{prop:basechange}
\end{prop}
\begin{proof}
The dual claim (replacing $\cC$ by $\cC^{\op}$ and left Kan extensions by right Kan extensions) is shown in \cite[Theorem 6.4.13]{Cisinski}.
\end{proof}

\begin{prop}
Suppose $\cC$ is compactly generated. Then the category $\Fun(I, \cC)$ is compactly generated. If $I$ is finite and $\Hom_I(x, y)$ is a finite space for every $x,y\in I$, then $\Fun(I, \cC^\omega) = \Fun(I, \cC)^\omega$ considered as subcategories of $\Fun(I, \cC)$.
\label{prop:compactcharacterization}
\end{prop}
\begin{proof}
For the first claim see \cite[Corollary 2.3]{AokiTensor} and for the second claim see \cite[Proposition 2.8]{AokiTensor}.
\end{proof}

\begin{prop}
For two small $\infty$-categories $I, J$ and $\cC,\cD\in\PrSt$ there is a natural equivalence
\[\boxtimes\colon \Fun(I, \cC)\otimes \Fun(J, \cD)\longrightarrow \Fun(I\times J, \cC\otimes\cD).\]
\label{prop:externaltensor}
\end{prop}
\begin{proof}
By \cite[Corollary 2.2]{AokiTensor} we have $\Fun(I, \cC)\cong \Fun(I, \cS)\otimes\cC$, so it is enough to prove the claim for $\cC=\cD=\cS$. Recall that by \cite[Proposition 4.8.1.17]{LurieHA} there is a natural identification between $\cC\otimes\cD$ and the $\infty$-category $\Fun^{\R}(\cC^{\op}, \cD)$ of limit-preserving functors $\cC^{\op}\rightarrow \cD$. Therefore,
\[\Fun(I, \cS)\otimes \Fun(J, \cS)\cong \Fun^{\R}(\Fun(I, \cS)^{\op}, \Fun(J, \cS)).\]
By \cite[Theorem 5.1.5.6]{LurieHTT} we have
\[\Fun^{\R}(\Fun(I, \cS)^{\op}, \Fun(J, \cS))\cong \Fun(I, \Fun(J, \cS))\]
which is naturally equivalent to $\Fun(I\times J, \cS)$.
\end{proof}

\subsection{Duality for stable presheaves}

Consider a small $\infty$-category $I$ as before and the $\infty$-category $\Sp^I = \Fun(I, \Sp)$ of functors. Recall from \cref{ex:presheafdual} that $\Sp^I$ is dualizable with dual $\Sp^{I^{\op}}$.

\begin{prop}
Let $f\colon I\rightarrow J$ be a functor of small $\infty$-categories and $f^{\op}\colon I^{\op}\rightarrow J^{\op}$ its opposite. Consider the duality $(\Sp^I)^\vee\cong\Sp^{I^{\op}}$ and $(\Sp^J)^\vee\cong \Sp^{J^{\op}}$ from \cref{ex:presheafdual}. Then $f_\sharp\colon \Sp^I\rightarrow \Sp^J$ is dual to $(f^{\op})^*\colon \Sp^{J^{\op}}\rightarrow \Sp^{I^{\op}}$. Moreover, the following conditions are equivalent:
\begin{itemize}
    \item The functor $f^*\colon \Sp^J\rightarrow \Sp^I$ admits a colimit-preserving right adjoint $f_*\colon \Sp^I\rightarrow \Sp^J$.
    \item The functor $(f^{\op})_\sharp\colon \Sp^{I^{\op}}\rightarrow \Sp^{J^{\op}}$ admits a left adjoint $(f^{\op})^\sharp\colon \Sp^{J^{\op}}\rightarrow \Sp^{I^{\op}}$.
\end{itemize}
\label{prop:presheafduality}
\end{prop}
\begin{proof}
For the first statement we have to show that the diagram
\[
\xymatrix@C=2cm{
\Sp \ar^{\coev}[r] \ar^{\coev}[d] & \Sp^{I^{\op}}\otimes \Sp^I \ar^{\id\otimes f_\sharp}[d] \\
\Sp^{J^{\op}}\otimes \Sp^J \ar^{(f^{\op})^*\otimes \id}[r] & \Sp^{I^{\op}}\otimes \Sp^J
}
\]
naturally commutes. This follows from the natural equivalence
\[\Sigma^\infty_+\Hom_J(f(i), j)\cong \int^{i\in I} \Sigma^\infty_+\Hom(k, i)\otimes \Sigma^\infty_+\Hom(f(k), j).\]

Given the adjunction $f^*\dashv f_*$, passing to the duals we get an adjunction $(f_*)^\vee\dashv (f^*)^\vee\cong (f^{\op})_\sharp$ which shows the second statement.
\end{proof}

Consider the projection $p\colon I\rightarrow \ast$ and the diagonal map $\Delta\colon I\rightarrow I\times I$. Denote by $\bS_I=p^*\bS\in\Sp^I$ the constant object $i\in I\mapsto\bS$. Consider the following additional assumption on the $\infty$-category $I$.

\begin{defn}
A small $\infty$-category $I$ \defterm{admits duality} if the functor $\Delta_\sharp p^*\colon \Sp\rightarrow \Sp^I\otimes \Sp^I$ is the coevaluation of a self-duality of $\Sp^I$.
\label{def:admitduality}
\end{defn}

\begin{remark}
The evaluation is dual to coevaluation, so the functor $\Delta_\sharp p^*\colon\Sp\rightarrow \Sp^I\otimes \Sp^I$ is the coevaluation of a self-duality of $\Sp^I$ if, and only if, the functor $(p^{\op})_\sharp (\Delta^{\op})^*\colon \Sp^{I^{\op}}\otimes \Sp^{I^{\op}}\rightarrow \Sp$ is the evaluation of a self-duality of $\Sp^{I^{\op}}$. This is similar, but not identical, to the notion of a \emph{Verdier poset} from \cite[Definition 1.8]{AokiPosets}, where a finite poset $P$ is called Verdier if $p_*\Delta^*$ is the evaluation of a self-duality of $\Sp^P$ (note the difference in the choice of the pushforward functors).
\end{remark}

If $I$ admits duality, then we may compare the self-duality of $\Sp^I$ to the duality between $\Sp^I$ and $\Sp^{I^{\op}}$ from \cref{ex:presheafdual} using the duality functor
\[\bD\colon \Sp^{I^{\op}}\longrightarrow \Sp^I\]
given by
\[\bD(\cF)(i) = \colim_j \Hom(j, i)\otimes \cF(j).\]

\begin{remark}
For an $\infty$-category $I$ admitting duality the functor $\bD$ is an equivalence. This functor is a generalization of the Costenoble--Waner duality functor, see \cref{rmk:CWduality}.
\end{remark}

\begin{prop}
Let $f\colon I\rightarrow J$ be a functor of small $\infty$-categories which admit duality and consider the corresponding self-duality of $\Sp^I$ and $\Sp^J$. Then there is a natural transformation
\begin{equation}\label{eq:pushforwarddual}
(f_\sharp)^\vee\longrightarrow f^*.
\end{equation}
If $f$ is smooth, \eqref{eq:pushforwarddual} is an equivalence and the following conditions are equivalent:
\begin{itemize}
    \item The functor $f^*\colon \Sp^J\rightarrow \Sp^I$ admits a colimit-preserving right adjoint $f_*\colon \Sp^I\rightarrow \Sp^J$.
    \item The functor $f_\sharp\colon \Sp^I\rightarrow \Sp^J$ admits a left adjoint $f^\sharp\colon \Sp^J\rightarrow \Sp^I$.
\end{itemize}
\label{prop:presheafselfduality}
\end{prop}
\begin{proof}
Consider the Cartesian diagram of $\infty$-categories
\[
\xymatrix@C=2cm{
I \ar^{(f\times \id)\circ \Delta^I}[r] \ar^{f}[d] & J\times I \ar^{\id\times f}[d] \\
J \ar^{\Delta^J}[r] & J\times J
}
\]
and the corresponding base change natural transformation
\[(f_\sharp\boxtimes \id)\Delta^I_\sharp f^*\longrightarrow (\id\boxtimes f^*)\Delta^J_\sharp.\]
Applying it to $\bS_J$ we get a morphism
\[(f_\sharp\boxtimes \id)\Delta^I_\sharp \bS_I\longrightarrow (\id\boxtimes f^*)\Delta^J_\sharp \bS_J\]
which produces a natural transformation $(f_\sharp)^\vee\rightarrow f^*$.

If $f\colon I\rightarrow J$ is smooth, then $(f\times \id)\colon J\times I\rightarrow J\times J$ is smooth, since smooth functors are stable under base change \cite[Proposition 4.4.2]{Cisinski}. In particular, by the base change formula from \cref{prop:basechange} the morphism $(f_\sharp\boxtimes \id)\Delta^I_\sharp \bS_I\rightarrow (\id\boxtimes f^*)\Delta^J_\sharp \bS_J$ is an isomorphism.

The second statement is proven analogously to the second statement in \cref{prop:presheafduality}.
\end{proof}

If $I$ is a small $\infty$-category which admits duality, the product $I\times I$ admits duality. Then \eqref{eq:pushforwarddual} gives a natural transformation
\begin{equation}\label{eq:Deltasharpdual}
\Delta_\sharp^\vee\longrightarrow \Delta^*.
\end{equation}

The functor $\Delta^*$ defines a symmetric monoidal structure on $\Sp^I$ with the unit $\bS_I$. Its left adjoint, $\Delta_\sharp$, defines a symmetric comonoidal structure with the counit $p_\sharp$. Passing to the dual we obtain a \emph{new} symmetric monoidal structure on $\Sp^I$ given by $\Delta_\sharp^\vee$ with the same unit: $(p_\sharp)^\vee(\bS)\cong p^*\bS \cong \bS_I$. In particular, there is the unitor natural isomorphism
\begin{equation}\label{eq:Deltasharpunitor}
\Delta_\sharp^\vee(\bS_I \boxtimes \cF)\cong \cF.
\end{equation}

Then \eqref{eq:Deltasharpdual} defines a lax symmetric monoidal structure on the identity functor with respect to the two symmetric monoidal structures $\Delta_\sharp^\vee$ and $\Delta^*$.

If we assume that $I$ admits duality and $\bS_I$ is compact, we have further structure:
\begin{itemize}
    \item $\bS_I$ admits a relative dual $\zeta_I=p^\sharp\bS\in\Sp^I$. From \cref{prop:CWduality} we get the formula
    \[\zeta_I = (p_*\boxtimes \id)\Delta_\sharp \bS_I.\]
    The adjunction $p^\sharp\dashv p_\sharp$ is given by a natural equivalence
    \begin{equation}\label{eq:zetapushforward}
    \bHom_{\Sp^I}(\zeta_I, -)\cong p_\sharp(-).
    \end{equation}
    \item We obtain a formula for the evaluation as
    \begin{equation}\label{eq:evzeta}
    \ev_{\Sp^I} \xrightarrow[\sim]{\eqref{eq:evdual}} \coev_{\Sp^I}^\vee \cong (p^*)^\vee \circ \Delta_\sharp^\vee \xrightarrow[\sim]{\eqref{eq:pushforwarddual}} p_\sharp\circ \Delta_\sharp^\vee\xrightarrow[\sim]{\eqref{eq:zetapushforward}} \bHom_{\Sp^I}(\zeta_I, \Delta_\sharp^\vee(- \boxtimes -)).
    \end{equation}
\end{itemize}

\begin{lm}\label{lm:counitforunit}
    Under the equivalence
    \[
    \ev_{\Sp^I}(\bS_I \boxtimes \zeta_I) \xrightarrow[\sim]{\eqref{eq:evzeta}} \bHom_{\Sp^I}(\zeta_I, \Delta_\sharp^\vee(\bS_I \boxtimes \zeta_I)) \xrightarrow[\sim]{\eqref{eq:Deltasharpunitor}} \bHom_{\Sp^I}(\zeta_I, \zeta_I),
    \]
    the coevaluation $\bS \to \ev_{\Sp^I}(\bS_I \boxtimes \zeta_I)$ of the relative duality between $\bS_I$ and $\zeta_I$ is mapped to the identity morphism $\zeta_I\rightarrow \zeta_I$.
\end{lm}
\begin{proof}
First note that the relative duality between $\bS_I$ and $\zeta_I = p^\sharp \bS$ is established by the adjunction
\[
p^\sharp \dashv (p^*)^\vee.
\]
By \cref{lm:reldualiscompact} it follows that the unit of the adjunction $\bS \to (p^*)^\vee(\zeta_I) = \ev_{\Sp^I}(\bS \boxtimes \zeta_I)$ is the coevaluation of the relative duality. But the unit of the adjunction is determined by the requirement that under the isomorphism
\[
\ev_{\Sp^I}(\bS_I \boxtimes \zeta_I) = (p^*)^\vee(\zeta_I) \overset{\eqref{eq:pushforwarddual}}{\cong} p_\sharp(\zeta_I) \overset{\eqref{eq:zetapushforward}}{\cong} \bHom_{\Sp^I}(\zeta_I, \zeta_I)
\]
it is sent to the identity. The statement thus follows once we show that the natural transformation
\[
    (p^*)^\vee(-) = \ev_{\Sp^I}(\bS_I \boxtimes -) \xrightarrow[\sim]{\eqref{eq:evzeta}} \bHom_{\Sp^I}(\zeta_I, \Delta_\sharp^\vee(\bS_I \boxtimes -)) \xrightarrow[\sim]{\eqref{eq:Deltasharpunitor}} \bHom_{\Sp^I}(\zeta_I, -),
\]
is equivalent to
\[
(p^*)^\vee(-) \overset{\eqref{eq:pushforwarddual}, \eqref{eq:zetapushforward}}{\cong} \bHom_{\Sp^I}(\zeta_I, -).
\]
Unpacking the definition of \eqref{eq:evzeta} we obtain
\[
    (p^*)^\vee(-) = \ev_{\Sp^I}(\bS_I \boxtimes -) \xrightarrow[\sim]{\eqref{eq:evdual}} (p^*)^\vee \circ \Delta_\sharp^\vee (\bS_I \boxtimes -) \overset{\eqref{eq:Deltasharpunitor}}{\cong} (p^*)^\vee(-) \xrightarrow[\sim]{\eqref{eq:pushforwarddual}, \eqref{eq:zetapushforward}} \bHom_{\Sp^I}(\zeta_I, -).
\]
We claim that the automorphism of $(p^*)^\vee$ appearing is homotopic to the identity. Spelling out the definition of \eqref{eq:Deltasharpunitor} we arrive at the following presentation of that automorphism
\[
    (p^*)^\vee = \ev_{\Sp^I} \circ (p^* \otimes \id) \xrightarrow[\sim]{\eqref{eq:evdual}} (p^*)^\vee \circ \Delta_\sharp^\vee \circ (p^* \otimes \id) \xleftarrow[\sim]{\eqref{eq:pushforwarddual}} (p^*)^\vee \circ \Delta_\sharp^\vee \circ (p_\sharp^\vee \otimes \id) \cong (p^*)^\vee (( p_\sharp \otimes \id) \circ \Delta_\sharp)^\vee \cong (p^*)^\vee.
\]
By passing to duals, it is given by
\[
p^* = (p^*)^{\vee \vee} =  ((p^*)^\vee \otimes \id)(\Delta_\sharp p^*) \xleftarrow[\sim]{\eqref{eq:pushforwarddual}} (p_\sharp \otimes \id)(\Delta_\sharp p^*) = p^*.
\]
which is equivalent to
\[
p^* \xleftarrow[\sim]{\eqref{eq:pushforwarddual}} p_\sharp^\vee =  (p_\sharp \otimes \id)(\Delta_\sharp p^*) = p^*.
\]
To show that this composite is homotopic to the identity it suffices to show that the adjoint map $p_\sharp p^* \to \id$ is the counit of the adjunction $p_\sharp \dashv p^*$. But \eqref{eq:pushforwarddual} is defined by its adjoint map $p_\sharp p_\sharp^\vee(\bS) = (p_\sharp \otimes p_\sharp)\Delta_\sharp p^*\bS \to \bS$ as the composition of
\[
(p_\sharp \otimes p_\sharp)\Delta_\sharp p^* \cong (\Delta_{\pt})_\sharp p_\sharp p^* = p_\sharp p^* \to \id,
\]
where we used the equalities $(p \times p)\circ \Delta_I = \Delta_{\pt} \circ p = p$,
so that precomposing with $p^* \cong (\id \otimes p_\sharp)\Delta_\sharp p^*$ gives the desired counit.
\end{proof}

\subsection{The case of spaces}

In this section we specialize the previous discussion to the case when $I=X$ is a space.

\begin{defn}
The $\infty$-category of \defterm{parametrized spectra} is
\[\Sp^X = \Fun(X, \Sp).\]
\end{defn}

\begin{remark}
Let $M$ be a CW complex and $X=\Sing(M)$ its underlying homotopy type. Then $M$ is locally of singular shape in the sense of \cite[Definition A.4.15]{LurieHA} and so one may identify $\Sp^X$ with the full subcategory $\Shv_{\mathrm{lc}}(M; \Sp)\subset \Shv(M; \Sp)$ of locally constant sheaves of spectra over $M$ \cite[Theorem A.1.15]{LurieHA}.
\end{remark}

\begin{remark}
All constructions and theorems in this section apply equally well to the $\infty$-category of local systems $\LocSys(X) = \Fun(X, \Mod_k)$, where $k$ is an $\bE_\infty$-ring. To simplify the notation, in this section we only consider the case $k=\bS$.
\end{remark}

A map of spaces $f\colon X\rightarrow Y$ induces functors
\[
\xymatrix@C=2cm{
\Sp^X \ar@/^2.0pc/^{f_\sharp}[r] \ar@/_2.0pc/^{f_*}[r] & \Sp^Y \ar_{f^*}[l]
}
\]
with $f_\sharp\dashv f^* \dashv f_*$. In this case the formulas \eqref{eq:Kanextensions} reduce to
\begin{equation}
(f_\sharp \cL)_y\cong \colim_{x\in f^{-1}(y)} \cL_x,\qquad (f_*\cL)_y\cong \lim_{x\in f^{-1}(y)} \cL_x
\label{eq:pushforwardformula}
\end{equation}
as follows from \cite[Theorem 1.2]{HaugsengCoends}. For instance, for $p\colon X\rightarrow \pt$ we have the constant parametrized spectrum $\bS_X=p^*\bS$ and
\[p_\sharp \bS_X\cong \colim_{x\in X} \bS\cong \Sigma^\infty_+ X.\]

The construction of parametrized spectra has the following compatibility with colimits.

\begin{prop}
Consider a functor $I\rightarrow \cS$ given by $i\mapsto X_i$ with colimit $X\in\cS$ and denote the natural projections by $f_i\colon X_i\rightarrow X$. Then the counits $(f_i)_\sharp (f_i)^*\rightarrow \id$ identify
\[\colim_i (f_i)_\sharp (f_i)^*\cong \id.\]
\label{prop:colimitapproximation}
\end{prop}
\begin{proof}
Consider $x\in X$ and $\cF\in\Sp^X$. Then we have to show that
\[\colim_i \colim_{y\in f_i^{-1}(x)} \cF_{f_i(y)}\rightarrow \cF_x\]
is an isomorphism. This is equivalent to showing that $\colim_i \colim_{y\in f_i^{-1}(x)} \pt\cong \colim_i f_i^{-1}(x)\in\cS$ is contractible. But since colimits in $\cS$ are universal (\cite[Lemma 6.1.3.14]{LurieHTT}), we have $\colim_i f_i^{-1}(x)\cong \id^{-1}(x)$.
\end{proof}

Consider the $\infty$-category $\cS_{/X}$ of spaces over $X$ and $(\cS_{/X})_\ast$ of retractive spaces over $X$. There are fiberwise stabilization functors
\[\Sigma^\infty_X\colon (\cS_{/X})_\ast\longrightarrow \Sp^X,\qquad \Sigma^\infty_{+X}\colon \cS_{/X}\longrightarrow \Sp^X.\]
For a space $f\colon Y\rightarrow X$ over $X$ we may identify $f_\sharp \bS_Y\cong \Sigma^\infty_{+X} Y$.

Our next goal is to establish a self-duality of the $\infty$-category of parametrized spectra. For this we will extend parametrized spectra to a bivariant functor using the results of \cite{GaitsgoryRozenblyum1,Macpherson}. Since $\cS$ is freely generated under colimits by $\pt\in\cS$, there is a unique colimit-preserving functor
\[\Sp_\sharp\colon \cS\rightarrow \PrSt\]
which sends $\pt\in\cS$ to $\Sp\in\PrSt$. Moreover, the symmetric monoidal structure on $\Sp$ induces a symmetric monoidal structure on the functor $\Sp_\sharp$. Explicitly, $\Sp_\sharp$ sends $X$ to $\Sp^X$ and $f\colon X\rightarrow Y$ to $f_\sharp\colon \Sp^X\rightarrow \Sp^Y$.

Consider the symmetric monoidal $(\infty, 2)$-category $\Corr(\cS)$ which has the following informal description:
\begin{itemize}
\item Its objects are spaces.

\item Its 1-morphisms from $X$ to $Y$ are correspondences $X\leftarrow C\rightarrow Y$ of spaces.

\item Its 2-morphisms from $X\leftarrow C_1\rightarrow Y$ to $X\leftarrow C_2\rightarrow Y$ are morphisms or correspondences: diagrams of the shape
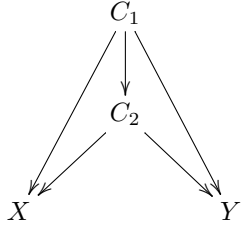
\[
\xymatrix{
& C_1 \ar[d] \ar[ddl] \ar[ddr] & \\
& C_2 \ar[dl] \ar[dr] & \\
X && Y
}
\]
\end{itemize}
We refer to \cite[Chapter 7]{GaitsgoryRozenblyum1} for a construction of this $(\infty, 2)$-category.

There is a natural symmetric monoidal functor $\cS\rightarrow \Corr(\cS)$ which is the identity on objects and which sends $f\colon X\rightarrow Y$ to $X\xleftarrow{\id} X\xrightarrow{f} Y$. In particular, any symmetric monoidal functor out of $\Corr(\cS)$ restricts to a symmetric monoidal functor out of $\cS$.

\begin{thm}
There is a unique symmetric monoidal functor of $(\infty, 2)$-categories
\[\Sp_\sharp^*\colon \Corr(\cS)\longrightarrow \PrSt,\]
whose restriction to $\cS$ is $\Sp_\sharp$.
\label{thm:Spbivariant}
\end{thm}
\begin{proof}
By \cref{prop:basechange} the functor $\Sp_\sharp\colon \cS\rightarrow \PrSt$ satisfies the left Beck--Chevalley condition from \cite[Chapter 7, Definition 3.1.5]{GaitsgoryRozenblyum1}. Therefore, by \cite[Chapter 7, Theorem 3.2.2]{GaitsgoryRozenblyum1} the functor $\Sp_\sharp$ uniquely extends to the functor $\Sp_\sharp^*$ and by \cite[Chapter 9, Proposition 3.1.5]{GaitsgoryRozenblyum1} the symmetric monoidal structure uniquely extends.
\end{proof}

The functor $\Sp_\sharp^*\colon \Corr(\cS)\rightarrow \PrSt$ constructed in the previous theorem has the following informal description:
\begin{itemize}
\item On the level of objects it sends $X$ to $\Sp^X$.

\item On the level of 1-morphisms it sends $X\xleftarrow{f} C\xrightarrow{g} Y$ to $g_\sharp f^*\colon \Sp^X\rightarrow \Sp^Y$.

\item On the level of 2-morphisms it sends
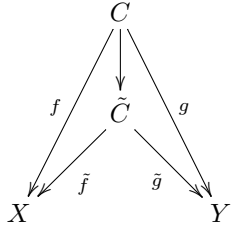
\[
\xymatrix{
& C \ar[d] \ar_{f}[ddl] \ar^{g}[ddr] & \\
& \tilde{C} \ar^{\tilde{f}}[dl] \ar_{\tilde{g}}[dr] & \\
X && Y
}
\]
to the natural transformation $g_\sharp f^*\Rightarrow \tilde{g}_\sharp \tilde{f}^*$ induced by $C\rightarrow \tilde{C}$.
\end{itemize}

Let us state several important corollaries of this construction. First, $\Sp^X$ carries a natural symmetric monoidal structure given by $\Sp^X\otimes \Sp^X\cong \Sp^{X\times X}\xrightarrow{\Delta^*} \Sp^X$, where $\Delta\colon X\rightarrow X\times X$ is the diagonal map. This symmetric monoidal structure gives rise to a self-duality datum on $\Sp^X$.

\begin{prop}
Let $X$ be a space and denote by $p\colon X\rightarrow \pt$ the natural projection. The functors
\[\ev\colon \Sp^X\otimes \Sp^X\xrightarrow{\Delta^*} \Sp^X\xrightarrow{p_\sharp} \Sp\]
and
\[\coev\colon \Sp\xrightarrow{p^*} \Sp^X\xrightarrow{\Delta_\sharp}\Sp^{X\times X}\cong\Sp^X\otimes \Sp^X\]
establish a self-duality of $\Sp^X\in\PrSt$. In particular, $X$ admits duality in the sense of \cref{def:admitduality}.
\label{prop:Spselfdual}
\end{prop}
\begin{proof}
By \cite[Corollary 12.5]{HaugsengSpans} (see also \cite[Proposition 4.1]{BZNNonlinear}) every object $X\in\Corr(\cS)$ is self-dual with the evaluation map $X\times X\xleftarrow{\Delta} X\xrightarrow{p} \pt$ and the coevaluation map $\pt\xleftarrow{p} X\xrightarrow{\Delta} X\times X$. Since $\Sp_\sharp^*$ is symmetric monoidal, $\Sp(X)$ becomes self-dual with the asserted evaluation and coevaluation functors.
\end{proof}

\begin{remark}\label{rmk:CWduality}
Using the self-duality data of $\Sp^X$ from \cref{prop:Spselfdual} we see that a relative duality for $\cF\in\Sp^X$ consists of an object $\cF^\vee\in\Sp^X$ together with maps $\cF\boxtimes\cF^\vee\rightarrow \Delta_!\bS_X$ and $\bS\rightarrow p_\sharp(\cF\otimes \cF^\vee)$ satisfying the duality axioms. This data coincides with the notion of Costenoble--Waner duality of parametrized spectra from \cite[Chapter 18]{MaySigurdsson}.
\end{remark}

Using the duality data we will now compute the dimension of $\Sp^X$. Consider the composite functor
\[\cS\xrightarrow{\Sp_\sharp} \PrDual\xrightarrow{\dim} \Sp\]
which sends $X$ to $\dim(\Sp^X)$ and $f\colon X\rightarrow Y$ to $\dim(f_\sharp)\colon \dim(\Sp^X)\rightarrow \dim(\Sp^Y)$.

\begin{prop}
There is a natural equivalence
\[\dim(\Sp^{-})\cong \Sigma^\infty_+ L(-)\]
of functors $\cS\rightarrow \Sp$. For instance, for a map $f\colon X\rightarrow Y$ the transfer $\dim(f_\sharp)$ fits into a commutative diagram
\[
\xymatrix{
\dim(\Sp^X) \ar^{\sim}[r] \ar^{\dim(f_\sharp)}[d] & \Sigma^\infty_+ LX \ar^{Lf}[d] \\
\dim(\Sp^Y) \ar^{\sim}[r] & \Sigma^\infty_+ LY
}
\]
\label{prop:LocSysTHH}
\end{prop}
\begin{proof}
We refer to \cite[Theorem 3.18]{CCRY} for the natural equivalence $\dim(X)\cong LX$ for $X\in\Corr(\cS)$. Post-composing with $\Sp^*_\sharp$ and using the fact that $\dim$ is symmetric monoidal we get the result.
\end{proof}

Finally, let us describe the case when $\bS_X\in\Sp^X$ is compact.

\begin{defn}
A space $X\in\cS$ is \defterm{finitely dominated} if it is a retract of a finite space.
\end{defn}

Equivalently, $X$ is finitely dominated if, and only if, it is a compact object in $\cS$ (see \cite[Remark 5.4.1.6]{LurieHTT}).

\begin{remark}
Concretely, suppose $M$ is a CW complex. Then $\Sing(M)\in\cS$ is finitely dominated in the above sense if there is a finite CW complex $N$ together with maps $i\colon M\rightarrow N$ and $r\colon N\rightarrow M$ such that $r\circ i$ is homotopic to the identity.
\end{remark}

\begin{prop}
Suppose $X$ is finitely dominated and let $p\colon X\rightarrow \pt$. Then the functor $p_*\colon \Sp^X\rightarrow \Sp$ preserves colimits. In other words, $\bS_X\in\Sp^X$ is compact.
\label{prop:finitetypeproper}
\end{prop}
\begin{proof}
The functor $p_*\colon \Sp^X\rightarrow \Sp$ preserves finite limits, so by stability (see \cite[Proposition 1.1.4.1]{LurieHA}) it preserves finite colimits. Therefore, we only have to show that it preserves filtered colimits.

Let $F\colon I\rightarrow \Sp^X$ be a functor, where $I$ is filtered. The map
\[\colim_{i\in I} (p_* F(i))\longrightarrow p_*(\colim_{i\in I} F(i))\]
is given by
\[\colim_{i\in I} \lim_{x\in X} F(i)_x\longrightarrow \lim_{x\in X} \colim_{i\in I} F(i)_x.\]
In $\Sp$ filtered colimits preserve finite limits. By \cite[Lemma 4.8]{HaugsengBG} it follows that filtered colimits preserve limits indexed by finitely dominated spaces.
\end{proof}

\begin{cor}
Let $X$ be a finitely dominated space. Then $\Sp^X$ is smooth.
\label{cor:finitelydominatedsmooth}
\end{cor}
\begin{proof}
The functor $\coev=\Delta_\sharp\circ p^*$ admits a right adjoint $p_*\circ \Delta^*$. The functor $\Delta^*$ is always colimit-preserving and $p_*$ is colimit-preserving by \cref{prop:finitetypeproper} since $X$ is finitely dominated.
\end{proof}

If $X$ is finitely dominated, by \cref{prop:finitetypeproper} $\bS_X\in\Sp^X$ is compact and, therefore, it admits a relative dual $\zeta_X\in\Sp^X$.

\begin{remark}
By \cite[Corollary 5.1]{Klein}, if $M$ is a finite CW complex and $X=\Sing(M)$, the parametrized spectrum $\zeta_X$ is a suspension of the Spivak normal fibration of $M$.
\end{remark}

From \cref{prop:CWduality} and the explicit formula for the evaluation $\ev=p_\sharp\Delta^*$ of the self-duality of $\Sp^X$ we get the following.

\begin{prop}
Suppose $X$ is a finitely dominated space. Then there is a natural isomorphism
\[p_*(-)\cong p_\sharp\Delta^*((-)\boxtimes \zeta_X)\]
of functors $\Sp^X\rightarrow \Sp$.
\label{prop:zetapushforward}
\end{prop}

\subsection{The case of posets}

In this section we consider functor categories $\Fun(P, \cC)$, where $P$ is a poset. We begin with the computation of additive invariants of $\Fun(P, \cC)$.

\begin{thm}
Let $P$ be a finite poset and $\cC$ an idempotent-complete small stable $\infty$-category. Denote by $P^\delta$ the same set equipped with the trivial poset structure and $\pi\colon P^\delta\rightarrow P$ the functor given by the identity map on objects. Let $F\colon \Catperf\rightarrow \Sp$ be an additive invariant. Then the functors $\pi_\sharp, \pi_*\colon \Fun(P^\delta, \cC)\rightarrow \Fun(P, \cC)$ exist and induce an equivalence
\[\pi_\sharp, \pi_*\colon F(\Fun(P^\delta,\cC))\longrightarrow F(\Fun(P, \cC)).\]
\label{thm:posetadditivity}
\end{thm}
\begin{proof}
The existence of the Kan extension functors $\pi_\sharp,\pi_*$ follows since the relevant limits and colimits in \eqref{eq:Kanextensions} are finite and $\cC$ admits finite limits and colimits.

Let us now prove the claim that $\pi_\sharp\colon F(\Fun(P^\delta,\cC))\rightarrow F(\Fun(P, \cC))$ is an equivalence. We will use induction on the cardinality of $P$ by mimicking the proof of \cite[Theorem A]{AokiPosets}. The claim is obvious for $P$ empty. For an arbitrary finite poset $P$ choose a minimum element $m\in P$ and let $i\colon \{m\}\rightarrow P$ be the inclusion of this element and $j\colon P\setminus m\rightarrow P$ the inclusion of the complement. Consider the sequences
\[
\xymatrix@C=2cm{
\Fun(P\setminus m, \cC) \ar@/^1.0pc/^{j_\sharp}[r] & \Fun(P, \cC) \ar@/^1.0pc/_{j^*}[l] \ar@/^1.0pc/^{i^*}[r] & \cC \ar@/^1.0pc/_{i_*}[l]
}
\]
where the top functors are left adjoint to the bottom functors. From the formulas \eqref{eq:Kanextensions} we see that the functors $j_\sharp$ and $i_*$ are given by extension by zero, so this is a split exact sequence (see also \cite[Lemma 4.1]{AokiPosets}). Therefore, by additivity we get an exact sequence of spectra
\[
F(\Fun(P\setminus m, \cC))\xrightarrow{j_\sharp} F(\Fun(P, \cC)) \xrightarrow{i^*} F(\cC).
\]
Consider a commutative diagram
\[
\xymatrix{
\Fun(P\setminus m, \cC) \ar^{j_\sharp}[r] & \Fun(P, \cC) \ar^{i^*}[r] & \cC \\
\Fun(P^\delta\setminus m, \cC) \ar^{j_\sharp}[r] \ar^{\pi_\sharp}[u] & \Fun(P^\delta, \cC) \ar^{i^*}[r] \ar^{\pi_\sharp}[u] & \cC \ar^{\id}[u]
}
\]
Here the bottom sequence is defined analogously to the top sequence replacing the posets by the same sets with the trivial partial order. The commutativity of the square on the left is obvious and the commutativity of the square on the right follows since $m$ is a minimum. Therefore, we obtain a commutative diagram of spectra
\[
\xymatrix{
F(\Fun(P\setminus m, \cC))\ar^{j_\sharp}[r] & F(\Fun(P, \cC)) \ar^-{i^*}[r] & F(\cC) \\
F(\Fun(P^\delta\setminus m, \cC))\ar^{j_\sharp}[r] \ar^{\pi_\sharp}[u] & F(\Fun(P^\delta, \cC)) \ar^-{i^*}[r] \ar^{\pi_\sharp}[u] & F(\cC) \ar^{\id}[u]
}
\]
where both rows are exact. The claim that the middle vertical map is an equivalence then follows by induction.

The fact that $\pi_*$ induces an equivalence is proven by an analogous induction by splitting off a maximum.
\end{proof}

\begin{example}
Consider the poset $P=\{0\leq 1\}=\Delta^1$. Then $\Fun(P^\delta, \cC)=\cC\times\cC$ and the functors
\[\pi_\sharp,\pi_*\colon \cC\times\cC\longrightarrow \Fun(\Delta^1, \cC)\]
are given by
\[\pi_\sharp(x, y) = (x\rightarrow x\oplus y),\qquad \pi_*(x, y) = (x\oplus y\rightarrow y).\]
\end{example}

The posets we encounter will be the face posets of triangulations. Namely, let $M$ be a compact polyhedron and choose a triangulation of $M$ with face poset $T$. For $\tau \in T$ let $\ost(\tau)\subset M$ denote its open star, i.e. the union of the interiors of simplices containing $\tau$.

\begin{thm}
Let $T$ be the face poset of a triangulation of a closed PL manifold $M$. Then:
\begin{enumerate}
    \item $(p^{\op})^\sharp(\bS) = \zeta_{T^{\op}}\in\Sp^{T^{\op}}$ is locally constant and invertible, i.e. for every $\tau\in T$ the spectrum $\zeta_{T^{\op}}(\tau)$ is invertible and for every $\tau\subseteq \sigma$ the induced map $\zeta_{T^{\op}}(\sigma)\rightarrow \zeta_{T^{\op}}(\tau)$ is an isomorphism.
    \item $T^{\op}$ admits duality in the sense of \cref{def:admitduality}.
\end{enumerate}
\label{thm:faceposetduality}
\end{thm}
\begin{proof}
Let us first prove the first statement. Since $T$ is finite, $\bS_T\in\Sp^T$ is compact. In particular, $(p^{\op})^\sharp$ is well-defined. Let $\coev(\bS) \in \Sp^{T^{\op}}\otimes \Sp^T$ be the coevaluation from \cref{ex:presheafdual}. Then by \cref{prop:presheafduality} we get the formula $\zeta_{T^{\op}} \cong (\id\boxtimes p_*) \coev(\bS)$, i.e.
\[\zeta_{T^{\op}}(\tau)\cong\lim_{\sigma\supseteq \tau}\bS.\]
This object is Spanier--Whitehead dual to
\[\xi(\tau)\cong \colim_{\sigma\supseteq \tau}\bS.\]
So, it is enough to show that $\xi\in\Sp^T$ is locally constant and invertible. This object fits into a cofiber sequence
\[\colim_{\sigma\not\supseteq \tau} \bS\longrightarrow \colim_\sigma \bS\longrightarrow \xi(\tau).\]
The first colimit may be identified with $\Sigma^\infty_+ (M\setminus \ost(\tau))$ and the second colimit with $\Sigma^\infty_+ M$. Therefore,
\[\xi(\tau)\cong \Sigma^\infty_+ M / \Sigma^\infty_+(M\setminus \ost(\tau)).\]
Consider $\tau_1\subseteq \tau_2$ and a point $x$ in the interior of $\tau_1$. Then
\[\Sigma^\infty_+ M / \Sigma^\infty_+(M\setminus \ost(\tau_i))\longrightarrow \Sigma^\infty_+ M / \Sigma^\infty_+(M\setminus \{x\})\]
is an isomorphism since $M\setminus \ost(\tau_i)\rightarrow M\setminus \{x\}$ is a deformation retract. This shows that $\xi(\tau_1)\rightarrow \xi(\tau_2)$ is an isomorphism. Since $M$ is a PL manifold, the cofiber $M/(M\setminus \{x\})$ is homotopy equivalent to $S^n$ for some $n$ which finishes the proof of the first statement.

Next, let us prove the second statement. It is shown in \cite[Proposition 6.6.1]{CDHHLMNNS1} that $\Sp^T$ is self-dual with the evaluation pairing $\ev_T = p_*\Delta^*$. Using the duality between $\Sp^T$ and $\Sp^{T^{\op}}$ from \cref{ex:presheafdual} we transfer the self-duality of $\Sp^T$ to a self-duality of $\Sp^{T^{\op}}$ with the coevaluation pairing given by $\coev_{T^{\op}}(\bS) = \Delta_\sharp \zeta_{T^{\op}}$ by using that the coevaluation is dual to evaluation and the computation of dual functors from \cref{prop:presheafduality}. Since $\zeta_{T^{\op}}$ is locally constant and invertible, we have a natural isomorphism
\[\Delta_\sharp \bS_{T^{\op}}\cong (\zeta_{T^{\op}}^{-1}\boxtimes \bS_{T^{\op}}) \otimes \Delta_\sharp \zeta_{T^{\op}}\]
and hence $\widetilde{\coev}_{T^{\op}}(\bS)=\Delta_\sharp \bS_{T^{\op}}$ is also the coevaluation of a self-duality of $\Sp^{T^{\op}}$.
\end{proof}

\section{Assembly maps and the Euler characteristic}
\label{sect:assembly}

In this section we introduce assembly maps, the Euler characteristic and its lifts along the assembly map.

\subsection{Assembly map}

For the following definition see e.g. \cite{WeissWilliams}.

\begin{defn}
    Let $\bF\colon \cS\rightarrow \Sp$ be a functor from spaces to spectra. The \defterm{universal excisive approximation} is a colimit-preserving functor $\bF^{\%}\colon \cS\rightarrow \Sp$ together with a natural transformation $\bF^{\%}(-)\rightarrow \bF(-)$ which is an equivalence on $\pt$. We call this morphism the \defterm{assembly map}. In particular, we can always write it as
    \[
    \alpha \colon \Sigma^\infty_+X \otimes \bF(\pt) \longrightarrow \bF(X).
    \]
\end{defn}

We will mainly be concerned with the following examples:
\begin{itemize}
    \item For $X\in\cS$ a topological space let $\bA(X)\in\Sp$ be the $K$-theory spectrum of the stable $\infty$-category $(\Sp^X)^\omega$, the \defterm{$A$-theory} of the space $X$. Denote by $A(X) = \Omega^\infty \bA(X)$ the underlying space. Then we have the assembly map
    \[\alpha\colon \Sigma^\infty_+X\otimes \bA(\pt)\longrightarrow \bA(X).\]
    
    \item Consider the functor $X\mapsto \dim(\Sp^X)\cong \Sigma^\infty_+ LX$ (see \cref{prop:LocSysTHH} for the equivalence). The assembly map is given by the inclusion of constant loops
    \[\alpha\colon \Sigma^\infty_+ X\longrightarrow \Sigma^\infty_+ LX.\]

    \item The Dennis trace for $A$-theory defines a morphism of spectra
    \[\tr\colon \bA(X)\longrightarrow \Sigma^\infty_+ LX.\]
    By the universal property of the assembly map, we have a commutative diagram
    \[
    \xymatrix{
    \Sigma^\infty_+X\otimes \bA(\pt) \ar^{\alpha}[r] \ar^{\tr}[d] & \bA(X) \ar^{\tr}[d] \\
    \Sigma^\infty_+ X \ar^{\alpha}[r] & \Sigma^\infty_+ LX.
    }
    \]
\end{itemize}

\begin{remark}
    We follow the definition of the $A$-theory space given in \cite[Lecture 21]{LurieKtheory}, which coincides with the space $A(X)$ defined in \cite{DwyerWeissWilliams} in terms of the Waldhausen category of (homotopy) \emph{finitely dominated} retractive spaces over $X$. The original definition due to Waldhausen \cite{Waldhausen} uses the category of (homotopy) \emph{finite} retractive spaces over $X$. The difference only affects $\pi_0$.
\end{remark}

Given a compact parametrized spectrum $\cF\in\Sp^X$, we will be interested in lifts of its class $[\cF]\in A(X)$ under the assembly map, so let us introduce the relevant space. For $\chi\in A(X)$ denote by
\[\Lift(\chi,\alpha) = \pi_0 \hofib_{\chi}\left(\alpha\colon \Omega^\infty(\Sigma^\infty_+X\otimes \bA(\pt))\rightarrow A(X)\right),\]
where the homotopy fiber is taken at $\chi\in A(X)$. We will use the same notation for the other assembly maps.

Let us now introduce the Euler characteristic. Suppose $X$ is a finitely dominated space. Then the constant parametrized spectrum $\bS_X\in\Sp^X$ is compact and hence it defines a class $[\bS_X]\in A(X)$ in $A$-theory; applying the Dennis trace we obtain an element $[\bS_X]\in\Omega^\infty\Sigma^\infty_+ LX$.

\begin{defn}
    Let $X$ be a finitely dominated space. The \defterm{$A$-theoretic Euler characteristic} is
    \[\chi_A(X) = [\bS_X]\in A(X)\]
    and the \defterm{$\THH$ Euler characteristic} is
    \[\chi_\THH(X) = [\bS_X]\in \Omega^\infty \Sigma^\infty_+ LX.\]
\end{defn}

Denote by $p\colon X\rightarrow \pt$ the natural projection. Then under the composite
\[\pi_0(A(X))\longrightarrow \pi_0(A(\pt)) = \Z\]
the image of $\chi_A(X)$ is the usual Euler characteristic of $X$. So, $\chi_A(X)$ may be considered as a local version of the usual Euler characteristic.

\subsection{Lifts of the Euler characteristic}

In this section we recall connections between the assembly map for $A$-theory and (simple) homotopy theory. Recall the following space introduced in \cite{Waldhausen}.

\begin{defn}
    Let $X$ be a space. The \defterm{PL Whitehead spectrum} $\bWh^{PL}(X)$ is the cofiber of the assembly map $\alpha\colon \Sigma^\infty_+X\otimes \bA(\pt)\rightarrow \bA(X)$. The \defterm{PL Whitehead space} is the underlying space $\Wh^{PL}(X) = \Omega^\infty\bWh^{PL}(X)$.
\end{defn}

\begin{remark}
    Suppose $X$ is connected and based. Then
    \[\pi_0(\Wh^{PL}(X))\cong \tilde{K}_0(\Z[\pi_1(X)])\]
    is the reduced $K$-theory group and
    \[\pi_1(\Wh^{PL}(X))\cong \Wh(\pi_1(X))\]
    is the Whitehead group.
\end{remark}

\begin{defn}
    Let $X$ be a finitely dominated space. \defterm{Wall's finiteness obstruction} $w(X)\in\pi_0(\Wh^{PL}(X))$ is the image of $\chi_A(X)\in A(X)$ under $A(X)\rightarrow \Wh^{PL}(X)$.
\end{defn}

The existence of a fiber sequence
\[\Omega^\infty(\Sigma^\infty_+X\otimes \bA(\pt))\xrightarrow{\alpha} A(X)\longrightarrow \Wh^{PL}(X)\]
implies the following:
\begin{itemize}
    \item $\Lift(\chi_A(X), \alpha)$ is nonempty if, and only if, Wall's finiteness obstruction $w(X)$ vanishes.
    \item If $\Lift(\chi_A(X), \alpha)$ is nonempty, it is a torsor over the Whitehead group $\pi_1(\Wh^{PL}(X))$.
\end{itemize}

A fundamental result due to Wall \cite{Wall} asserts that $w(X) = 0$ if, and only if, there is a finite CW complex $M$ (equivalently, a finite polyhedron) such that $X\cong \Sing(M)$. We will need an explicit construction of the trivialization of Wall's finiteness obstruction.

Let $M$ be a finite polyhedron with $X=\Sing(M)$ and choose a PL triangulation of $M$ with face poset $T$. The collection of open stars assembles into a functor $\ost\colon T^{\op}\rightarrow \cU(M)$ to the category of open subsets of $M$.

\begin{lm}
The inclusion $\ost(\tau)\subset M$ of open stars induces an isomorphism
\[\colim_{\tau\in T^{\op}} \Sing(\ost(\tau))\cong \Sing(M).\]
\label{lm:openstargoodcover}
\end{lm}
\begin{proof}
For $x\in M$ denote by $T^{\op}_x\subset T^{\op}$ the subcategory of simplices $\tau$ such that $x\in\ost(\tau)$. If $x$ lies in the interior of $\sigma\in T$, then $T^{\op}_x\cong (T_{\leq \sigma})^{\op}$. The category $T_{\leq \sigma}$ is the face poset of the simplex $\sigma$ and hence its nerve is weakly contractible. The claim then follows from \cite[Theorem A.3.1]{LurieHA}.
\end{proof}

The colimit $\colim_{\tau\in T^{\op}} \pt\in\cS$ is the localization of the category $T^{\op}$ obtained by inverting all morphisms. So, we obtain a functor of $\infty$-categories
\begin{equation}
t\colon T^{\op}\longrightarrow X.
\label{eq:Tlocalization}
\end{equation}
Therefore, restriction along $t$ and the left Kan extension define an adjunction
\[
\xymatrix@C=2cm{
\Sp^{T^{\op}} \ar@/^1.0pc/^{t_\sharp}[r] & \Sp^X \ar@/^1.0pc/_{t^*}[l]
}
\]
with $t^*$ fully faithful. In particular, $t_\sharp t^*\bS_X\cong \bS_X$. Since $T$ is finite, $\bS_T=t^*\bS_X\in\Sp^{T^{\op}}$ is compact, and so we obtain a lift of $[\bS_X]\in A(X)$ along
\[t_\sharp\colon \bK(\Sp^{T^{\op}, \omega})\longrightarrow \bA(X).\]
Next, let $T^\delta$ be the set $T$ equipped with the trivial poset structure and $\pi\colon T^\delta\rightarrow T^{\op}$ the functor given by the identity on objects. By \cref{thm:posetadditivity} applied to $\cC=\Sp^\omega$ (note that $\Fun(T^{\op}, \Sp^\omega)=\Fun(T^{\op}, \Sp)^\omega$ by \cref{prop:compactcharacterization}) the map
\[\pi_\sharp\colon \bA(T^\delta)\longrightarrow \bK(\Sp^{T^{\op}, \omega})\]
is an equivalence and hence we obtain a lift of $[\bS_X]\in A(X)$ along
\[(t\circ \pi)_\sharp\colon \bA(T^\delta)\longrightarrow \bA(X).\]
Now consider a commutative diagram of assembly maps
\[
\xymatrix{
\bA(T^\delta) \ar[r] & \bA(X) \\
\Sigma^\infty_+T^\delta\otimes \bA(\pt) \ar^{t\circ \pi}[r] \ar^{\alpha}[u] & \Sigma^\infty_+X\otimes \bA(\pt) \ar^{\alpha}[u]
}
\]

Since $T^\delta$ is a finite set, the assembly map $\Sigma^\infty_+T^\delta\otimes A(\pt)\rightarrow A(T^\delta)$ is an equivalence and hence in this way we obtain a lift of $[\bS_X]\in A(X)$ along the assembly map
\[\alpha\colon \Sigma^\infty_+ X\otimes \bA(\pt)\longrightarrow \bA(X).\]

\begin{defn}
Let $M$ be a finite polyhedron with underlying homotopy type $X=\Sing(M)$. The lift of $[\bS_X]\in A(X)$ defined above is the \defterm{Whitehead lift} $\lambda_{\Wh}(M)\in\Lift(\chi_A(X), \alpha)$.
\label{def:Whiteheadlift}
\end{defn}

\begin{remark}
    The construction of the assembly map is equivalent to \cite[Construction 9.]{LurieKtheory}. More precisely, in loc. cit. it is defined as the dual of $\Sp^X \to \Fun(T, \Sp)$ (thought of as inclusion of locally constant into $T$-constructible sheaves) and Verdier duality is used to identify $\Fun(T, \Sp)^\vee \cong \Fun(T, \Sp)$. In the above description we instead take $\Fun(T, \Sp)^\vee \cong \Fun(T^{\op}, \Sp)$ and use \cref{prop:presheafduality} to identify the dual functor.
\end{remark}
	
	Let $f\colon M_1\rightarrow M_2$ be a homotopy equivalence of finite polyhedra, inducing an isomorphism $f\colon X_1\rightarrow X_2$ of their homotopy types. The commutative diagram
	\[
	\xymatrix{
		\Sigma^\infty_+X_1\otimes \bA(\pt) \ar^-{\alpha}[r] \ar^{f}[d] & \bA(X_1) \ar^{f_\sharp}[d] \\
		\Sigma^\infty_+X_2\otimes \bA(\pt) \ar^-{\alpha}[r] & \bA(X_2)
	}
	\]
	as well as the natural equivalence $f_\sharp\bS_{X_1}\cong \bS_{X_2}$ gives a map
	\[f\colon \Lift(\chi_A(X_1), \alpha)\longrightarrow \Lift(\chi_A(X_2), \alpha).\]
	In particular, we obtain an element
	\[\tau(f) = f(\lambda_{\Wh}(M_1)) - \lambda_{\Wh}(M_2)\in\pi_1(\Wh^{PL}(X_2)).\]

	The following statement follows by comparing the Whitehead torsion of a bounded acyclic complex of free modules to the path in the K-theory space obtained using the additivity theorem (see \cite[Lecture 27, Proposition 5]{LurieKtheory}).
	
	\begin{prop}
		Let $f\colon M_1\rightarrow M_2$ be a homotopy equivalence of finite polyhedra, inducing an isomorphism $f\colon X_1\rightarrow X_2$ of their underlying homotopy types. The element $\tau(f)\in \pi_1(\Wh^{PL}(X_2))$ constructed above coincides with the Whitehead torsion of $f$.
		\label{prop:whiteheadtorsion}
	\end{prop}

	Homotopy equivalences $f\colon M_1\rightarrow M_2$ with vanishing Whitehead torsion $\tau(f)\in\pi_1(\Wh^{PL}(X_2))$ are known as \defterm{simple homotopy equivalences}. Therefore, we see that the lift $\lambda_{\Wh}(M)\in\Lift(\chi_A(X), \alpha)$ is simple homotopy invariant, but is not homotopy invariant.

 \begin{example}
 If $f\colon M_1\rightarrow M_2$ is a homeomorphism of finite polyhedra, it is a simple homotopy equivalence \cite{Chapman}.
 \end{example}

\section{Pontryagin--Thom lift}

The goal of this section is to construct a lift of the $\THH$ Euler characteristic along the assembly map using intersection theory and compare it to the Whitehead lift from \cref{def:Whiteheadlift}. More precisely, we construct a $\THH$ lift using a version of Pontryagin--Thom collapse, which we define using the configuration space of two points. Next, we show that this lift coincides with the trace of the Whitehead lift defined in the previous section.

\subsection{Lift diagrams}\label{sec:liftdiagrams}

To construct a lift of the $\THH$ Euler characteristic geometrically, it will be convenient to present the data of a lift in terms of a commutative diagram, which we call a \emph{lift diagram}. In fact, with a view towards the proof of \cref{thm:florian}, we will define the notion of lifts for a nice class of $\infty$-categories, which, in particular, contains finitely dominated spaces (viewed as $\infty$-groupoids). Let $I$ be a small $\infty$-category.

Consider the commutative diagram
\[
\xymatrix@C=1cm{
\Sp^I \ar^{\Delta_\sharp}[r] \ar^{\Delta_\sharp}[d] & \Sp^I\otimes \Sp^I \ar^{\id\boxtimes \Delta_\sharp}[d] \\
\Sp^I\otimes \Sp^I \ar^-{\Delta_\sharp\boxtimes \id}[r] & \Sp^I\otimes \Sp^I\otimes \Sp^I
}
\]
Passing to right adjoints of the horizontal functors, we obtain a natural transformation
\[\Delta_\sharp\circ \Delta^*\Rightarrow (\Delta^*\boxtimes \id)\circ (\id\boxtimes \Delta_\sharp).\]
Pre- and post-composing it with $\bS_I$ and $p_\sharp$, we obtain the natural transformation
\begin{equation}\label{eqn:transftoT}
\id\Rightarrow \left(\cT_I=(p_\sharp\boxtimes \id)\circ(\Delta^*\boxtimes \id) \circ (\id\boxtimes \Delta_\sharp(\bS_I))\right).
\end{equation}
Taking its trace we obtain a morphism
\[\beta_I\colon \THH(\Sp^{I, \omega})\cong\dim(\Sp^I)\longrightarrow \tr(\cT_I)\cong p_\sharp \Delta^*\Delta_\sharp \bS_I.\]
\begin{remark}
The morphism $\beta_I$ may be identified with
\[\int^{i\in I}\Sigma^\infty_+\Hom_I(i, i)\longrightarrow \colim_{i,j\in I} \Sigma^\infty_+(\Hom_I(i, j)\times \Hom_I(i, j))\]
induced by the map $\Hom_I(i, i)\rightarrow \Hom_I(i, i)\times \Hom_I(i, i)$ sending $f \mapsto f \times \id_i$.
Moreover, $\tr(\cT_I)$ can be identified with the suspension spectrum of the geometric realization of the category of parallel arrows $\Fun(\Delta^1 \sqcup_{\partial \Delta^1} \Delta^1, I)$.
\end{remark}

The unit map $\bS_I\rightarrow \Delta^*\Delta_\sharp \bS_I$ induces a morphism
\[\alpha_I\colon p_\sharp \bS_I\longrightarrow p_\sharp\Delta^*\Delta_\sharp \bS_I.\]

If we further assume that $\bS_I\in\Sp^I$ is compact, then we may define the element $\chi_{\THH}(I) = [\bS_I]\in\Omega^\infty\dim(\Sp^I)$.

\begin{defn}\label{def:liftcategory}
Let $I$ be a small $\infty$-category with $\bS_I\in\Sp^I$ compact. A \defterm{lift of $\chi_{\THH}(I)$} is an element $e(I)\in \Omega^\infty p_\sharp \bS_I$ together with a homotopy $\alpha_I(e(I))\sim \beta_I([\bS_I])$. We denote by
\[\Lift(\beta_I(\chi_{\THH}(I)), \alpha_I) = \pi_0\hofib_{\beta_I(\chi_{\THH}(I))}(\alpha_I\colon \Omega^\infty p_\sharp\bS_I\rightarrow \tr(\cT_I))\]
the set of lifts.
\end{defn}

\begin{example}
If $X$ is an $\infty$-groupoid, the morphism $\id\Rightarrow \cT_X$ is an equivalence and hence the morphism
\[\beta_X\colon \dim(\Sp^X)\rightarrow p_\sharp \Delta^*\Delta_\sharp \bS_X\]
is an equivalence. The composite
\[\Sigma^\infty_+ X\xrightarrow{\alpha_X} \tr(\cT_X)\xleftarrow[\beta_X]{\sim} \dim(\Sp^X)\cong \Sigma^\infty_+ LX\]
is given by the inclusion of constant loops, and hence it coincides with the $\THH$ assembly map. Therefore, if $X$ is finitely dominated (so that $\bS_X\in\Sp^X$ is compact), the set of lifts $\Lift(\beta_X(\chi_{\THH}(X)), \alpha_X)$ is naturally isomorphic to the set of lifts $\Lift(\chi_{\THH}(X), \alpha)$ of the $\THH$ Euler characteristic along the assembly map.
\label{ex:generalliftspace}
\end{example}

\begin{example}
If $T$ is a poset, the unit $\id\rightarrow \Delta^*\Delta_\sharp$ is an equivalence, which follows from the fact that $\Hom_T(i, j)\rightarrow \Hom_T(i, j)\times \Hom_T(i, j)$ is an equivalence. In particular,
\[\alpha_T\colon \Sigma^\infty_+ |T|\cong p_\sharp \bS_T\longrightarrow \tr(\cT_T)\]
is an equivalence and $\Lift(\beta_I(\chi_{\THH}(I)), \alpha_I)$ is a singleton (assuming it is defined, i.e. $\bS_I$ is compact).
\label{ex:posetalpha}
\end{example}

\begin{defn}
Let $I$ be a small $\infty$-category which admits duality (see \cref{def:admitduality}). A \defterm{lift diagram} for $I$ is given by the following data:
    \begin{itemize}
        \item An object $\zeta_I\in\Sp^I$.
        \item A morphism $\epsilon \colon \bS_I \boxtimes \zeta_I\rightarrow \Delta_\sharp \bS_I$ which exhibits $\zeta_I$ as the relative dual of $\bS_I$.
        \item An \defterm{Euler class} $e \colon \zeta_I \to \bS_I$.
        \item A homotopy commuting diagram
        \[
        \begin{tikzcd}
            \Delta_\sharp \zeta_I \ar[r] \ar[d, "\Delta_\sharp e"] & \zeta_I \boxtimes \bS_I \ar[dl, "\epsilon"] \\
            \Delta_\sharp \bS_I
        \end{tikzcd}
        \]
    \end{itemize}
\end{defn}

\begin{remark}
By \cref{prop:Spselfdual} any $\infty$-groupoid $X$ admits duality in the sense of \cref{def:admitduality}. Therefore, the previous definition applies to $\infty$-groupoids.
\end{remark}

Our goal now is to relate lift diagrams to lifts of the Euler characteristic from \cref{def:liftcategory}. For this, we begin with a technical lemma.

\begin{lm}\label{lm:descofcT}
    The following diagram commutes:
    \[
    \begin{tikzcd}
        \ev_{\Sp^I} \ar[r, "\eqref{eqn:transftoT}"] & \ev_{\Sp^I} \circ (\cT_I \otimes \id) \\
        p_\sharp \Delta_\sharp^\vee \ar[r, "\eqref{eq:Deltasharpdual}"] \ar[u, "\eqref{eq:evzeta}", "\sim" right] & p_\sharp \Delta^*, \ar[u, "\sim" right]
    \end{tikzcd}
    \]
    where the right vertical map is the identification
    \[
    \ev_{\Sp_I} \circ (\cT_I \otimes \id) (\cong p_\sharp \Delta^*) \circ (\id \boxtimes \ev_{\Sp_I}) \circ (\id \boxtimes \coev_{\Sp_I} \boxtimes \id) \cong p_\sharp \Delta^*
    \]
\end{lm}
\begin{proof}
Applying another coevaluation we instead show that the diagram
\[
\begin{tikzcd}
    \coev(\bS) \ar[r] & (\cT \otimes \id)(\coev(\bS)) \\
    \Delta_\sharp \bS_I \ar[r] \ar[u, "\sim"] & (\Delta^*)^\vee \bS_I \ar[u, "\sim"]
\end{tikzcd}
\]
commutes. The top arrow is obtained from
\[
\begin{tikzcd}[column sep=2cm]
    \Sp^I \ar[r, "\Delta_\sharp"]  \ar[d, "(\Delta_\sharp \otimes \id) \circ \Delta_\sharp"]& \Sp^I \otimes \Sp^I \ar[d, "\Delta_\sharp \otimes \Delta_\sharp"] \\
    \Sp^I \otimes \Sp^I \otimes \Sp^I \ar[r, "\id \otimes \Delta_\sharp \otimes \id"'] & \Sp^I \otimes \Sp^I \otimes \Sp^I \otimes \Sp^I,
\end{tikzcd}
\]
by passing to horizontal right adjoints and pre-/post-composing with $\bS_I \otimes \bS_I$ and $\id \otimes p_\sharp \otimes \id$, respectively. Note that the diagram is expressing the fact that $(f_\sharp \otimes \id) \circ (\id \otimes f_\sharp)\Delta_\sharp \cong \Delta_\sharp f_\sharp$ for $f = \Delta$ (and up to a permutation of the factors). But this is exactly how the map $f_\sharp \to (f^*)^\vee$ is constructed in \cref{prop:presheafselfduality}.
\end{proof}

\begin{thm}
\label{thm:liftdiagrams}
Let $I$ be a small $\infty$-category which admits duality and such that $\bS_I\in\Sp^I$ is compact. Then the set $\Lift(\beta_I(\chi_{\THH}(I)), \alpha_I)$ is naturally isomorphic to the set of lift diagrams up to homotopy.
\end{thm}
\begin{proof}
Since $\bS_I\in\Sp^I$ is compact, by \cref{prop:CWduality} it admits a relative dual. Therefore, by the usual uniqueness of duality data the space of pairs $(\zeta_I, \epsilon)$ is contractible. Given that, we see that the space of lift diagrams is identified with the fiber of
\[\Delta_\sharp\colon \bHom_{\Sp^I}(\zeta_I, \bS_I)\longrightarrow \bHom_{\Sp^{I\times I}}(\Delta_\sharp\zeta_I, \Delta_\sharp\bS_I)\]
over the point given by the composition
\[
\Delta_\sharp \zeta_I \to \zeta_I \boxtimes \bS_I \xrightarrow{\epsilon} \Delta_\sharp \bS_I \in \bHom_{\Sp^{I \times I}}(\Delta_\sharp \zeta_I, \Delta_\sharp\bS_I).
\]
So, the claim follows once we identify the first morphism with $\alpha_I$ and the given point with $\beta_I([\bS_I])$. The first morphism is equivalent to
\[
\Hom_{\Sp^I}(\zeta_I, \bS_I) \to \Hom_{\Sp^I}( \zeta_I, \Delta^* \Delta_\sharp \bS_I)
\]
where we postcompose with the unit $\id \to \Delta^*\Delta_\sharp$, thus coinciding with the map $\alpha_I$ by using $p_\sharp \cong \Hom_{\bS^I}(\zeta_I, -)$. 

The class $\beta_I([\bS_I])$ is obtained by taking the trace of $(\Sp, \id_\Sp) \overset{\bS_I}{\to} (\Sp^I, \id_{\Sp^I}) \to (\Sp^I, \cT_I)$ in $\End \PrSt$. Hence by functoriality of traces and \cite[Lemma 2.4]{HSS} it is computed by traversing the following diagram

\begin{equation}\label{diag:HCEulerClass}
\begin{tikzcd}
    \Sp \ar[r,equal] \ar[d, equal] & \Sp \ar[r, equal] \ar[d, "\bS_I \boxtimes \zeta_I"] \ar[dl, Rightarrow, shorten = 2.5ex] & \Sp \ar[r,equal] \ar[d, "\bS_I \boxtimes \zeta_I"'] & \Sp \ar[d,equal] \ar[dl, Rightarrow, shorten = 2.5ex] \\
    \Sp \ar[r, "\Delta_\sharp \bS_I"] \ar[d, equal] & \Sp^I \otimes \Sp^I \ar[r, equal] \ar[d, equal] & \Sp^I \otimes \Sp^I \ar[r, "\ev_{\Sp^I}"] \ar[d, equal] \ar[dl, Rightarrow, shorten = 2.5ex] & \ar[d, equal] \Sp \\
    \Sp \ar[r, "\Delta_\sharp \bS_I"] & \Sp^I \otimes \Sp^I \ar[r, "\cT \otimes \id"] \ar[rr, bend right=30, "p_\sharp \Delta^*"', ""{name=U}] & \Sp^I \otimes \Sp^I \ar[r, "\ev_{\Sp^I}"] & \Sp,
\end{tikzcd}
\end{equation}
where only the non-invertible two-cells are indicated. \Cref{lm:descofcT} shows that 
\[
\begin{tikzcd}
   \Sp^I \otimes \Sp^I \ar[r, equal] \ar[d, equal] & \Sp^I \otimes \Sp^I \ar[d, equal] \ar[dl, Rightarrow, shorten = 2.5ex] & \\
   \Sp^I \otimes \Sp^I \ar[r, "\cT \otimes \id"] \ar[rr, bend right=30, "p_\sharp \Delta^*"', ""{name=U}] & \Sp^I \otimes \Sp^I \ar[r, "\ev_{\Sp^I}"] & \Sp
\end{tikzcd}
\]
is equivalent to
\[
\begin{tikzcd}
\Sp^I \otimes \Sp^I \ar[r, "{\Delta_\sharp^\vee}" pos=0.45, ""'{name=S}, bend left=30] \ar[r, bend right=30, "\Delta^*"', ""{name=T}] \ar[from=S, to=T, "\eqref{eq:Deltasharpdual}", Rightarrow] & \Sp^I \ar[r, "p_\sharp"] & \Sp,
\end{tikzcd}
\]
and \cref{lm:counitforunit} shows that
\[
\begin{tikzcd}
\Sp \ar[r, equal] \ar[d, "\bS_I \boxtimes \zeta_I"'] & \Sp \ar[d, equal] \ar[dl, Rightarrow, shorten = 2.5ex] \\
\Sp^I \otimes \Sp^I \ar[r, "\ev_{\Sp^I}"] & \Sp 
\end{tikzcd}
 \quad = \quad
\begin{tikzcd}
\Sp \ar[r, equal] \ar[d, "\bS_I \boxtimes \zeta_I"'] & \Sp \ar[r, equal] \ar[d, "\zeta_I"] \ar[dl, Rightarrow, shorten = 2.5ex] & \Sp \ar[d, equal] \ar[dl, Rightarrow, shorten = 2.5ex] \\
\Sp^I \otimes \Sp^I \ar[r, "\Delta_\sharp^\vee"] & \Sp^I \ar[r, "p_\sharp"] & \Sp.
\end{tikzcd}
\]
With that we simplify diagram \eqref{diag:HCEulerClass} and obtain
\[
\begin{tikzcd}
    \Sp \ar[r, equal] \ar[d, equal] & \Sp \ar[r, equal] \ar[d, "\bS_I \boxtimes \zeta_I"] \ar[dl, Rightarrow, shorten = 2.5ex] & \Sp \ar[r, equal] \ar[d, "\zeta_I"] \ar[dl, Rightarrow, shorten = 2.5ex] & \Sp \ar[d, equal] \ar[dl, Rightarrow, shorten = 2.5ex] \\
\Sp \ar[r, "\Delta_\sharp \bS_I"]& \Sp^I \otimes \Sp^I \ar[r, "\Delta_\sharp^\vee", ""'{name=S}, bend left=0] \ar[r, bend right=50, "\Delta^*"', ""{name=T}] & \Sp^I \ar[r, "p_\sharp"] \ar[from=S, to=T, Rightarrow] & \Sp.
\end{tikzcd}
\]
Finally, using the compatibility of the unit $\bS_I$ with $\Delta_\sharp^\vee \xrightarrow{\eqref{eq:Deltasharpdual}} \Delta^*$ this simplifies further to
\[
\begin{tikzcd}
    \Sp \ar[r, equal] \ar[d, equal] & \Sp \ar[r, equal] \ar[d, "\bS_I \boxtimes \zeta_I"] \ar[dl, Rightarrow, shorten = 2.5ex] & \Sp \ar[r, equal] \ar[d, "\zeta_I"] & \Sp \ar[d, equal] \ar[dl, Rightarrow, shorten = 2.5ex] \\
\Sp \ar[r, "\Delta_\sharp \bS_I"']& \Sp^I \otimes \Sp^I \ar[r, "\Delta^*"'] & \Sp^I \ar[r, "p_\sharp"'] & \Sp,
\end{tikzcd}
\]
which shows the claim.
\end{proof}

\subsection{Poincare duality for manifolds}
	\label{sect:manifoldPoincareduality}
	In this section we give a version of (twisted) Poincar\'{e} duality for closed manifolds. We follow the notations of \cite{MaySigurdsson} apart from denoting the left adjoint to $f^*$ by $f_\sharp$:
 \begin{itemize}
     \item For a topological space $M$ we consider the category $\Top_M$ of ex-spaces (alias, retractive spaces) over $M$ as in \cite{MaySigurdsson}, i.e. topological spaces $N$ together with continuous maps $r\colon N\rightarrow M$ and $i\colon M\rightarrow N$ satisfying $r\circ i=\id_M$. It admits a model structure with underlying $\infty$-category $(\cS_{/\Sing(M)})_\ast$ as well as a symmetric monoidal structure $\wedge_M$.
     \item For a continuous map $f\colon X\rightarrow Y$ of topological spaces we consider $(X, f)_+ = X\sqcup Y$ as a retractive space over $Y$.
     \item For a closed subset $K\subset L$ over $M$ we denote by $\C_M(K, L)\in\Top_M$ the unreduced relative mapping cone, see \cite[Chapter 18.4]{MaySigurdsson}. For morphisms $K\rightarrow L\rightarrow M$ in $\cS$ we denote
     \[\C_M(K, L)=L\coprod_K M\in(\cS_{/M})_\ast\]
     the corresponding (homotopy) pushout.
     \item For a vector bundle $V\rightarrow M$ we denote by $S^V\in\Top_M$ the corresponding retractive space obtained by a fiberwise one-point compactification and by $\bS^V=\Sigma^\infty_X(S^V)\in\Sp^{\Sing(M)}$ its fiberwise stabilization.
 \end{itemize}

 Let $M$ be a closed manifold and $X=\Sing(M)$ its underlying homotopy type. Consider an embedding $M\subset V$ into a Euclidean space. Let $\nu$ be the normal bundle of the embedding and
	\[\eta\colon \bS^V\longrightarrow p_\sharp \bS^\nu\]
	the corresponding Pontryagin--Thom collapse map. We denote by the same letter
	\[\eta\colon \bS\longrightarrow p_\sharp\bS^{-\T M}\]
	the map obtained by tensoring it with $\bS^{- V}$. Consider the Riemannian metric on $M$ induced by the embedding $M\subset V$ and the Pontryagin--Thom collapse along the diagonal
	\[\PT_\Delta\colon (M\times M, \id)_+\longrightarrow \Delta_\sharp S^{\T M}\] in $\Ho(\Top_{M\times M})$ induced by the diagram
	\begin{equation}
		\label{eqn:HWThomcollapse}
		(M \times M, \id)_+ \longrightarrow (M^I, \pi)_+ \wedge_{M \times M} (S^{\T M} \times M) \xleftarrow{\sim} \Delta_\sharp S^{\T M},
	\end{equation}
	where the maps are defined as follows. Choose a tubular neighborhood $U\subset M\times M$ of the diagonal which can be identified with the tangent bundle by sending a pair $(n, m)$ to the tangent vector (which we will denote by $m-n$) at $m$ of the unique geodesic $\omega(n, m)$ connecting $n$ and $m$. Given a point $(n, m)\in M\times M$ we send it to the pair $(\omega(n, m), m-n)$. Here we think of $S^{\T M}$ as the quotient of the $\epsilon$-disk bundle inside $\T M$ by its boundary, where $\epsilon$ is smaller than the injectivity radius, so that the formula is well-defined. The second map is given by the inclusion of constant paths and is a weak equivalence. Taking fiberwise suspension spectra we obtain a map
	\begin{equation}
		\PT_\Delta\colon \bS_{X\times X}\longrightarrow \Delta_\sharp \bS^{\T M}.
	\end{equation}
	We denote by the same letter
	\[\PT_\Delta\colon \bS_X\boxtimes \bS^{-\T M}\longrightarrow \Delta_\sharp\bS_X\]
	the tensor product of the previous map with $\bS_M\boxtimes \bS^{-\T M}$. The following is shown in \cite[Theorem 18.6.1]{MaySigurdsson}.

        From the description above we see that applying the flip map exchanging the two factors of $X$, we get the Pontryagin--Thom collapse map post-composed with the action of $(-1)$ along the fibers of $\T M$.

	\begin{prop}
		Consider the self-duality data of $\Sp^X$ from \cref{prop:Spselfdual}. Then the morphisms
		\[\PT_\Delta\colon \bS_X\boxtimes \bS^{-\T M}\longrightarrow \Delta_\sharp\bS_X=\coev(\bS),\qquad \eta\colon \bS\longrightarrow p_\sharp\bS^{-\T M}=\ev(\bS^{-\T M}, \bS_X)\]
		establish a relative duality between $\bS_X$ and $\bS^{-\T M}$. In particular, $\zeta_X \cong \bS^{-\T M}$.
		\label{prop:manifoldCWduality}
	\end{prop}

	The definition of the Pontryagin--Thom collapse along the diagonal \eqref{eqn:HWThomcollapse} does no longer depend on the embedding $M\subset V$, but a priori does depend on the Riemannian metric on $M$. It will be useful to present another equivalent presentation for the Pontryagin--Thom collapse which does not use a Riemannian metric.
	
	\begin{defn}
		Let $M$ be a manifold. The \defterm{Fulton--MacPherson compactification of the configuration space of two points} $\FM_2(M)$ is the real oriented blowup of $M\times M$ along the diagonal.
	\end{defn}
	
	Let $\UT M = (\T M\setminus M) / \bR_+$ be the unit tangent bundle. By construction $\FM_2(M)$ is a manifold with boundary $\UT M$ and we have a commutative diagram
	\[
	\xymatrix{
		\UT M \ar[r] \ar[d] & \FM_2(M) \ar[d] \\
		M \ar[r] & M\times M.
	}
	\]
	In particular, we obtain a morphism
	\begin{equation}
		\Delta_\sharp \C_M(M, \UT M)\longrightarrow \C_{M\times M}(M\times M, \FM_2(M))
		\label{eq:FM2map}
	\end{equation}
	of unreduced relative mapping cones.

	\begin{lm}
		Suppose
		\[
		\xymatrix{
			X \ar^{f}[r] \ar^{\pi}[d] & Y \ar^{\tilde{\pi}}[d] \\
			X' \ar^{f'}[r] & Y'
		}
		\]
		is a coCartesian diagram in $\cS$. Then
		\[Y'\coprod_X X'\longrightarrow Y'\coprod_Y Y'\]
		is an equivalence. In particular,
		\[f'_\sharp \C_{X'}(X', X)\longrightarrow \C_{Y'}(Y', Y)\]
		is an equivalence in $(\cS_{/Y'})_\ast$.
		\label{lm:pushoutequivalence}
	\end{lm}
	\begin{proof}
		The first statement follows since the right-hand side is the iterated pushout of
		\[
		\xymatrix{
			X \ar^{f}[r] \ar^{\pi}[d] & Y \ar^{\tilde{\pi}}[r] & Y' \\
			X',
		}
		\]
		and hence the pushout of the outer square.
		The second statement follows from the definitions $C_X(A,B) := X \coprod_B A$ and $f'_\sharp A := Y' \coprod_{X'} A$.
	\end{proof}
	
	\begin{cor}
		The map \eqref{eq:FM2map} is a weak equivalence.
	\end{cor}
	\begin{proof}
		The morphism $\UT M\rightarrow \FM_2(M)$ is an inclusion of the boundary; in particular, it is a $q$-cofibration. The diagram
		\[
		\begin{tikzcd}
			\UT M \ar[r] \ar[d] & \FM_2(M) \ar[d] \\ M \ar[r] & M \times M,
		\end{tikzcd}
		\]
		is a pushout. Since the $q$-model structure on $\Top$ is left proper, it is a homotopy pushout. The claim follows from \cref{lm:pushoutequivalence}.
	\end{proof}
	
	We can thus define another version of the Pontryagin--Thom collapse via the composite
	\begin{equation}
		\label{eqn:FMThomcollaps}
		(M \times M)_+ \to \C_{M\times M}(M \times M, \FM_2(M)) \xleftarrow{\sim}  \Delta_\sharp\C_M(M, \UT M)\cong \Delta_\sharp S^{\T M}.
	\end{equation}
	
	\begin{prop}
		\label{prop:FMequalHW}
		The induced maps \eqref{eqn:FMThomcollaps} and \eqref{eqn:HWThomcollapse} are homotopic. 
	\end{prop}
	\begin{proof}
		It suffices to show that the diagrams \eqref{eqn:HWThomcollapse} and \eqref{eqn:FMThomcollaps} fit into a commuting square. For that, note first that the first map in \eqref{eqn:HWThomcollapse}, when restricted to $\FM_2(M)$ admits a canonical null-homotopy. Recall that the map sends a pair of nearby points $(x,y)$ to the pair consisting of the connecting geodesic $\omega(x,y)$ and the connecting tangent vector (based at $x$). The assignment of the corresponding unit tangent vector extends continuously to $\FM_2(M)$. The required homotopy moves the point in $S^{\T M}$ along that unit tangent vector to the basepoint at $\infty$. We have thus constructed a commutative diagram
		\begin{equation}
			\begin{tikzcd}
				(M \times M)_+ \ar[r] \ar[dr] & (M^I,\pi)_+ \wedge_{M \times M} (S^{\T M} \times M) & \ar[l] \Delta_\sharp S^{\T M} \\
				&  \C_{M \times M}(M \times M, \FM_2(M)) \ar[u] & \ar[l] \ar[u] \Delta_\sharp \C_M(M, \UT M),
			\end{tikzcd}
			\label{eq:FMequalHW}
		\end{equation}
		where we noticed that the above defined homotopy, when restricted to the diagonal, merely identifies $\C_M(M, \UT M)$ with $S^{\T M}$.
	\end{proof}

\subsection{Pontryagin--Thom lift}
\label{sect:PTlift}

	Let $M$ be a closed manifold and $X$ its underlying homotopy type. In this section we construct a lift of the $\THH$ Euler characteristic $\chi_\THH(X)\in\Omega^\infty\Sigma^\infty_+ LX$ along the $\THH$ assembly, i.e. along the inclusion of constant loops
	\[\alpha\colon \Sigma^\infty_+ X\longrightarrow \Sigma^\infty_+ LX.\]

	Consider the map
	\[0_M\colon (M, \id)_+\rightarrow S^{\T M}\]
	of ex-spaces given by the zero section of the tangent bundle. Its suspension gives rise to a morphism
	\[e(M)\colon \bS_X\longrightarrow \bS^{\T M}\]
	which we will call the \defterm{Euler class}. Note that $\Hom_{\Sp^X}(\bS_X, \bS^{\T M})\cong \Omega^\infty\Sigma^\infty_+ X$ and we may equivalently consider $e(M)\in \Omega^\infty\Sigma^\infty_+ X$.
	
    Consider the Fulton--MacPherson compactification $\FM_2(M)$ of the configuration space of two points on $M$. It fits into a pushout diagram
		\begin{equation}
		\xymatrix{
			\UT M\ar[r] \ar[d] & \FM_2(M) \ar[d] \\
			M \ar^{\Delta}[r] & M\times M.
		}
            \label{eq:FM2diagram}
		\end{equation}
		The unreduced relative mapping cone of the left vertical map gives the morphism
		\[0_M\colon (M, \id)_+\rightarrow S^{\T M}\]
		on ex-spaces. Taking relative suspensions, we obtain a commutative diagram
            \[
		\xymatrix{
			\Delta_\sharp \bS_X \ar[r] \ar_{\Delta_\sharp e(M)}[d] & \bS_{X\times X} \ar[dl] \\
			\Delta_\sharp \bS^{\T M} &
		}
		\]
		of parameterized spectra over $X\times X$, where the diagonal map $\bS_{X\times X}\rightarrow \Delta_\sharp \bS^{\T M}$ is the Pontryagin--Thom collapse map along the diagonal by \cref{prop:FMequalHW}. Twisting by $\bS^{-\T M}$, we obtain a lift diagram
            \begin{equation}
		\xymatrix{
			\Delta_\sharp \bS^{-\T M} \ar[r] \ar_{\Delta_\sharp e(M)}[d] & \bS_X\boxtimes \bS^{-\T M} \ar[dl] \\
			\Delta_\sharp \bS_X &
		}
            \label{eq:FM2triangle}
		\end{equation}
            where we note that by \cref{prop:manifoldCWduality} the diagonal map is nondegenerate. Thus, by \cref{thm:liftdiagrams} from this lift diagram we obtain a lift (the \defterm{Pontryagin--Thom lift})
		\[\lambda_{\PT}(M)\in\Lift(\chi_{\THH}(X), \alpha)\]
		of $[\bS_X]\in\Omega^\infty \Sigma^\infty_+ LX$ along the assembly map.

    \begin{remark}
    The composite $X\rightarrow LX\rightarrow X$ is the identity, where the first map is the inclusion of constant loops and the second map is the evaluation at the basepoint of $S^1$. Therefore, the image of $[\bS_X]$ under the projection $LX\rightarrow X$ coincides with the Euler class $e(M)\in \Omega^\infty\Sigma^\infty_+ X$.
    \end{remark}

	The Dennis trace map from $K$-theory to $\THH$ gives a commutative diagram of assembly maps
	\[
	\xymatrix{
		\Sigma^\infty_+ X\otimes \bA(\pt) \ar^-{\alpha}[r] \ar^{\tr}[d] & \bA(X) \ar^{\tr}[d] \\
		\Sigma^\infty_+ X \ar^{\alpha}[r] & \Sigma^\infty_+ LX
	}
	\]
	
	As $\tr(\chi_A(X)) = \chi_{\THH}(X)$, the above commutative diagram induces a map
	\[\tr\colon \Lift(\chi_A(X), \alpha)\longrightarrow \Lift(\chi_{\THH}(X), \alpha).\]

    \begin{thm}
    Let $M$ be a closed manifold with underlying homotopy type $X=\Sing(M)$. Then the lifts $\tr(\lambda_{\Wh}(M))$ and $\lambda_{\PT}(M)$ of the $\THH$ Euler characteristic $\chi_{\THH}(X)$ are equivalent.
    \label{thm:florian}
    \end{thm}

\subsection{Proof of \texorpdfstring{\cref{thm:florian}}{theorem \ref{thm:florian}}}

Fix a triangulation of $M$ with face poset $T$. Denote by $T^\delta$ the same set with the trivial poset structure and let $\pi\colon T^\delta\rightarrow T^{\op}$ be the functor given by the identity on objects. Let $t\colon T^{\op}\rightarrow X$ be the localization functor from \eqref{eq:Tlocalization} and $M_\tau=\ost(\tau)$ the open star of the simplex $\tau\in T$.

Let us recall that the Whitehead lift is obtained by traversing the diagram
\[
\begin{tikzcd}
    & \bS \ar[d, "{[\bS_{T^\op}]}"] \ar[dr, "{[\bS_X]}"] \\
    \bA(\Sp^{T^\delta, \omega}) \ar[r, "\sim" below, "\pi_\sharp" above] & \bA(\Sp^{T^{\op}, \omega}) \ar[r, "t_\sharp"]  & \bA(\Sp^{X, \omega}) \\
    \Sigma^\infty_+ T^\delta \otimes \bA(\pt) \ar[u, "\alpha", "\sim"'] \ar[rr] && \Sigma^\infty_+ X \otimes \bA(\pt), \ar[u, "\alpha"]
\end{tikzcd}
\]
where we used that the natural morphism $t_\sharp \bS_{T^{\op}}\rightarrow \bS_X$ is an equivalence. By naturality of the Dennis trace map it follows that the trace of the Whitehead lift is obtained from the analog diagram where $\bA$ is replaced with $\THH$. Using the definitions from \cref{sec:liftdiagrams} the diagram can be extended as follows:


\[
\begin{tikzcd}
    & \bS \ar[d, "{[\bS_T]}"] \ar[dr, "{[\bS_X]}"] \\
    \THH(\Sp^{T^\delta, \omega}) \ar[r, "\sim" below, "\pi_\sharp" above] \ar[d, "\sim" right, "\beta_{T^\delta}" left] & \THH(\Sp^{T^{\op}, \omega}) \ar[r, "t_\sharp"] \ar[d, "\beta_T" left] & \THH(\Sp^{X, \omega}) \ar[d, "\sim" right, "\beta_X" left] \\
    \tr(\cT_{T^\delta}) \ar[r] & \tr(\cT_{T^\op}) \ar[r] & \tr(\cT_X) \\
    \Sigma^\infty T^\delta \ar[r] \ar[u, "\sim" right, "\alpha_{T^\delta}" left] & \Sigma^\infty |T^\op| \ar[u, "\sim" right, "\alpha_T" left] \ar[r, "\sim"] & \Sigma^\infty X. \ar[u, "\alpha_X" left]
\end{tikzcd}
\]

In other words, the Whitehead lift $\tr(\lambda_{\Wh}(M))$ is the unique lift lying in the image of
\[t_\sharp\colon \Lift(\beta_{T^{\op}}(\chi_{\THH}(T^{\op})), \alpha_{T^{\op}})\longrightarrow \Lift(\beta_X(\chi_{\THH}(X)), \alpha_X) \cong \Lift(\chi_{\THH}(X), \alpha),\]
where the set $\Lift(\beta_{T^{\op}}(\chi_{\THH}(T^{\op})), \alpha_{T^{\op}})$ consists of a single element as $\alpha_T$ is an isomorphism.

To show that $\tr(\lambda_{\Wh}(M))=\lambda_{\PT}(M)$, it suffices to show that the Pontryagin--Thom lift is compatible with the triangulation, i.e. it also lies in the image of $t_\sharp$. For this, recall from \cref{thm:faceposetduality} that $T^{\op}$ admits duality, i.e. $\Delta_\sharp p^*\colon \Sp\rightarrow \Sp^{T^{\op}}\otimes \Sp^{T^{\op}}$ defines the coevaluation of a self-duality of $\Sp^{T^{\op}}$. In particular, we may identify elements of $\Lift(\beta_{T^{\op}}(\chi_{\THH}(T^{\op})), \alpha_{T^{\op}})$ with lift diagrams using \cref{thm:liftdiagrams}. By \cref{thm:faceposetduality} the relative dual of $\bS_T$ is locally constant. To construct a model of the relative dual, we will use the following observation.

\begin{lm}
Consider the self-dualities of $\Sp^{T^{\op}}$ and $\Sp^X$ from \cref{def:admitduality}. Consider an object $\zeta_X\in\Sp^X$ together with a morphism
\[\epsilon\colon \bS_{T^{\op}}\boxtimes t^*\zeta_X\longrightarrow \Delta^{T^{\op}}_\sharp \bS_{T^{\op}}.\]
Then $\epsilon$ exhibits $t^*\zeta_X$ as the relative dual of $\bS_{T^{\op}}$ if, and only if,
\[(t_\sharp\boxtimes t_\sharp)(\epsilon)\colon \bS_X\boxtimes \zeta_X\longrightarrow \Delta^X_\sharp \bS_X\]
exhibits $\zeta_X$ as the relative dual of $\bS_X$.
\label{lm:constructibleevaluation}
\end{lm}
\begin{proof}
The condition that $\epsilon$ exhibits $t^*\zeta_X$ as the relative dual of $\bS_{T^{\op}}$ is that the induced morphism
\[\epsilon'\colon t^*\zeta_X\longrightarrow (p_* t_\sharp\boxtimes \id) \Delta^T_\sharp\bS_{T^{\op}}\cong (p_*\boxtimes t^*)\Delta^X_\sharp\bS_X\]
is an isomorphism. Similarly, the condition that $(t_\sharp\boxtimes t_\sharp)(\epsilon)$ exhibits $\zeta_X$ as the relative dual of $\bS_X$ is that the induced morphism
\[t_\sharp \epsilon'\colon \zeta_X\longrightarrow (p_*\boxtimes \id) \Delta^X_\sharp \bS_X\]
is an isomorphism. But the two conditions are equivalent since $t^*\colon \Sp^X\rightarrow \Sp^{T^{\op}}$ is fully faithful.
\end{proof}

Let us now construct a version of the Pontryagin--Thom lift compatible with the cover $\{ M_\tau \}_{\tau \in T^\op}$. Consider a morphism $f\colon E\rightarrow M\times M$ in $\Top$. Then $\{E|_{M_\tau\times M_\mu}\}$ defines an open cover $E$, where $E|_{M\tau\times M_\mu}\subset E$ denotes the strict inverse image of $M_\tau\times M_\mu\subset M\times M$. Consider the functor
\[(\tau, \mu\mapsto \Sigma^\infty_+ E|_{M_\tau\times M_\mu})\in\Sp^{T^{\op} \times T^{\op}}.\]

\begin{lm}
There is a natural equivalence
\[(t_\sharp\boxtimes t_\sharp)((\tau, \mu)\mapsto \Sigma^\infty_+ E|_{M_\tau\times M_\mu})\cong \Sigma^\infty_{+(M\times M)}E\]
in $\Sp^{M\times M}$.
\label{lm:constructiblesuspension}
\end{lm}
\begin{proof}
Consider the functor $T^{\op}\times T^{\op}\rightarrow \cU(E)$ to the category of open subsets of $E$ given by $\tau,\mu\mapsto E|_{M_\tau\times M_\mu}$. As in \cref{lm:openstargoodcover} we get an equivalence
\[\colim_{\tau,\mu}(\Sing(E|_{M_\tau\times M_\mu})\rightarrow \Sing(M_\tau\times M_\mu))\cong (\Sing(E)\rightarrow \Sing(M\times M))\]
in $\cS$. Denote by $i^E_{\tau\mu}\colon \Sing(E|_{M_\tau\times M_\mu})\rightarrow \Sing(E)$ and $i^M_{\tau\mu}\colon \Sing(M_\tau\times M_\mu)\rightarrow \Sing(M\times M)$ the natural inclusions. For any $x,y\in M$ we have to show that the natural morphism
\[\colim_{\tau,\mu\in T^{\op}}((i^M_{\tau\mu})_\sharp (f|_{M_\tau\times M_\mu})_\sharp(i^E_{\tau\mu})^*\rightarrow f_\sharp\]
is an isomorphism. By functoriality this reduces to showing that
\[\colim_{\tau,\mu\in T^{\op}} (i^E_{\tau\mu})_\sharp(i^E_{\tau\mu})^*\rightarrow \id\]
is an isomorphism which is the content of \cref{prop:colimitapproximation}.
\end{proof}

Let us apply the above paradigm to \eqref{eq:FM2diagram}. Namely, restricting that diagram to $M_\tau\times M_\mu$ we obtain
\begin{equation}
    \label{eqn:diagram repeated in proof}    
    \begin{tikzcd}
        \UT(M_\tau \cap M_\mu) \ar[r] \ar[d] & \FM_2|_{M_\tau \times M_\mu} \ar[d] \\
        M_\tau \cap M_\mu \ar[r] & M_\tau \times M_\mu
    \end{tikzcd}
\end{equation}
which is a pushout diagram for every $(\tau,\mu) \in T^\op \times T^\op$. Taking relative suspension spectra we obtain a commutative diagram
\[
\xymatrix{
\Delta_\sharp \bS_{T^{\op}} \ar[r] \ar_{\Delta_\sharp e(T)}[d] & \bS_{T^{\op}\times T^{\op}} \ar[dl] \\
\Delta_\sharp \bS^{\T M_\bullet} &
}
\]
in $\Sp^{T^{\op}\times T^{\op}}$, where $\bS^{\T M_\bullet}\in\Sp^{T^{\op}}$ is the functor $\tau\mapsto \bS^{\T M_\tau}$. By \cref{thm:faceposetduality} this spectrum is locally constant and invertible. Therefore, twisting by its inverse we obtain a commutative diagram
\begin{equation}
\xymatrix{
\Delta_\sharp \bS^{-\T M_\bullet} \ar[r] \ar_{\Delta_\sharp e(T)}[d] & \bS_{T^{\op}}\boxtimes \bS^{-\T M_\bullet} \ar[dl] \\
\Delta_\sharp \bS_{T^{\op}} &
}
\label{eq:constructibleFM2triangle}
\end{equation}

By \cref{lm:constructiblesuspension} applying $t_\sharp\boxtimes t_\sharp$ we obtain the commutative diagram \eqref{eq:FM2triangle}. In particular, by \cref{lm:constructibleevaluation} the morphism $\bS_{T^{\op}}\boxtimes \bS^{-\T M_\bullet}\rightarrow \Delta_\sharp \bS_{T^{\op}}$ exhibits $\bS^{-\T M_\bullet}$ as the relative dual of $\bS_{T^{\op}}$ in $\Sp^{T^{\op}}$.

Therefore, \eqref{eq:constructibleFM2triangle} is a lift diagram, i.e. it defines an element of $\Lift(\beta_{T^{\op}}(\chi_{\THH}(T^{\op})), \alpha_{T^{\op}})$ which is sent to $\lambda_{\PT}(M)\in \Lift(\beta_X(\chi_{\THH}(X)), \alpha_X)$ under $t_\sharp$. As $\Lift(\beta_{T^{\op}}(\chi_{\THH}(T^{\op})), \alpha_{T^{\op}})$ consists of a single element, this finishes the proof that $\tr(\lambda_{\Wh}(M))=\lambda_{\PT}(M)$.

\section{String topology}\label{sect:stringtopology}

In this section we recall the string topology operations on the loop space homology of a manifold and study their homotopy invariance. We first introduce a relative intersection product that depends on a lift of the $\THH$ Euler characteristic. We then define the string topology operations using this relative intersection product and show that they agree with the definition given in \cite{HingstonWahl2} for the Pontryagin--Thom lift. The study of the homotopy invariance is thus reduced to the same question about the relative intersection product, which in turn is reduced to the invariance of lifts of the $\THH$ Euler characteristic, so we can apply the results of the previous section.

For concreteness, we replace the $\infty$-category of parametrized spectra $\Sp^X=\Fun(X, \Sp)$ by the $\infty$-category $\LocSys(X) = \Fun(X, \Mod_\Z)$ of local systems of chain complexes of abelian groups. We denote by $\C_\bullet(-)=\C_\bullet(-;\Z)$ the singular chain complex with integral coefficients and $\H_\bullet(-)=\H_\bullet(-;\Z)$ the integral homology. Given $f \colon X \to Y$ we denote by $\H_\bullet(Y,X) = \H_\bullet(Y,X;\Z)$ the homology of the cone of $\C_\bullet(X) \to \C_\bullet(Y)$ even if $f$ is not a inclusion of a subspace. Note that the content in this section works verbatim when $\Z$ is replaced by another $\bE_\infty$-ring $k$, such that $M$ is oriented over $k$, i.e. one fixes an equivalence $\zeta_M \to k[-d]$.

\subsection{Relative intersection product}
\label{sect:relativeintersection}

In this preliminary section we define the notion of a relative intersection product on a space equipped with a lift of the $\THH$ Euler characteristic. Let us briefly introduce the situation we wish to generalize. Recall the intersection product
\[
\H_{\bullet+d}(M \times M) \to \H_\bullet(M)
\]
for an oriented $d$-manifold $M$ defined either by using Poincar\'{e} duality or by geometrically intersecting transverse chains. Similarly, given a local system $\cE$ over $M \times M$, one can define the intersection product with twisted coefficients
\[
\H_{\bullet+d}(M \times M; \cE) \to \H_\bullet(M; \Delta^* \cE).
\]
Our goal is to generalize it to a relative setting as follows. Suppose we are given a local system $\cF$ over $M$ together with a map $\Delta_\sharp \cF \to \cE$ or, equivalently, a map $\cF\to \Delta^* \cE$. Given a lift of the $\HH$ Euler characteristic we are going to refine the intersection product to a map
\[
\H_{\bullet+d}(M \times M; \cE/\Delta_\sharp\cF) \to \H_\bullet(M; \Delta^* \cE/\cF).
\]

\ 

Let $M$ be a closed oriented $d$-manifold and consider the $\infty$-category $\LocSys(M) = \Fun(M, \Mod_\Z)$ with its self-duality data given by \cref{prop:Spselfdual}. Denote by $\Z_M\in\LocSys(M)$ the constant local system which is compact by \cref{prop:finitetypeproper}. By \cref{prop:manifoldCWduality} its relative dual is $\Z_M[-d]$ (using the orientation) with the evaluation $\epsilon_p\colon \Z_{M\times M}[-d]\rightarrow \Delta_\sharp \Z_M$ given by the Pontryagin--Thom collapse $\PT_\Delta$ along the diagonal. Let $p_E\colon E\rightarrow M\times M$ be a Hurewicz fibration and $\cE=(p_E)_\sharp \Z_E\in\LocSys(M\times M)$ the corresponding local system over $M\times M$.

\begin{defn}
The \defterm{intersection product} is the composite
\[\cE[-d]\xrightarrow{\epsilon_p} \cE\otimes \Delta_\sharp \bS_M\cong \Delta_\sharp\Delta^*\cE.\]
Taking homology, we obtain the map
\[\H_{\bullet+d}(E)\longrightarrow \H_\bullet(E\times_{M\times M} M).\]
\label{def:intersection}
\end{defn}

Let $\sigma\colon M\times M\cong M\times M$ be the isomorphism exchanging the two factors.

\begin{lm}\label{lm:intersectionsymmetry}
Let $p_E\colon E\rightarrow M\times M$ be a Hurewicz fibration and consider the space $E'=E$ equipped with the map $E\xrightarrow{p_E} M\times M\xrightarrow{\sigma} M\times M$. Then the intersection products
\[\Int^E, \Int^{E'}\colon \H_{\bullet+d}(E)\longrightarrow \H_\bullet(E\times_{M\times M} M)\]
with respect to $E$ and $E'$ satisfy $\Int^E = (-1)^d \Int^{E'}$.
\end{lm}
\begin{proof}
From the description of the Pontryagin--Thom collapse map $\PT_\Delta\colon \Z_{M\times M}\rightarrow \Delta_\sharp \Z^{\T M}$ given in \cref{sect:manifoldPoincareduality}, we see that
\[\sigma^* \PT_\Delta = \mathrm{inv}\circ \PT_\Delta,\]
where $\mathrm{inv}\colon \Z^{\T M}\rightarrow \Z^{\T M}$ is given by acting by $(-1)$ on the tangent bundle. The map $\epsilon_p\colon \Z_{M\times M}[-d]\rightarrow \Delta_\sharp \Z_M$ is a composite of the Pontryagin--Thom collapse map with the trivialization $\Z^{\T M}\cong \Z_M[-d]$. Under this isomorphism the map $\mathrm{inv}$ becomes the multiplication by $(-1)^d$.
\end{proof}

By \cref{prop:LocSysTHH} we may identify $\dim(\LocSys(M))\cong \C_\bullet(LM)$. We refer to the class $[\Z_M]\in\C_\bullet(LM)$ as the \defterm{$\HH$ Euler characteristic}. To define the relative intersection product, suppose we are given its lift along the assembly map, i.e. along the inclusion of constant loops $\alpha\colon \C_\bullet(M)\rightarrow \C_\bullet(LM)$. By \cref{thm:liftdiagrams} we get a commutative diagram
\begin{equation}
	\label{eqn:echarliftdiag}
	\begin{tikzcd}
		\Delta_\sharp \Z_M[-d] \ar[r] \ar[d, "\Delta_\sharp e(M)"] & \Z_{M\times M}[-d] \ar[d, "\epsilon_p"] \\
		\Delta_\sharp \Z_M \ar[r, equal] & \Delta_\sharp \Z_M,
	\end{tikzcd}
\end{equation}
where $e(M)\in\C^d(M)$ is the Euler class.

Consider a map $p_F\colon F\rightarrow M$ which fits into a commutative square
\begin{equation}\label{eq:FEdiagram}
\xymatrix{
F \ar^{a}[r] \ar^{p_F}[d] & E \ar^{p_E}[d] \\
M \ar[r] & M\times M
}
\end{equation}
Let $\cF=(p_F)_\sharp \Z_F\in\LocSys(M)$ be the corresponding local system over $M$ together with the map $a\colon \Delta_\sharp\cF\rightarrow \cE$ and, by adjunction, $\bar{a}\colon \cF\rightarrow \Delta^*\cE$.

Consider a commutative diagram
\begin{equation}
\label{eqn:relintersectionpreliminary}
\xymatrix@C=1.5cm{
\Delta_\sharp \cF[-d] \ar[r] \ar^{\Delta_\sharp(\id\otimes e(M))}[d] & \Delta_\sharp\Delta^*\Delta_\sharp \cF[-d] \ar[r] \ar^{\Delta_\sharp(\id\otimes e(M))}[d] & \Delta_\sharp \cF[-d] \ar^{a}[r] \ar^{\epsilon_p}[d] & \cE[-d] \ar^{\epsilon_p}[d] \\
\Delta_\sharp \cF \ar[r] & \Delta_\sharp \Delta^*\Delta_\sharp \cF \ar@{=}[r] & \Delta_\sharp\Delta^*\Delta_\sharp \cF \ar^-{\Delta_\sharp\Delta^*(a)}[r] & \Delta_\sharp \Delta^*\cE,
}
\end{equation}
where the middle square is obtained by tensoring \eqref{eqn:echarliftdiag} with $\Delta_\sharp \cF$ on the left. This gives rise to a commuting square
\begin{equation}
	\xymatrix@C=1.5cm{
		\Delta_\sharp \cF[-d] \ar^{a}[r] \ar^{\Delta_\sharp(\id\otimes e(M))}[d] & \cE[-d] \ar^{\epsilon_p}[d] \\
		\Delta_\sharp\cF \ar_{\Delta_\sharp(\bar{a})}[r] & \Delta_\sharp\Delta^*\cE.
	}
	\label{eqn:relintersection}
\end{equation}

\begin{remark}
On the level of singular chains the relative intersection square \eqref{eqn:relintersection} is
\begin{equation}
\xymatrix{
\C_\bullet(F)[-d] \ar^{e(M)}[d] \ar^{a}[r] & \C_\bullet(E)[-d] \ar[d] \\
\C_\bullet(F) \ar^-{\bar{a}}[r] & \C_\bullet(E\times_{M\times M} M)
}
\label{eq:relintersectionhomology}
\end{equation}
\end{remark}

\begin{remark}
\label{rem:relintersectionastensorproduct}
The above construction can also be understood as follows. Since $\Delta_\sharp$ is an oplax symmetric monoidal functor, the category of diagrams $\Delta_\sharp \cF \to \cE$, equivalently $\cF \to \Delta^*\cE$, has a pointwise symmetric monoidal structure. The diagram \eqref{eqn:echarliftdiag} represents a morphism
\[(\Delta_\sharp \Z_M[-d]\to\Z_{M\times M}[-d])\longrightarrow (\Delta_\sharp \Z_M\xrightarrow{\id} \Delta_\sharp \Z_M).\]
Diagram \eqref{eqn:relintersection} is then obtained by tensoring with that morphism and identifying the functor
\[
(\Delta_\sharp \cF \to \cE) \mapsto (\Delta_\sharp \cF \to \cE) \otimes (\Delta_\sharp \Z_M \xrightarrow{\id} \Delta_\sharp \Z_M)
\]
with
\[
(\Delta_\sharp \cF \to \cE) \mapsto (\Delta_\sharp \cF \to \Delta_\sharp\Delta^*\cE).
\]
\end{remark}

\begin{defn}
\label{def:relint}
Let $M$ be a closed oriented manifold equipped with a lift of the $\HH$ Euler characteristic. The \defterm{relative intersection product} is the map
\[
(\cE/\Delta_\sharp\cF)[-d]\longrightarrow \Delta_\sharp(\Delta^*\cE/\cF)
\]
induced by diagram \eqref{eqn:relintersection}. Taking homology we obtain the map
\[\H_{\bullet+d}(E, F)\longrightarrow \H_\bullet(E\times_{M\times M} M, F)\]
which we also call the relative intersection product.
\end{defn}

\begin{remark}
Consider the case $\cF=0$. Then the relative intersection product reduces to the usual intersection product $\cE[-d]\rightarrow \Delta_\sharp\Delta^*\cE$ for which no lift of the $\HH$ Euler characteristic is required. 
\end{remark}

\begin{lm}\label{lm:relativeintersectionsymmetry}
Consider the diagram \eqref{eq:FEdiagram}, where $p_E\colon E\rightarrow M\times M$ is a Hurewicz fibration. Let $E'\rightarrow M\times M$ be the Hurewicz fibration with the two factors of $M$ exchanged, as in \cref{lm:intersectionsymmetry}. Choose the Pontryagin--Thom lift $\lambda_{\PT}(M)$ of the $\HH$ Euler characteristic. Then the relative intersection products
\[\Int^{E, F}, \Int^{E', F}\colon \H_{\bullet+d}(E, F)\longrightarrow \H_\bullet(E\times_{M\times M} M, F)\]
with respect to the pairs $(E, F)$ and $(E', F)$ satisfy $\Int^{E, F} = (-1)^d \Int^{E', F}$.
\end{lm}
\begin{proof}
The diagram \eqref{eq:FM2diagram} is invariant under the exchange of the two factors of $M$, if we further compose with the isomorphism $\UT M\rightarrow \UT M$ given by acting by $(-1)$ along the fibers. Therefore, the diagram \eqref{eqn:echarliftdiag} used in the definition of the relative intersection product is invariant under the change of the factors of $M$ if we replace $\epsilon_p$ by $(-1)^d\epsilon_p$ and $e(M)$ by $(-1)^d e(M)$.
\end{proof}

Choose the Pontryagin--Thom lift $\lambda_{\PT}(M)$ of the $\HH$ Euler characteristic. We will now give a concrete formula for the relative intersection product. Denote by $\Th_M\in\H^d(S^{\T M}, M)$ the Thom class and consider an embedding $M\subset V$ into a Euclidean space which induces a Riemannian metric on $M$. Let $\cM(-)$ be the Moore path space. Since $p_E \colon E\rightarrow M\times M$ is a Hurewicz fibration, we may choose a path-lifting function (also known as a \emph{Hurewicz connection}), i.e. a section of $\cM(E)\rightarrow \cM(M \times M) \times_{M\times M} E$. Evaluating at the endpoints we obtain the transport function which we denote by
\begin{align*}
    \cM(M \times M) \times_{M \times M} E &\to E \\
    ((p_1 \times p_2), e) &\mapsto (p_1 \times p_2).e.
\end{align*}

Recall the definition of the geometrically defined Pontryagin--Thom collapse map \eqref{eqn:HWThomcollapse},
\begin{align*}
(M \times M, \id)_+ &\to (M^I, \pi)_+ \wedge_{M \times M} (S^{\T M} \times M) \\
(x,y) &\mapsto (\omega(x,y), y-x),
\end{align*}
where $\omega(x,y)$ is the geodesic connecting $x$ to $y$ and $y-x \in T_xM$ is the corresponding tangent vector. Note that the same formula lifts to Moore paths, which is what we will use below.

With that notation we can now define the morphism
\begin{align}
	R \colon E &\longrightarrow (E\times_{M\times M} M)^{\T M} \label{eq:R} \\
	e &\longmapsto (\const_x \times \omega(x,y)^{-1}).e \wedge (y-x), \nonumber
\end{align}
where $(x,y) = p_E(e)$.

\begin{prop}
Suppose $M$ is a closed oriented manifold equipped with the Pontryagin--Thom lift of the $\HH$ Euler characteristic. The relative intersection product is given by
\[\H_\bullet(E, F)\xrightarrow{(R, 0_M)} \H_\bullet((E\times_{M\times M} M)^{\T M}, F^{\T M})\xrightarrow{\cap \Th_M} \H_{\bullet-d}(E\times_{M\times M} M, F).\]
\label{prop:HWrelint}
\end{prop}
\begin{proof}
By \cref{prop:FMequalHW} (and its proof) the square \eqref{eqn:echarliftdiag} has a lift to ex-spaces over $M \times M$ and is equivalent to
\[
\begin{tikzcd}
    \Delta_\sharp(M, \id)_+ \ar[r] \ar[d, "0_M"] & (M \times M, \id)_+ \ar[d] \\
    \Delta_\sharp(S^{\T M}) \ar[r, "\sim"] & \cM(M)_+ \wedge_{M \times M} (S^{\T M} \times M).
\end{tikzcd}
\]
    This allows us to write \eqref{eqn:relintersection} (again in ex-spaces over $M\times M$) as
	\[
        \xymatrix{
		\Delta_\sharp (F, p_F)_+ \ar[r] \ar^{0_M}[dd] & (E, p_E)_+ \ar[d] \\
            & (E, p_E)_+ \wedge_{M \times M} \cM(M)_+ \wedge_{M \times M} (S^{\T M}\times M) \\
            \Delta_\sharp ((F, p_F)_+ \wedge_M S^{\T M}) \ar[r] & (E, p_E)_+ \wedge_{M \times M} \Delta_\sharp S^{\T M} \cong \Delta_\sharp( \Delta^* (E, p_E)_+ \wedge_M S^{\T M}) \ar_{\sim}[u]
	}
        \]
        A retraction to the wrong-way map is given using the path-lifting function for $E$ by the formula
        \[
        (e, p, u) \mapsto ((\const_x \times p^{-1}).e, u).
        \]
        With that the vertical composition on the right becomes exactly $R$.
\end{proof}

\Cref{prop:HWrelint} may be applied to $F=\varnothing$ to obtain the result about the usual intersection product.

\begin{cor}
The intersection product is given by
\[\H_\bullet(E)\xrightarrow{R} \H_\bullet((E\times_{M\times M} M)^{\T M})\xrightarrow{\cap \Th_M} \H_{\bullet-d}(E\times_{M\times M} M).\]
\label{cor:HWint}
\end{cor}

\subsection{String topology operations}

We begin with the definition of the loop product as in \cite{CohenJones}. Consider the path space $M^I$ of continuous maps $[0, 1]\rightarrow M$. By \cite[Lemma 3.4]{Cole} the projection $M^I\rightarrow M\times M$ to the endpoints is a Hurewicz fibration.

Let $LM$ be the space of continuous maps $S^1\rightarrow M$ with $\ev\colon LM\rightarrow M$ the evaluation at the basepoint $0\in S^1$. As $LM=M^I\times_{M\times M} M$, $\ev\colon LM\rightarrow M$ is a Hurewicz fibration. Concatenation of loops defines a morphism $LM\times_M LM\rightarrow LM$. It is a morphism of spaces over $M$, so on the level of corresponding local systems it gives rise to a morphism
\[m\colon \ev_\sharp\Z_{LM}\otimes \ev_\sharp\Z_{LM}\longrightarrow \ev_\sharp\Z_{LM}.\]

\begin{defn}
The \defterm{loop product} is given by the composite
\[\wedge\colon \H_\bullet(LM)\otimes \H_\bullet(LM)\longrightarrow \H_{\bullet-d}(LM\times_M LM)\xrightarrow{m} \H_{\bullet-d}(LM),\]
where the first morphism is given by the intersection product as in \cref{def:intersection}.
\end{defn}

\begin{lm}
The loop product is graded commutative: for $\alpha\in\H_n(LM)$ and $\beta\in\H_m(LM)$ we have
\[\alpha\wedge \beta = (-1)^{n+m+d}\beta\wedge \alpha.\]
\end{lm}
\begin{proof}
The loop product is given by the composite
\[
\H_n(LM)\otimes \H_m(LM)\rightarrow \H_{n+m}(LM\times LM)\rightarrow \H_{n+m-d}(LM\times_M LM)\xrightarrow{m} \H_{n+m-d}(LM).
\]
The first map given by the K\"unneth isomorphism introduces a sign $(-1)^{n+m}$ when we exchange $\alpha$ and $\beta$. The second map given by the intersection product introduces a sign $(-1)^d$ when we exchange the two $LM$ factors by \cref{lm:intersectionsymmetry}. Finally, the maps $LM\times_M LM\rightarrow LM$ given by $(\gamma_1, \gamma_2)\mapsto \gamma_1\ast \gamma_2$ and $(\gamma_1, \gamma_2)\mapsto \gamma_2 \ast \gamma_1 = \gamma_2\ast (\gamma_1\ast \gamma_2)\ast \gamma_2^{-1}$ are homotopic, so they induce equal maps on homology.
\end{proof}

We now give a description of the loop product that matches the one given in \cite[Section 1.4]{HingstonWahl1} for the Chas--Sullivan product. Consider the Hurewicz fibration $\ev\colon LM\rightarrow M$ given by evaluating at $0$ and let
\[\ev\times\ev\colon E=LM\times LM\rightarrow M\times M.\]
Let us briefly recall the map $R$ from \eqref{eq:R}. We first choose the path-lifting function for the fibration $LM\rightarrow M$ given by conjugating a loop in $LM$ with a given Moore path. The corresponding transport function is
\begin{align*}
    \cM(M) \times_{M \times M} LM &\to LM \\
    (\omega, \gamma) &\mapsto \omega^{-1} \star \gamma \star \omega.
\end{align*}
We choose the path-lifting function for $E\rightarrow M\times M$ to be the product of the above. We thus obtain the map
\[
R_{CS} \colon LM \times LM \to (LM\times_M LM)^{\T M},
\]
which can be described as follows. Fix a tubular neighborhood of the diagonal $M \subset U \subset M \times M$. Let $(\gamma_1, \gamma_2)$ be an element in $LM \times LM$ with $\gamma_1(0) = x$ and $\gamma_2(0) = y$. If $(x,y) \not\in U$, then send it to the base point. Otherwise, we send it to the pair consisting of $(\gamma_1, \omega(x,y) \star \gamma_2 \star \omega(x,y)^{-1}) \in LM \times_M LM$ and $(x,y) \in U/\partial U \cong S^{\T M}$, where $\omega(x,y)$ is the geodesic connecting $x$ to $y$.

\begin{prop}\label{prop:stringproductintersection}
The loop product is given by the composite
\[
\H_\bullet(LM) \otimes \H_\bullet(LM) \xrightarrow{R_{CS}} \H_\bullet((LM \times_M LM)^{\T M}) \xrightarrow{\cap \Th_M} \H_{\bullet - d}(LM \times_M LM) \xrightarrow{m} \H_{\bullet-d}(LM).
\]
\end{prop}

Next, let us recall the definition of the loop coproduct (also known as the Goresky--Hingston coproduct). As before, $M$ is a closed oriented manifold. Consider the space
\[\FE = LM\times_M LM\longrightarrow M\]
of maps $S^1\vee S^1\rightarrow M$ (the space of figure eights). We have a natural morphism $\iota_0\cup \iota_1\colon LM\sqcup LM\rightarrow \FE$ given by figure eights which are constant on each half. The natural projection $\FE\rightarrow LM\times LM$ fits into a commutative diagram
\begin{equation}
	\xymatrix@C=3cm{
		\FE \ar[r] & LM\times LM \\
		LM\sqcup LM \ar^-{\id\times \ev\cup \ev\times \id}[r] \ar^{\iota_0\cup \iota_1}[u] & M\times LM\cup LM\times M \ar^{i\times \id\cup \id\times i}[u]
	}
	\label{eq:F8forgetbasepoint}
\end{equation}

For $t\in[0, 1]$ denote by $\ev_t\colon LM\rightarrow M$ the evaluation map at time $t$. For instance, $\ev_0=\ev_1$ is the map $\ev$ considered before. Consider the reparametrization map $J_s\colon LM\rightarrow LM$ for $s\in [0, 1]$ defined by $J_s(\gamma) = \gamma\circ \theta_{\frac{1}{2}\rightarrow s}$, where $\theta_{\frac{1}{2}\rightarrow s}$ is the piecewise-linear map that fixes $0$ and $1$ and takes $\frac{1}{2}$ to $s$.

Consider the commutative diagram
\begin{equation}
	\xymatrix{
		LM\sqcup LM \ar^-{J_0\cup J_1}[r] \ar[d] & LM \ar^{(\ev_0, \ev_{1/2})}[d] \\
		M \ar^-{\Delta}[r] & M\times M
	}
	\label{eq:F8intersection}
\end{equation}
The map $(\ev_0, \ev_{1/2})\colon LM\rightarrow M\times M$ is again a Hurewicz fibration: indeed, $LM$ is homeomorphic to $(M^I\times_M M^I)\times_{M\times M} M$ and $(\ev_0, \ev_{1/2}, \ev_1)\colon M^I\times_M M^I\rightarrow M\times M\times M$ is a Hurewicz fibration. The pullback of the bottom right corner is exactly $\FE$ and the induced map $LM\sqcup LM\rightarrow \FE$ coincides with $\iota_0\cup \iota_1$.

Consider the homotopy commutative diagram
\begin{equation}\label{eq:LMshift}
\xymatrix{
\C_\bullet(LM) \ar[r] \ar[d] & 0 \ar[d] \\
\C_\bullet(LM)\oplus \C_\bullet(LM) \ar^-{J_0\oplus J_1}[r] & \C_\bullet(LM)
}
\end{equation}
where the left vertical map is $x\mapsto (x, -x)$. This induces a morphism
\[\H_{\bullet-1}(LM)\rightarrow \H_{\bullet}(LM, LM\sqcup LM).\]

\begin{defn}
The \defterm{loop coproduct} is given by the composite
\begin{align*}
\vee \colon \H_{\bullet+d-1}(LM) \rightarrow \H_{\bullet+d}(LM, LM \sqcup LM) &\longrightarrow \H_\bullet(\FE, LM\sqcup LM) \\
&\longrightarrow \H_\bullet(LM\times LM, M\times LM\cup LM\times M).
\end{align*}
where the second map is given by the relative intersection product using the commutative diagram \eqref{eq:F8intersection} and the second map is the pushforward along \eqref{eq:F8forgetbasepoint}.
\label{def:stringcoproduct}
\end{defn}

\begin{remark}
It follows from \cref{prop:HWrelint} that $\vee$ coincides with the Goresky--Hingston coproduct as described in \cite{HingstonWahl1}, see also \cite[Proposition 4.4]{NaefRiveraWahl}. Note that in \emph{loc. cit.} the loop coproduct has as domain $\H_\bullet(LM, M)$, i.e. it is shown that $\vee$ descends along $\H_\bullet(LM) \to \H_\bullet(LM,M)$. But since that map is split, no information is lost in the above presentation and we obtain the coproduct with domain $\H_\bullet(LM,M)$ by precomposing with that splitting.
\end{remark}

\begin{remark}
If we work with field coefficients or the homology $\H_\bullet(LM, M)$ has no torsion, there is the K\"unneth isomorphism $\H_\bullet(LM\times LM, M\times LM\cup LM\times M)\cong \H_\bullet(LM, M)\otimes \H_\bullet(LM, M)$.
\end{remark}

\begin{remark}\label{rmk:loopcoproductspectra}
If $k$ is an $\bE_\infty$-ring, such that $M$ is oriented over $k$, the loop coproduct defines a map of spectra $\C_{\bullet+d-1}(LM; k)\rightarrow \C_\bullet(LM, M; k)\otimes \C_\bullet(LM, M; k)$.
\end{remark}

\begin{prop}
If we denote by $\vee^{\op}$ the loop coproduct post-composed with the flip of the two $LM$ factors, then for $\alpha\in\H_\bullet(LM)$ we have
\[\vee^{\op}(\alpha) = (-1)^{d+1} \vee(\alpha).\]
\end{prop}
\begin{proof}
We denote by $\flip$ the morphisms exchanging the two factors of $LM$. We have a morphism of pairs
\[
\xymatrix{
LM \ar[r] & LM \\
LM\sqcup LM \ar^{\flip}[r] \ar^{J_0\cup J_1}[u] & LM\sqcup LM, \ar^{J_0\cup J_1}[u]
}
\]
where the top morphism $LM\rightarrow LM$ is given by rotating the loop half-way, i.e. by the map $t\mapsto t + 1/2$.

Consider the diagram
\[
\xymatrix@C=0.5cm{
\H_{\bullet+d-1}(LM) \ar[r] \ar^{\id}[d] & \H_{\bullet+d}(LM, LM \sqcup LM) \ar[r] \ar^{\flip}[d] & \H_\bullet(\FE, LM\sqcup LM) \ar[r] \ar^{\flip}[d] & \H_\bullet(LM\times LM, M\times LM\cup LM\times M) \ar^{\flip}[d] \\
\H_{\bullet+d-1}(LM) \ar[r] & \H_{\bullet+d}(LM, LM \sqcup LM) \ar[r] & \H_\bullet(\FE, LM\sqcup LM) \ar[r] & \H_\bullet(LM\times LM, M\times LM\cup LM\times M)
}
\]
where the squares \emph{commute up to signs} and the two horizontal composites are both given by the loop coproduct $\vee$.

The first morphism in the definition of the loop coproduct is given by \eqref{eq:LMshift}, which introduces a factor of $(-1)$ when we exchange the two factors of $LM$. By \cref{lm:relativeintersectionsymmetry} the second morphism, which is the relative intersection product, introduces a factor of $(-1)^d$ when we exchange the two factors. Finally, the diagram \eqref{eq:F8forgetbasepoint} is invariant under the exchange of the two loops, so the last morphism does not introduce extra factors.
\end{proof}

\subsection{Simple homotopy invariance of the string topology operations}

For $M$ a closed oriented $d$-manifold \cref{prop:manifoldCWduality} provides an equivalence $\ori_M\colon \zeta_M\xrightarrow{\sim} \Z_M[-d]$ in $\LocSys(M)$. Now let $M$ and $N$ be closed oriented $d$-manifolds and $f\colon M\rightarrow N$ a homotopy equivalence in which case we have a canonical isomorphism $f^*\zeta_N\cong \zeta_M$. We say that $f$ is \defterm{orientation-preserving} if the composite
\[\Z_M[-d]\xleftarrow{\ori_M} \zeta_M\cong f^*\zeta_N\xrightarrow{f^*(\ori_N)} \Z_M[-d]\]
is equivalent to the identity. As the intersection product $\cE[-d]\rightarrow \Delta_\sharp \Delta^*\cE$ is constructed from the duality pairing $\Z_M\boxtimes \zeta_M\rightarrow \Delta_\sharp \Z_M$, which is homotopy invariant, as well as the trivialization $\ori_M$ of $\zeta_M$, we see that the intersection product is invariant under orientation-preserving homotopy equivalences.

Let us now study homotopy invariance of the relative intersection product. Suppose $M$ and $N$ are homotopy-equivalent closed oriented $d$-manifolds; we identify the underlying homotopy types using the 
given homotopy equivalence. Let
\[\Int_M, \Int_N\colon (\cE/\Delta_\sharp \cF)[-d]\longrightarrow \Delta_\sharp(\Delta^*\cE/\cF)\]
be the two relative intersection products defined using the Pontryagin--Thom lifts $\lambda_{\PT}(M),\lambda_{\PT}(N)$ of the $\HH$ Euler characteristic. Using \cref{thm:florian} the difference between the two lifts can be identified with the trace of the Whitehead torsion of the homotopy-equivalence $f \colon N \to M$, i.e.
\[
\tr(\tau(f)) = f(\lambda_{\PT}(N)) - \lambda_{\PT}(M) \in \Hom_{\LocSys(M)}(\Z_M[1-d], (\Delta^*\Delta_\sharp \Z_M)/\Z_M) \cong \H_1(LM, M).
\]

\begin{prop}\label{prop:relintersectionhomotopy}
Let $f\colon M\rightarrow N$ be an orientation-preserving homotopy equivalence. The difference $\Int_M - \Int_N$ is given by the composite
\begin{align*}
(\cE/\Delta_\sharp \cF)[-d] &\longrightarrow \Delta_\sharp \cF[1-d] \\
&\xrightarrow{\tr(\tau(f))} \Delta_\sharp \left( \cF\otimes (\Delta^*\Delta_\sharp \Z_M)/\Z_M \right) \\
&\longrightarrow \Delta_\sharp \left( (\Delta^*\Delta_\sharp \cF)/\cF \right) \\
&\xrightarrow{\alpha} \Delta_\sharp(\Delta^*\cE/\cF).
\end{align*}
On homology it simplifies to
\begin{align*}
\H_{\bullet + d}(E, F) &\longrightarrow \H_{\bullet + d -1}(F) \\
&\xrightarrow{\tr(\tau(f))} \H_{\bullet}(F\times_M LM, F) \\
&\longrightarrow \H_\bullet(E\times_{M\times M} M, F),
\end{align*}
where the second map is given by taking the (non-relative) intersection product with the class $\tr(\tau(f))$ and the third map is given by
\[\H_\bullet(F\times_{M\times M} M^I, F)\xrightarrow{a} \H_\bullet(E\times_{M\times M} M^I, F)\xleftarrow{\sim} \H_\bullet(E\times_{M\times M} M, F),\]
where the last map is an isomorphism since $E\rightarrow M\times M$ is a fibration.
\end{prop}
\begin{proof}
By construction, the two relative intersection products are obtained from the commutative diagram \eqref{eqn:echarliftdiag} (see also \cref{rem:relintersectionastensorproduct}). Their difference is induced by the diagram
\begin{equation}
	\begin{tikzcd}
		\Delta_\sharp \Z_M[-d] \ar[r] \ar[d, "\Delta_\sharp (e(M) - e(N))"] & \Z_{M\times M}[-d] \ar[d, "0"] \\
		\Delta_\sharp \Z_M \ar[r, equal] & \Delta_\sharp \Z_M,
	\end{tikzcd}
\end{equation}
which is adjoint to
\begin{equation}
	\begin{tikzcd}
		\Z_M[-d] \ar[r, equal] \ar[d, "(e(M) - e(N))"] & \Z_{M}[-d] \ar[d, "0"] \\
		\Z_M \ar[r] & \Delta^*\Delta_\sharp \Z_M.
	\end{tikzcd}
\end{equation}
This diagram naturally factors as
\begin{equation}
	\begin{tikzcd}
        \Z_M[-d] \ar[r, equal] \ar[d, equal] & \Z_M[-d] \ar[d] \\
        \Z_M[-d] \ar[r] \ar[d, "\tr(\tau(f))"] & 0 \ar[d, equal] \\
        (\Delta^*\Delta_\sharp \Z_M)/\Z_M[-1] \ar[r] \ar[d] & 0 \ar[d] \\
        \Z_M \ar[r] & \Delta^*\Delta_\sharp \Z_M,
	\end{tikzcd}
\end{equation}
which induces the diagram
\begin{equation}
    \begin{tikzcd}
        \Delta_\sharp \cF[-d] \ar[r] \ar[d, equal] & \cE[-d] \ar[d] \\
        \Delta_\sharp \cF[-d] \ar[r] \ar[d, "\tr(\tau(f))"] & 0 \ar[d, equal] \\
        \Delta_\sharp (\cF \otimes (\Delta^*\Delta_\sharp \Z_M)/\Z_M[-1]) \ar[r] \ar[d] & 0 \ar[d] \\
        \Delta_\sharp \cF \ar[r] & \Delta_\sharp\Delta^*\cE.
    \end{tikzcd}
\end{equation}
Passing to horizontal cofibers this gives the stated description.
\end{proof}

\begin{remark}
Taking $\cF=0$ in the previous statement, we recover the fact that the usual intersection product $\cE[-d]\rightarrow \Delta_\sharp \Delta^*\cE$ is homotopy invariant.
\end{remark}

Let $\vee_M,\vee_N$ be the two loop coproducts defined using the relative intersection products $\Int_M,\Int_N$. Let $\nu = \tr(\tau(f))\in\H_1(LM, M)$. Consider the map $\Delta'\colon LM\rightarrow LM\times LM$ given by $\gamma\mapsto (\gamma^{-1}, \gamma)$.

Since $\H_0(LM, M)$ and $\H_0(LM)$ are torsion-free, we have the K\"unneth isomorphism
\[\H_1(LM\times LM, LM\times M)\cong \bigoplus_{i+j=1}\H_i(LM)\otimes \H_j(LM, M).\]
Consider a morphism of pairs $(LM, M)\xrightarrow{\Delta', i\times \id} (LM\times LM, LM\times M)$. This map, along with the K\"unneth isomorphism, gives rise to a morphism
\[\Delta'\colon \H_1(LM, M)\longrightarrow \bigoplus_{i+j=1}\H_i(LM)\otimes \H_j(LM, M)\]
and we denote by $\nu'\otimes \overline{\nu}''$ the image of $\nu$ under this map, where we use the Sweedler notation. We similarly denote by
\[\overline{\nu}'\otimes \nu''\in \bigoplus_{i+j=1}\H_i(LM, M)\otimes \H_j(LM)\]
the image of $\nu$ under the map of pairs $(LM, M)\xrightarrow{\Delta', \id\times i}(LM\times LM, M\times LM)$.

\begin{cor}\label{cor:stringcoproducthomotopyinvariance}
Let $f\colon M\rightarrow N$ be an orientation-preserving homotopy equivalence. Then for every $\alpha\in \H_{n+d-1}(LM)$ we have
\[\vee_M(f(\alpha)) - f(\vee_N(\alpha)) = (-1)^{n+1}\overline{\nu}'\otimes (\nu''\wedge_N f(\alpha)) - (f(\alpha)\wedge_N \nu')\otimes \overline{\nu}''.\]
\end{cor}
\begin{proof}
We apply \cref{prop:relintersectionhomotopy} to the case where $F=F_0\sqcup F_1 = LM\sqcup LM$ and $E=LM$ to obtain that the difference is given by
\begin{align*}
\H_{n + d-1}(LM) &\cong \H_{n+d}(LM, LM \sqcup LM) \\
&\to \H_{n+d-1}(LM \sqcup LM) \\
&\xrightarrow{\tr(\tau(f))} \H_{n}(LM\times_M (LM\sqcup LM), LM\sqcup LM) \\
&\to \H_{n}(\FE, LM \sqcup LM) \\
&\to \H_{n}(LM\times LM, LM\times M\cup M\times LM).
\end{align*}

We have a commutative diagram
\[
\xymatrix@C=1cm{
\H_{n+d-1}(LM \sqcup LM) \ar^-{\tr(\tau(f))}[r] & \H_{n}(LM\times_M (LM\sqcup LM), LM\sqcup LM) \ar[r] & \H_n(\FE, LM\sqcup LM) \\
\H_{n+d-1}(LM)^{\oplus 2} \ar^-{\tr(\tau(f))\oplus \tr(\tau(f))}[r] \ar^{\sim}[u] & \H_n(LM\times_M LM, LM)^{\oplus 2} \ar[r] \ar[u] & \H_n(\FE, LM)^{\oplus 2} \ar[u]
}
\]
The composite 
\[\H_{n + d-1}(LM) \cong \H_{n+d}(LM, LM \sqcup LM)\to \H_{n+d-1}(LM\sqcup LM)\cong \H_{n+d-1}(LM)^{\oplus 2}\]
is given by $x\mapsto (x, -x)$. Let us now describe the composite of the remaining maps from each summand in $\H_{n+d-1}(LM)^{\oplus 2}$.

The map from the first summand is given by the composite
\begin{align*}
\H_{n+d-1}(LM) &\xrightarrow{\id\otimes \tr(\tau(f))} \H_{n+d-1}(LM)\otimes \H_1(LM, M) \\
&\xrightarrow{\Int_N} \H_n(LM\times_M LM, LM) \\
&\rightarrow \H_n(\FE, LM) \\
&\rightarrow \H_n(LM\times LM, M\times LM),
\end{align*}
where the map $LM\times_M LM\rightarrow \FE$ is given by $(\gamma_1, \gamma_2)\mapsto (\gamma_2^{-1}, \gamma_2\ast \gamma_1)$.

By naturality of the intersection product we have a commutative diagram
\[
\xymatrix{
\H_1(LM, M)\otimes \H_{n+d-1}(LM) \ar^{\Int_N}[r] \ar^{\Delta'\otimes \id}[dd] & \H_n(LM\times_M LM, LM) \ar[d] \\
& \H_n(LM\times LM, M\times LM) \\
\bigoplus_{i+j=1} \H_i(LM, M)\otimes \H_j(LM) \otimes \H_{n+d-1}(LM) \ar^-{\id\otimes \Int_N}[r] & \bigoplus_{i+j=1} \H_i(LM, M)\otimes \H_{n+j-1}(LM\times_M LM) \ar^{\id\otimes m}[u]
}
\]

Using the symmetry of the K\"unneth isomorphism and the intersection product (see \cref{lm:intersectionsymmetry}) we see that the composites
\[\H_{n+d-1}(LM) \xrightarrow{\id\otimes \tr(\tau(f))} \H_{n+d-1}(LM)\otimes \H_1(LM, M)\xrightarrow{\Int_N} \H_n(LM\times_M LM, LM)\xrightarrow{\sigma} \H_n(LM\times_M LM, LM),\]
where the last map is given by the flip of the two $LM$ factors, and
\[\H_{n+d-1}(LM) \xrightarrow{\tr(\tau(f))\otimes \id} \H_1(LM, M)\otimes \H_{n+d-1}(LM)\xrightarrow{\Int_N} \H_n(LM\times_M LM, LM)\]
differ by $(-1)^{n+d-1+d}=(-1)^{n+1}$.

This shows that the map from the first summand is given by $\alpha\mapsto (-1)^{n+1}\overline{\nu}'\otimes (\nu''\wedge \alpha)$. Similarly, the map from the second summand is given by the composite
\begin{align*}
\H_{n+d-1}(LM) &\xrightarrow{\id\otimes \tr(\tau(f))} \H_{n+d-1}(LM)\otimes \H_1(LM, M) \\
&\xrightarrow{\Int_N} \H_n(LM\times_M LM, LM) \\
&\rightarrow \H_n(\FE, LM) \\
&\rightarrow \H_n(LM\times LM, LM\times M),
\end{align*}
where the map $LM\times_M LM\rightarrow \FE$ is given by $(\gamma_1, \gamma_2)\mapsto (\gamma_1\ast \gamma_2^{-1}, \gamma_2)$. Similarly to the first summand, we see that the map from the second summand is given by $\alpha\mapsto (\alpha\wedge \nu')\otimes \overline{\nu}''$.
\end{proof}

\begin{remark}
    Over the integers the error term in \cref{cor:stringcoproducthomotopyinvariance} is in the image of the K\"unneth map 
    \[
    \H_\bullet(LM,M) \otimes \H_\bullet(LM,M) \to \H_\bullet(LM \times LM, M \times LM \cup LM \times M).
    \]
    Now suppose $k$ is an $\bE_\infty$-ring, such that $M$ is oriented over $k$. In this case we have the spectral-level loop coproduct from \cref{rmk:loopcoproductspectra}. Then the same transformation formula is valid over $k$, where both sides are understood as maps $\C_{\bullet+d-1}(LM; k) \to \C_\bullet(LM, M; k)\otimes \C_\bullet(LM, M; k)$.
\end{remark}

\section{Shklyarov copairing and its nullhomotopy}
\label{sect:Shklyarovexamples}

\subsection{Mukai pairing and Shklyarov copairing}
\label{sect:Shklyarov}

Fix an $\bE_\infty$-ring $k$ and let $\cC\in\PrSt_k$ be a $k$-linear stable presentable $\infty$-category. Let us recall the Mukai pairing and the Shklyarov copairing from \cite{Caldararu1,Shklyarov}.

Recall that for a proper $\infty$-category $\cC\in\PrSt_k$ the evaluation pairing $\ev\colon \cC\otimes \cC^\vee\rightarrow \Mod_k$ admits a colimit-preserving right adjoint and for a smooth $\infty$-category the coevaluation pairing $\coev\colon \Mod_k\rightarrow \cC^\vee\otimes \cC$ admits a colimit-preserving right adjoint. In particular, by \cref{def:transfermap} we get an induced map on dimensions.

\begin{defn}
    Let $\cC\in\PrSt_k$ be a $k$-linear stable presentable $\infty$-category.
    \begin{itemize}
        \item If $\cC$ is proper, the \defterm{Mukai pairing} is the composite
        \[[\ev]\colon \dim(\cC)\otimes \dim(\cC)\cong \dim(\cC\otimes \cC^\vee)\xrightarrow{\dim(\ev)} k.\]
        \item If $\cC$ is smooth, the \defterm{Shklyarov copairing} is the composite
        \[[\coev]\colon k\xrightarrow{\dim(\coev)}\dim(\cC^\vee\otimes \cC)\cong \dim(\cC)\otimes \dim(\cC).\]
    \end{itemize}
\end{defn}

A relationship between the Mukai pairing and copairing is given by the following straightforward observation, see \cite[Theorem 6.2]{Shklyarov}

\begin{prop}
Suppose $\cC\in\PrSt_k$ is a smooth and proper $\infty$-category. The Mukai pairing and Shklyarov copairing establish a self-duality of $\dim(\cC)\in\Mod_k$.
\label{prop:Mukaiduality}
\end{prop}
\begin{proof}
The assumptions guarantee that $\ev$ and $\coev$ are morphisms in $(\PrSt_k)^\dual$ and thus $\cC$ is dualizable in $(\PrSt_k)^\dual$. Using the functoriality of dimensions from \cref{thm:ExistenceOfDim} we get that $\dim(\cC)\in\Mod_k$ is dualizable with dual $\dim(\cC^\vee)\cong\dim(\cC)$.
\end{proof}

In this section we will present several examples of smooth $\infty$-categories with the following additional structures:
\begin{itemize}
    \item There is a morphism $\widetilde{\dim}(\cC)\rightarrow \dim(\cC)$.
    \item The element $[\coev]\in\dim(\cC)\otimes \dim(\cC)$ lies in the image of $\widetilde{\dim}(\cC)\otimes \widetilde{\dim}(\cC)\rightarrow \dim(\cC)\otimes \dim(\cC)$.
\end{itemize}

We will show in \cref{sect:TQFT} (see \cref{def:secondarycoproduct}) that this structure allows us to define an operation
\[\fZ(\cC)\rightarrow \overline{\dim}(\cC)\otimes \overline{\dim}(\cC)[-1],\]
where $\overline{\dim}(\cC)$ is the cofiber of $\widetilde{\dim}(\cC)\rightarrow \dim(\cC)$.

\subsection{Examples in algebra}\label{sect:exalgebra}

Let $k$ be a field and $\Alg_k$ the $\infty$-category of dg algebras over $k$. Let $A^e = A\otimes A^{\op}$ be the enveloping algebra.

Recall the following notions:
\begin{itemize}
    \item A dg algebra $A$ is \defterm{finite cell} if it is isomorphic as a graded algebra to the free algebra $k\langle x_1, \dots, x_n\rangle$ with the differential satisfying $dx_i\in k\langle x_1, \dots, x_{i-1}\rangle$.
    \item A dg algebra $A$ is \defterm{locally of finite presentation} if it is compact in $\Alg_k$.
    \item A dg algebra $A$ is \defterm{smooth} if the diagonal bimodule (i.e. $A$ as an $A^e$-module via the left and right actions of $A$ on itself) is compact in the $\infty$-category $\LMod_{A^e}$ of $A^e$-modules.
\end{itemize}

\begin{remark}
By \cite{ToenVaquie} a dg algebra $A$ is locally of finite presentation if, and only if, it is a retract of a finite cell algebra. Moreover, a dg algebra locally of finite presentation is smooth.
\end{remark}

\begin{remark}
For any dg algebra $A$, the $\infty$-category $\cC = \LMod_A\in\PrSt$ of dg $A$-modules is dualizable with $\cC^\vee = \LMod_{A^{\op}}$ and $\coev\colon \Mod_k\rightarrow \LMod_A\otimes \LMod_{A^{\op}}\cong \LMod_{A^e}$ given by sending $k$ to the diagonal bimodule $A$. So, a dg algebra $A$ is smooth in the sense of the above definition precisely if $\cC$ is smooth in the sense of \cref{def:smoothproper}.
\end{remark}

If $A$ is smooth, we have a natural point $[A]\in K(A^e)$ given by the class of the diagonal bimodule. Let $[A^e]\in K(A^e)$ be the class of $A^e$ as a left $A^e$-module.

\begin{lm}\label{lm:finitecellresolution}
Let $A = (k\langle x_1, \dots, x_n\rangle, d)$ be a finite cell dg algebra. Then there is a homotopy
\[[A]\sim [A^e](1-\sum_{i=1}^n (-1)^{|x_i|})\]
in $K(A^e)$.
\end{lm}
\begin{proof}
Consider the multiplication map $m\colon A\otimes A\rightarrow A$. It defines a surjective morphism of $A^e$-modules, where $A$ is the diagonal bimodule. Let $\Omega^1_A$ be the kernel of $m$. By \cite[Lemma 2.17]{Yeung} the $A^e$-module $\Omega^1_A$ is also finite cell on the generators $\{dx_1, \dots, dx_n\}$, where $dx_i = x_i\otimes 1 - 1\otimes x_i$. We have a short exact sequence
\[0\longrightarrow \Omega^1_A\longrightarrow A^e\xrightarrow{m} A\longrightarrow 0\]
which produces a homotopy $[A]\sim [A^e] - [\Omega^1_A]$. Since $\Omega^1_A$ is finite cell, it admits a filtration with associated graded given by the free dg $A^e$-module on the generators $\{dx_1,\dots, dx_n\}$. This defines a homotopy
\[[\Omega^1_A]\sim [A^e] \sum_{i=1}^n (-1)^{|x_i|}.\]
\end{proof}

\begin{remark}
It is announced by Efimov that, conversely, a dg algebra locally of finite presentation is quasi-isomorphic to a finite cell algebra if, and only if, $[A]$ is homotopic to a multiple of $[A^e]$.
\end{remark}

Let $A$ be a finite cell dg algebra and $\cC = \Mod_A\in\PrSt$ the derived $\infty$-category of dg $A$-modules. Then by \cref{lm:finitecellresolution} we see that the class $[\coev]\in\dim(\cC)\otimes \dim(\cC)$ is proportional to $[A^e]$. In particular, if we let $\widetilde{\dim}(\cC)=k$ mapping to $\dim(\cC)$ via the class of the free module $[A]$, then $[\coev]$ admits a lift along $\widetilde{\dim}(\cC)\otimes \widetilde{\dim}(\cC)\rightarrow \dim(\cC)\otimes \dim(\cC)$.

\subsection{Examples in homotopy theory}
\label{sect:topologyexamples}

Let $X$ be a finitely dominated space and consider $\cC=\Sp^X$. We have the following results:
\begin{itemize}
    \item $\dim(\Sp^X)\cong \Sigma^\infty_+ LX$ by \cref{prop:LocSysTHH}.
    \item $\Sp^X$ is smooth by \cref{cor:finitelydominatedsmooth}.
    \item $\bS_X\in\Sp^X$ is compact by \cref{prop:finitetypeproper}. In particular, we have the $\THH$ Euler characteristic
    \[\chi_{\THH}(X) = [\bS_X]\in\Omega^\infty\Sigma^\infty_+ LX.\]
    \item By \cref{prop:Spselfdual} and \cref{prop:LocSysTHH} we have
    \[[\coev] = (L\Delta) \chi_{\THH}(X)\in\Omega^\infty\Sigma^\infty_+(LX\times LX),\]
    where $L\Delta\colon LX\rightarrow LX\times LX$ is the diagonal map.
\end{itemize}

Let $\widetilde{\dim}(\Sp^X) = \Sigma^\infty_+ X$ mapping to $\dim(\Sp^X)$ via the inclusion of constant loops. From the above description of the assembly map, we see that if $X$ is equipped with a lift of the $\THH$ Euler characteristic, then $[\coev]$ lifts along $\widetilde{\dim}(\Sp^X)\otimes \widetilde{\dim}(\Sp^X)$.

Having an extra structure on $X$ allows us to give a stronger statement.
 
\begin{defn}
The \defterm{homological Euler class} $e(X)\in\Omega^\infty\Sigma^\infty_+ X$ is the image of $\chi_{\THH}(X)$ under $LX\rightarrow X$ given by evaluating the loop at the basepoint of $S^1$.
\end{defn}

\begin{remark}
We may identify
\[\pi_0(\Omega^\infty\Sigma^\infty_+ X)\cong \H_0(X;\Z).\]
If $X$ is a closed oriented $d$-manifold, we may also identify $\H_0(X;\Z)\cong \H^d(X;\Z)$ and under this identification the homological Euler class goes to the usual Euler class in cohomology.
\end{remark}

\begin{prop}
Let $X$ be a finitely dominated space equipped with the following structure:
\begin{itemize}
    \item A lift of the $\THH$ Euler characteristic along the assembly map.
    \item A trivialization $e(X)\sim 0$ of the homological Euler class.
\end{itemize}
Then $[\coev]\in\dim(\Sp^X)\otimes \dim(\Sp^X)$ admits a nullhomotopy.
\end{prop}
\begin{proof}
	The composite
	\[X\longrightarrow LX\longrightarrow X\]
	is the identity, where the first morphism is the inclusion of constant loops. Therefore, a lift of the $\THH$ Euler characteristic along the assembly map provides identifies $\chi_{\THH}(X)$ with the image of $e(X)$ under the assembly map. So, a trivialization of $e(X)$ gives rise to a trivialization of $\chi_{\THH}(X)$.
	
	The last statement follows from the fact that a connected finite CW complex carries the Whitehead lift of the $\THH$ Euler characteristic and
	\[\pi_0(\Omega^\infty\Sigma^\infty_+ X)\cong \H_0(X; \Z)\cong \Z.\]
\end{proof}

\begin{remark}
Suppose $M$ is a connected finite CW complex and let $X=\Sing(M)$ the underlying homotopy type. Then $X$ carries the Whitehead lift of the $\THH$ Euler characteristic and under the composite
\[\pi_0(\Omega^\infty\Sigma^\infty_+ X)\cong \H_0(X; \Z)\cong \Z\]
we have $e(X)\mapsto \chi(X)$. In particular, if $\chi(X)=0$, there is a trivialization $e(X)\sim 0$ of the homological Euler class.
\end{remark}

\begin{remark}
Given a finite CW complex $M$, the homological Euler class is represented by the $0$-chain $\sum_{a\in A}(-1)^{\dim(a)} \alpha_a$, where $A$ is the set of cells of $M$ and $\alpha_a$ is a point in the interior of $a$. A bounding $1$-chain for the homological Euler class is called a combinatorial Euler structure in \cite{TuraevEuler}.
\end{remark}

\subsection{Examples in symplectic geometry} \label{sect:exsymplecticgeometry}

Let $M$ be a Liouville manifold of dimension $2n$, i.e. an exact symplectic manifold with a convexity condition at infinity, satisfying the nondegeneracy condition (see e.g. \cite[Definition 1.1]{Ganatra}). For example, $M=\T^*N$ for a compact manifold $N$ satisfies the nondegeneracy condition \cite{AbouzaidCotangent}. More generally, any Weinstein manifold satisfies the non-degeneracy condition by \cite{CDGG}. In addition, choose a trivialization of $2c_1(\T_M)$ for some almost complex structure on $M$ (again, for $M=\T^* N$ there is a canonical such structure). Let $k$ be a field.

Let $\SH^\bullet(M)$ be the symplectic cohomology of $M$ (our grading conventions follow \cite{SeidelBiased}) and $\cW(M)$ the wrapped Fukaya category. Non-degeneracy of $M$ has the following implications:
\begin{itemize}
	\item The $\infty$-category $\Mod_\cW(M)=\Fun(\cW(M)^{\op}, \Mod_k)\in\PrSt_k$ of $\cW(M)$-modules is smooth, see \cite[Theorem 1.2]{Ganatra}.
	\item The natural open-closed map
	\[\OC\colon\H^\bullet(\dim(\Mod_\cW(M)))\longrightarrow\SH^\bullet(M)\]
	is an isomorphism, see \cite[Theorem 1.1]{Ganatra}.
\end{itemize}

Let $\widetilde{\dim}(\Mod_\cW(M))$ be the action zero part of symplectic cohomology, so that $\H^\bullet(\widetilde{\dim}(\Mod_\cW(M)))\cong \H^{\bullet+n}(M)$ (see e.g. \cite[Section 5]{Ritter}). The cohomology of the cofiber $\overline{\dim}(\Mod_\cW(M))$ is then the positive action symplectic cohomology $\SH^\bullet_{>0}(M)$. It fits into a long exact sequence
\[\longrightarrow \H^{\bullet+n}(M)\xrightarrow{e^*} \SH^\bullet(M)\longrightarrow \SH_{>0}^\bullet(M)\longrightarrow \H^{\bullet+n+1}(M)\longrightarrow\]

We claim that
\[[\coev]\in \HH_\bullet(\cW(M\times M))\cong \SH^\bullet(M\times M)\]
is equal to the image $e^*[\Delta]$ of the diagonal class $[\Delta]\in\H^{2n}(M\times M)$. Let us sketch the proof. By \cref{cor:Mukaicopairingelbow} the class $[\coev]$ can be computed in terms of the TQFT structure on Hochschild homology $\HH_\bullet(\cW(M))$. We refer to \cite{Ritter} for a construction of a TQFT structure on symplectic cohomology $\SH^\bullet(M)$. It can be shown that the open-closed map $\OC$ intertwines the TQFT structures similarly to the proof of \cite[Proposition 5.3]{Ganatra}. In particular, $[\coev] = \psi_Q(1)$ in the notation of \cite{Ritter}. By \cite[Theorem 6.10]{Ritter} we get $\psi_Q(1) = e^*(\psi_{Q'}(1))$, where $\psi_{Q'}(1)=[\Delta]$ is the diagonal class.

In particular, $[\coev]$ lifts along the map
\[\widetilde{\dim}(\Mod_\cW(M))\otimes \widetilde{\dim}(\Mod_\cW(M))\rightarrow \dim(\Mod_\cW(M))\otimes \dim(\Mod_\cW(M)).\]

\begin{example}
Consider a compact $n$-manifold $N$ and let $M=\T^* N$. Then the diagonal class $[\Delta_M]$ is proportional to $\chi(N)$. In particular, in the case $\chi(N)=0$ it vanishes. In fact, Abouzaid \cite{AbouzaidOpenViterbo} shows that $\Mod_{\cW(\T^* N)}$ is equivalent to the $\infty$-category of $k$-local systems on $N$ twisted by the second Stiefel--Whitney class $w_2(N)\in\H^2(N; \Z/2)$. So, this example fits both in the framework of symplectic geometry and homotopy theory as described in \cref{sect:topologyexamples}.
\end{example}

\subsection{Examples in algebraic geometry}

Throughout this section $X$ will be a quasi-compact and quasi-separated scheme $X$ over the base ring $k\supset \Q$. We denote by $p\colon X\rightarrow \Spec k$ the natural projection.

Let $\QCoh(X)$ be the unbounded derived $\infty$-category of quasi-coherent complexes on $X$. By \cite[Proposition 2.5]{Neeman} (see also \cite[Theorem 3.1.1]{BondalVDB} and \cite[Proposition 9.6.1.1]{LurieSAG}) $\QCoh(X)$ is compactly generated by the subcategory $\Perf(X)$ of perfect complexes. The proof of the following statement is identical to the proof of \cref{prop:Spselfdual}.

\begin{prop}
Let $X$ be a quasi-compact and quasi-separated scheme. Then $\QCoh(X)$ is self-dual with the duality functors
\[\ev\colon \QCoh(X)\otimes \QCoh(X)\xrightarrow{\Delta^*}\QCoh(X)\xrightarrow{p_*} \Mod_k\]
and
\[\coev\colon \Mod_k\xrightarrow{p^*}\QCoh(X)\xrightarrow{\Delta_*} \QCoh(X)\otimes \QCoh(X).\]
\label{prop:QCohselfdual}
\end{prop}

Denote
\[\HH_\bullet(X) = \HH_\bullet(\Perf(X))\cong \dim(\QCoh(X)).\]
From now on, we assume that $X$ is a smooth and separated scheme, in which case $\QCoh(X)$ is a smooth $\infty$-category by \cite{Lunts}. By \cite{Yekutieli} we have a quasi-isomorphism
\[\HH_\bullet(X) \cong \R\Gamma(X, \Sym(\Omega^1_X[1])).\]
By \cref{prop:QCohselfdual} the Shklyarov copairing is given by
\[[\coev] = \ch(\Delta_*\cO_X)\in\HH_\bullet(X)\otimes \HH_\bullet(X).\]

\begin{prop}
Suppose $X$ is a separated smooth affine variety of nonzero dimension. Then the diagonal class $[\coev]\in\HH_\bullet(X)\otimes \HH_\bullet(X)$ admits a nullhomotopy.
	\label{prop:affineEulerstructure}
\end{prop}
\begin{proof}
We have
\[[\ch(\Delta_*\cO_X)]\in \bigoplus_n \H^n(X, \Omega^n_X).\]
By the Grothendieck--Riemann--Roch theorem \cite[Theorem 15.2]{Fulton} $\ch(\Delta_* \cO_X) = \Delta_* (\Td(X)^{-1})$ whose nonzero components have $n\geq \dim(X)$. But $\H^n(X, \Omega^n_X) = 0$ for $n\geq 1$ since $X$ is affine.
\end{proof}

In the case when $X$ is a curve, we can describe such nullhomotopies geometrically as follows.

\begin{prop}
Suppose $X$ is a smooth and separated curve equipped with a rational one-form $\omega\in\Omega^1(k(X\times X))$ on $X\times X$, which is regular away from the diagonal and with a simple pole along the diagonal with residue the constant function $1$. Then there is a nullhomotopy of $[\coev]\in\HH_\bullet(X)\otimes \HH_\bullet(X)$.
\end{prop}
\begin{proof}
	We have a resolution
	\[0\longrightarrow \cO_{X\times X}(-\Delta)\longrightarrow \cO_{X\times X}\longrightarrow \Delta_*\cO_X\longrightarrow 0,\]
	where $\cL=\cO_{X\times X}(-\Delta)$ is a line bundle, since $\Delta$ is a Cartier divisor. This gives
	\[\ch(\Delta_*\cO_X) = 1 - (1 + c_1(\cL) + c_1(\cL)^2/2).\]
	In particular, a trivialization of $c_1(\cL)$ gives rise to a trivialization of $\ch(\Delta_*\cO_X)$.

	A trivialization of $c_1(\cL)$ is the same as a connection on $\cL$. In the trivialization given by the rational section $1$ of $\cL=\cO_{X\times X}(-\Delta)$, a connection is given by $\nabla=\d-\omega$ for some rational one-form $\omega$ which is regular away from the diagonal. The regularity of the connection on $\cL$ then implies that $\omega$ has a simple pole along the diagonal with residue $1$.
\end{proof}

\section{Topological quantum field theories}
\label{sect:TQFT}

Given a smooth $\infty$-category $\cC\in\PrSt$ together with a partial trivialization of the Shklyarov copairing $[\coev]$ as in \cref{sect:Shklyarov}, the goal of this section is to produce a secondary coproduct
\[\fZ(\cC)\longrightarrow \overline{\dim}(\cC)\otimes \overline{\dim}(\cC)[-1],\]
see \cref{def:secondarycoproduct}. We describe the above operation in terms of framed two-dimensional topological quantum field theories. We will use the cobordism hypothesis, see \cite{LurieCobordism,AyalaFrancis,GradyPavlov2}. Let us mention that for our purposes it is enough to work in the underlying (3, 2)-category of the $(\infty, 2)$-category of two-dimensional extended bordisms, so \cite{Pstragowski,Araujo} give us the relevant results. In \cref{sect:TFTexplicit} we provide a decomposition of bordisms into elementary ones which allows one to forego the cobordism hypothesis and write the relevant operations in terms of the duality data; the cobordism hypothesis then asserts that the operations are independent of the choice of the decomposition of the bordism into elementary ones.

\subsection{Background}
\label{sect:TFTbackground}

For a manifold we denote by $\underline{\bR}^k$ the trivial $k$-dimensional vector bundle.

\begin{defn}
	Let $M$ be a smooth $k$-manifold. An \defterm{$n$-framing} (for $n\geq k$) on $M$ is a trivialization of the stabilized tangent bundle:
	\[\T_M\oplus \underline{\bR}^{n-k}\cong \underline{\bR}^n.\]
\end{defn}

We will work with framings on extended bordisms. We refer to \cite{SchommerPries,Pstragowski} for precise definitions of bordisms and framings on those. Extended bordisms organize into the symmetric monoidal $(\infty, 2)$-category $\Bord_2^{\fr}$ which has the following informal description (see \cite{LurieCobordism,CalaqueScheimbauer} for a precise construction):
\begin{itemize}
	\item Its objects are closed 2-framed 0-manifolds (finite collections of points).
	
	\item Its 1-morphisms from $X$ to $Y$ are compact 2-framed 1-bordisms, i.e. compact 2-framed 1-manifolds $B$ with specified incoming and outgoing boundaries: $\partial B\cong X\sqcup Y$ (with a certain compatibility with the framing in a collared neighborhood of the boundary).
	
	\item Its 2-morphisms from $B_1\colon X\rightarrow Y$ to $B_2\colon X\rightarrow Y$ are 2-framed 2-bordisms, i.e. compact 2-framed 2-manifolds $\Sigma$ with corners whose boundary is decomposed into $X\times [0, 1]\sqcup Y\times [0, 1]$ and $B_1\sqcup B_2$.
\end{itemize}

\begin{defn}
	Let $\cA$ be a symmetric monoidal $(\infty, 2)$-category. An object $x\in\cA$ is \defterm{fully dualizable} if it is dualizable and the evaluation and coevaluation morphisms have a tower of adjoints
	\[\dots\dashv (\ev^\L)^\L\dashv \ev^\L\dashv \ev\dashv \ev^\R\dashv (\ev^\R)^\R\dashv \dots\]
\end{defn}

It turns out that to check that an object is fully dualizable it is enough to perform finitely many checks. The following is \cite[Theorem 3.9]{Pstragowski} and \cite[Proposition 4.2.3]{LurieCobordism}.

\begin{prop}
	Let $\cA$ be a symmetric monoidal $(\infty, 2)$-category. An object of $\cA$ is fully dualizable if, and only if, it is smooth and proper.
\end{prop}

\begin{defn}
	Let $\cA$ be a symmetric monoidal $(\infty, 2)$-category and $x\in\cA$ a smooth object. The \defterm{inverse Serre morphism} is the composite
	\[T\colon x\xrightarrow{\id\otimes\coev} x\otimes x^\vee\otimes x\xrightarrow{\sigma\otimes\id} x^\vee\otimes x\otimes x\xrightarrow{\coev^\R\otimes\id} x.\]
	\label{def:inverseSerre}
\end{defn}

\begin{remark}
	Despite the terminology, the morphism $T\colon x\rightarrow x$ is not always invertible. If we assume that $x\in\cC$ is smooth \emph{and proper}, one may use $\coev^\L$ to similarly define a morphism $S\colon x\rightarrow x$ inverse to $T$ (see the proof of \cite[Proposition 4.2.3]{LurieCobordism}).
\end{remark}

For a symmetric monoidal $(\infty, 2)$-category $\cA$ we denote by $(\cA^{\fd})^{\sim}$ the $\infty$-groupoid consisting of fully dualizable objects of $\cA$. The following is proven in \cite[Theorem 8.1]{Pstragowski} on the level of 2-categories and sketched in \cite[Theorem 2.4.6]{LurieCobordism} on the level of $(\infty, 2)$-categories.

\begin{thm}[Cobordism hypothesis]
	Let $\cA$ be a symmetric monoidal $(\infty, 2)$-category. The evaluation functor $\Fun^{\otimes}(\Bord_2^{\fr}, \cA)\rightarrow (\cA^{\fd})^{\sim}$ given by $Z\mapsto Z(\pt)$ is an equivalence.
	\label{thm:cobordismhyp}
\end{thm}

We will now consider a variant of the above statement. Consider the $(\infty, 2)$-subcategory $\Bord_2^{\fr, \nc}\subset \Bord_2^{\fr}$ which has the same objects and 1-morphisms, while the 2-morphisms
\[\begin{tikzcd}
	X \arrow[r, bend left=50, ""{name=U, below}, "B"] \arrow[r, bend right=50, ""{name=D}, "B'"{below}] & Y \arrow[Rightarrow, from=U, to=D, "\Sigma"]
\end{tikzcd}\]
are extended bordisms $\Sigma$, such that every connected component of $\Sigma$ has a nonempty intersection with $B'$.

\begin{remark}
	Lurie in \cite[Definition 4.2.10]{LurieCobordism} considers a version $\Bord_2^{\nc}$ of $\Bord_2^{\fr, \nc}$, where the bordisms are oriented and the condition on $\Sigma$ is that every connected component has a nonempty intersection with $B$ instead.
\end{remark}

\begin{thm}[Noncompact cobordism hypothesis]
	Let $\cA$ be a symmetric monoidal $(\infty, 2)$-category. The $\infty$-groupoid $\Fun^{\otimes}(\Bord_2^{\fr, \nc}, \cA)$ is equivalent to the $\infty$-groupoid of smooth objects of $\cA$ with the property that the inverse Serre morphism $T$ is invertible.
	\label{thm:noncompactcobordismhyp}
\end{thm}


We refer to symmetric monoidal functors $\Bord_2^{\fr, \nc}\rightarrow \cA$ as \defterm{positive boundary framed 2-dimensional TQFTs}. The following is shown in \cite[Section 9.3]{CalaqueScheimbauer}.

\begin{prop}
Let $\cA$ be a symmetric monoidal $(\infty, 2)$-category and $Z\colon \Bord_2^{\fr, \nc}\rightarrow \cA$ a positive boundary framed 2d TQFT. The inverse Serre morphism $T\colon Z(\pt)\rightarrow Z(\pt)$ is the image under $Z$ of a 1-bordism $\pt\rightarrow \pt$ represented by the interval with one twist in the framing.
\label{prop:Serretwist}
\end{prop}

\subsection{TQFT operations}
\label{sect:TFToperations}

Let $\cA$ be a symmetric monoidal $(\infty, 2)$-category and fix a symmetric monoidal functor $Z\colon \Bord_2^{\fr, \nc}\rightarrow \cA$. By \cref{thm:noncompactcobordismhyp} it corresponds to a smooth object $x=Z(\pt)\in\cA$ whose inverse Serre map $T\colon x\rightarrow x$ is invertible. In this section we describe some bordisms in $\Bord_2^{\fr, \nc}$ and the corresponding operations provided by the positive boundary TQFT.

Consider the circle $S^1$. The space of 2-framings compatible with the given orientation is a torsor over $\Map(S^1, \SO(2))\cong \Z$. Correspondingly, we may parametrize 2-framed circles by integers, and we denote the corresponding 2-framed circle by $S^1_n$ for $n\in\Z$, viewed as a 1-morphism $\varnothing\rightarrow \varnothing$ in $\Bord_2^\fr$. Our conventions are as follows:
\begin{itemize}
	\item The disk provides a 2-morphism $\varnothing\rightarrow S^1_1$ in $\Bord_2^{\fr, \nc}$.
	
	\item $S^1_0$ is the 2-framing on $S^1$ induced by the unique 1-framing.
	
	\item The disk provides a 2-morphism $S^1_{-1}\rightarrow \varnothing$ in $\Bord_2^\fr$. Note that it does not lie in $\Bord_2^{\fr, \nc}$ as this bordism has an empty intersection with the outgoing boundary.
\end{itemize}

\begin{remark}
	We number the 2-framings on the circle as in \cite{SternSzegedy}. The numbering in \cite{TelemanTFT} is slightly different: the three 2-framings above are $S^1_0, S^1_1, S^1_2$, respectively.
\end{remark}

Using \cref{prop:Serretwist} the values of $Z(S^1_n)\in\End_\cA(\bu)$ may be explicitly computed as follows.

\begin{prop}
	We have an equivalence $Z(S^1_n)\cong \ev\circ\sigma\circ (\id\otimes T^n)\circ\coev$. In particular, $Z(S^1_0)\cong \dim(x)$ and $Z(S^1_1)\cong \fZ(x)$.
	\label{prop:S1nvalue}
\end{prop}

\begin{figure}
	\begin{tikzpicture}[thick]
		\draw (0, 0) [partial ellipse = -180:0:0.3 and 0.15];
		\draw (0, 0) [dashed, partial ellipse = 180:0:0.3 and 0.15];
		\draw (1.5, 0) [partial ellipse = -180:0:0.3 and 0.15];
		\draw (1.5, 0) [dashed, partial ellipse = 180:0:0.3 and 0.15];
		\draw (0.3, 0) arc (180:0:0.45);
		\draw (-0.3,0) to[out=90, in=-90] (0.45,1.7);
		\draw (1.5+0.3,0) to[out=90, in=-90] (0.45+0.6,1.7);
		\draw (0.75, 1.7) ellipse (0.3 and 0.15);
		\draw (0, -0.5) node {$S^1_n$};
		\draw (1.5, -0.5) node {$S^1_m$};
		\draw (0.75, 2.2) node {$S^1_{n+m-1}$};
		
		\draw (3.3, 1) ellipse (0.3 and 0.15);
		\draw (3.3, 1) [partial ellipse = -180:0:0.3 and 0.6];
		\draw (3.3, 1.5) node {$S^1_1$};
		
		\draw (6, 0) [partial ellipse = -180:0:0.3 and 0.15];
		\draw (6, 0) [dashed, partial ellipse = 180:0:0.3 and 0.15];
		\draw (5.25, 1.7) ellipse (0.3 and 0.15);
		\draw (6.75, 1.7) ellipse (0.3 and 0.15);
		\draw (5.7, 0) to[out=90, in=-90] (4.95, 1.7);
		\draw (6.3, 0) to[out=90, in=-90] (7.05, 1.7);
		\draw (5.55, 1.7) arc (-180:0:0.45);
		\draw (6, -0.5) node {$S^1_{n+m+1}$};
		\draw (5.25, 2.2) node {$S^1_n$};
		\draw (6.75, 2.2) node {$S^1_m$};
	\end{tikzpicture}
	\caption{Product, unit and coproduct.}
	\label{fig:TFToperations}
\end{figure}
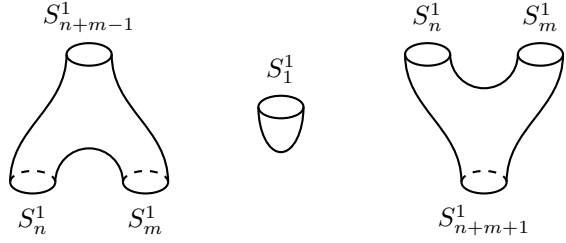

Consider the bordisms from \cref{fig:TFToperations}. Our convention is that the 1-morphisms go from left to right and 2-morphisms go from bottom to top in the picture. Applying the functor $Z$ we obtain the following operations:
\begin{itemize}
	\item A product $\wedge\colon Z(S^1_n)\otimes Z(S^1_m)\rightarrow Z(S^1_{n+m-1})$.
	\item A non-counital coproduct $\bbDelta\colon Z(S^1_1)\rightarrow Z(S^1_0)\otimes Z(S^1_0)$.
\end{itemize}


\begin{figure}[ht]
	\begin{tikzpicture}[thick]
		\draw (-5, 0) [partial ellipse = -180:0:0.3 and 0.15];
		\draw (-5, 0) [dashed, partial ellipse = 180:0:0.3 and 0.15];
		\draw (-4.25, 1.7) ellipse (0.3 and 0.15);
		\draw (-2, 0) ellipse (0.3 and 0.15);
		\draw (-5.3, 0) to[out=90, in=-90] (-4.55, 1.7);
		\draw (-4.7, 0) arc (180:0:0.45);
		\draw (-3.8, 0) arc (-180:0:1.05);
		\draw (-2.3, 0) arc (0:-180:0.45);
		\draw (-3.2, 0) to[out=90, in=-90] (-3.95, 1.7);
		\draw[thin, blue, dashed] (-4.1, 0.3) rectangle (-1.5, -1.25);
		\draw (-5, -0.5) node {$S^1_1$};
		\draw (-4.25, 2.2) node {$S^1_0$};
		\draw (-2, 0.5) node {$S^1_0$};
		\draw[orange, fill=orange] (-4.25, 0.45) circle (2pt);
		\draw[fill=black] (-2.75, -0.45) circle (2pt);
		\draw[orange, fill=orange] (-2.75, -1.05) circle (2pt);
		
		\draw[->, decorate, decoration={snake,amplitude=1pt, segment length=5pt}, -stealth'] (-1.5, 0.85) -- (-0.6, 0.85);
		\draw[->, decorate, decoration={snake,amplitude=1pt, segment length=5pt}, -stealth'] (2.1, 0.85) -- (3, 0.85);
		
		\draw (0.75, 0) [partial ellipse = -180:0:0.3 and 0.15];
		\draw (0.75, 0) [dashed, partial ellipse = 180:0:0.3 and 0.15];
		\draw (0, 1.7) ellipse (0.3 and 0.15);
		\draw (1.5, 1.7) ellipse (0.3 and 0.15);
		\draw (0.45, 0) to[out=90, in=-90] (-0.3, 1.7);
		\draw (1.05, 0) to[out=90, in=-90] (1.8, 1.7);
		\draw (0.3, 1.7) arc (-180:0:0.45);
		\draw (0.75, -0.5) node {$S^1_1$};
		\draw (0, 2.2) node {$S^1_0$};
		\draw (1.5, 2.2) node {$S^1_0$};
		\draw[fill=black] (0.75, 1.25) circle (2pt);
		
		\draw (3.5, 0) ellipse (0.3 and 0.15);
		\draw (5.75, 1.7) ellipse (0.3 and 0.15);
		\draw (6.5, 0) [partial ellipse = -180:0:0.3 and 0.15];
		\draw (6.5, 0) [dashed, partial ellipse = 180:0:0.3 and 0.15];
		\draw (6.8, 0) to[out=90, in=-90] (6.05, 1.7);
		\draw (6.2, 0) arc (0:180:0.45);
		\draw (5.3, 0) arc (0:-180:1.05);
		\draw (3.8, 0) arc (-180:0:0.45);
		\draw (4.7, 0) to[out=90, in=-90] (5.45, 1.7);
		\draw[thin, blue, dashed] (5.6, 0.3) rectangle (3, -1.25);
		\draw (3.5, 0.5) node {$S^1_0$};
		\draw (5.75, 2.2) node {$S^1_0$};
		\draw (6.5, -0.5) node {$S^1_1$};
		\draw[orange, fill=orange] (5.75, 0.45) circle (2pt);
		\draw[fill=black] (4.25, -0.45) circle (2pt);
		\draw[orange, fill=orange] (4.25, -1.05) circle (2pt);
	\end{tikzpicture}
	\caption{The homotopy $\vee'$.}
	\label{fig:TFThomotopy}
\end{figure}
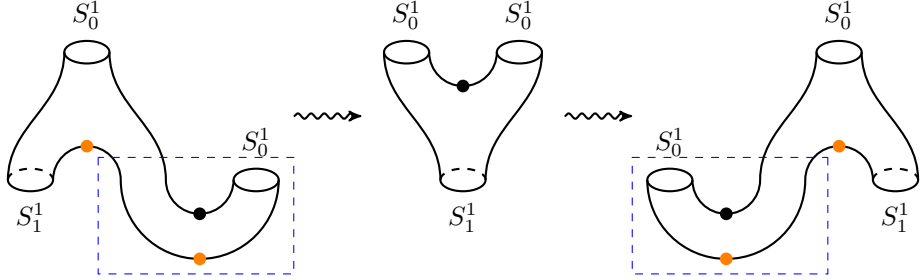

In addition, there is a homotopy between 2-bordisms given by \cref{fig:TFThomotopy} witnessing the Frobenius property. After applying $Z$ it gives rise to the homotopies
\begin{equation}
\vee'\colon ((-)\otimes 1)\wedge \bbDelta(1)\sim \bbDelta(-)\sim \bbDelta(1)\wedge (1\otimes (-))
\label{eq:TFThomotopy}
\end{equation}
of maps $Z(S^1_1)\rightarrow Z(S^1_0)\otimes Z(S^1_0)$. We will explain a connection between this homotopy and the loop coproduct from \cref{def:stringcoproduct} in \cref{thm:TFTHWequivalence}.

\subsection{Explicit description of the operations}
\label{sect:TFTexplicit}

One can give a generators-and-relations presentation of the bordism category $\Bord_2^{\fr, \nc}$ by equipping bordisms with Morse functions. Generators will then be given by bordisms equipped with Morse functions with a unique critical point in the interior. As an example of such approach, we refer to \cite[Theorem 7.1]{Pstragowski} for a presentation of the underlying $2$-category of $\Bord_2^{\fr}$ and \cite[Definition 4.2.1]{Araujo} for a presentation of the underlying $(3, 2)$-category of $\Bord_2^{\fr}$.

All 2-bordisms we draw are equipped with a natural Morse function given by the height. In this section we split the bordisms given in \cref{fig:TFThomotopy} into elementary ones to provide a definition of the TQFT operations independent of the cobordism hypothesis given in \cref{thm:noncompactcobordismhyp}. We begin with a description of the duality data for $\pt\in\Bord_2^{\fr, \nc}$. Note that by \cite[Proposition 4.2.3]{LurieCobordism} we have
\[\coev^\R = \ev\circ\sigma\circ(\id\otimes T).\]
In particular, $\coev^\R$ coincides with $\ev$ up to a twist in the framing. Taking the dual of both sides, we obtain
\[\ev^\L = (T\otimes \id)\circ\sigma\circ\coev.\]

The following is shown in \cite[Section 9.3]{CalaqueScheimbauer}; we will not describe the framings on the bordisms and refer to \cite{CalaqueScheimbauer} for a detailed description.

\begin{prop}
The point $\pt\in\Bord_2^{\fr, \nc}$ is a smooth object. The evaluation and coevaluation are shown in \cref{fig:TFTcoevev}. The counit and unit of the adjunction $\coev\dashv \coev^\R$ are shown in \cref{fig:TFTcoevadjunction}.
\end{prop}

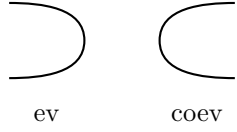
\begin{figure}
	\begin{tikzpicture}[thick]
		\draw (0, 0) [out=0, in=90] to (1, -0.5) [out=-90, in=0] to (0, -1);
		\draw (3, 0) [out=180, in=90] to (2, -0.5) [out=-90, in=-180] to (3, -1);
		\draw (0.5, -1.5) node {$\ev$};
		\draw (2.5, -1.5) node {$\coev$};
	\end{tikzpicture}
	\caption{Evaluation and coevaluation.}
	\label{fig:TFTcoevev}
\end{figure}

\begin{figure}
	\begin{tikzpicture}[thick]
		\draw (0, 0) -- (0, 2) -- (2, 2) -- (2, 0);
		\draw (0, 0.3) [partial ellipse = -90:0:0.4 and 0.3];
		\draw (-0.3, 0.3) [dashed, partial ellipse = 64:0:0.7 and 0.3];
		\draw (-0.3, 0.3) [partial ellipse = 90:64:0.7 and 0.3];
		\draw (2, 0.3) [partial ellipse = 180:270:0.7 and 0.3];
		\draw (1.7, 0.3) [dashed, partial ellipse = 180:90:0.4 and 0.3];
		\draw (0.85, 0.3) [partial ellipse = 180:0:0.45 and 0.6];
		\draw[dashed] (1.7, 0.6) -- (1.7, 2);
		\draw (1.7, 2) -- (1.7, 2.6) -- (-0.3, 2.6) -- (-0.3, 0.6);
		\draw (0.85, -0.3) node {$\epsilon_p$};
		
		\draw (6, 1.7) ellipse (0.7 and 0.35);
		\draw (6, 1.7) [partial ellipse = -180:0:0.7 and 1.4];
		\draw (6, -0.3) node {$\eta$};
	\end{tikzpicture}
	\caption{Adjunction $\coev\dashv \coev^\R$.}
	\label{fig:TFTcoevadjunction}
\end{figure}
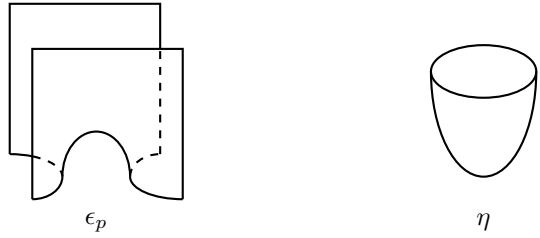

Using the symmetric monoidal functor $Z\colon \Bord_2^{\fr, \nc}\rightarrow \cA$, we obtain the structure of a smooth object on $x=Z(\pt)\in\cA$. We are now ready to describe the TQFT operations. The simplest operation is the unit on $Z(S^1_1)$.

\begin{prop}
	The unit map for $Z(S^1_1)$ is equivalent to the unit
	\[\eta\colon \id_\bu\rightarrow \coev^\R\circ\coev\cong Z(S^1_1)\]
        of the adjunction $\coev\dashv \coev^\R$.
	\label{prop:TFTunit}
\end{prop}

Next, we consider the multiplication on $Z(S^1_1)$.

\begin{prop}
	\label{def:tftproduct}
	The product $\wedge\colon Z(S^1_1)\otimes Z(S^1_1)\rightarrow Z(S^1_1)$ is equivalent to the composite
	\[Z(S^1_1)\otimes Z(S^1_1)\cong \coev^\R\circ\coev\circ \coev^\R\circ\coev\xrightarrow{\id\circ \epsilon\circ \id} \coev^\R\circ \coev\cong Z(S^1_1).\]
	Similarly, the product $\wedge\colon Z(S^1_1)\otimes Z(S^1_0)\rightarrow Z(S^1_0)$ is equivalent to the composite
	\[Z(S^1_0)\otimes Z(S^1_1)\cong \ev\circ\coev\circ\coev^\R\circ\coev\xrightarrow{\id\circ \epsilon\circ \id} \ev\circ \coev\cong Z(S^1_0).\]
	\label{prop:TFTmultiplication}
\end{prop}
\begin{proof}
	The cobordism representing the product $S^1_1\sqcup S^1_1\rightarrow S^1_1$ can be decomposed as shown in \cref{fig:TFTmultiplication} which implies the claim.
	
	\begin{figure}
		\begin{tikzpicture}[thick]
			\draw (-0.9, 2.3) [partial ellipse = 90:180:0.4 and 0.3];
			\draw (-0.6, 2.3) [partial ellipse = 180:270:0.7 and 0.3];
			\draw (-1.3, 2.3) -- (-1.3, 0.3);
			\draw (-0.6, 0.3) [partial ellipse = 180:270:0.7 and 0.3];
			\draw (-0.9, 0.3) [dashed, partial ellipse = 90:180:0.4 and 0.3];
			\draw (-0.6, 2) -- (-0.6, 0);
			\draw (-0.9, 2.6) -- (-0.9, 2);
			\draw[dashed] (-0.9, 2) -- (-0.9, 0.6);
			\draw (-1, -0.4) node {$\coev$};
			
			\draw (0, 0) -- (0, 2) -- (2, 2) -- (2, 0);
			\draw (0, 0.3) [partial ellipse = -90:0:0.4 and 0.3];
			\draw (-0.3, 0.3) [dashed, partial ellipse = 64:0:0.7 and 0.3];
			\draw (-0.3, 0.3) [partial ellipse = 90:64:0.7 and 0.3];
			\draw (2, 0.3) [partial ellipse = 180:270:0.7 and 0.3];
			\draw (1.7, 0.3) [dashed, partial ellipse = 180:90:0.4 and 0.3];
			\draw (0.85, 0.3) [partial ellipse = 180:0:0.45 and 0.6];
			\draw[dashed] (1.7, 0.6) -- (1.7, 2);
			\draw (1.7, 2) -- (1.7, 2.6) -- (-0.3, 2.6) -- (-0.3, 0.6);
			\draw (0, -0.3) node {$\coev^\R$};
			\draw (1.6, -0.4) node {$\coev$};
			
			\draw (2.4, 2.3) [partial ellipse = 90:0:0.7 and 0.3];
			\draw (2.7, 2.3) [partial ellipse = 0:-90:0.4 and 0.3];
			\draw (2.7, 0.3) [partial ellipse = 0:-90:0.4 and 0.3];
			\draw (2.4, 0.3) [dashed, partial ellipse = 0:64:0.7 and 0.3];
			\draw (2.4, 0.3) [partial ellipse = 90:64:0.7 and 0.3];
			\draw (3.1, 2.3) -- (3.1, 0.3);
			\draw (2.7, 2) -- (2.7, 0);
			\draw (2.4, 2.6) -- (2.4, 0.6);
			\draw (2.8, -0.3) node {$\coev^\R$};
			
			\draw (3.6, 1.3) node {$=$};
			
			\draw (4.5, 2.3) [partial ellipse = 90:180:0.4 and 0.3];
			\draw (4.8, 2.3) [partial ellipse = 180:270:0.7 and 0.3];
			\draw (4.8, 2) -- (6.8, 2);
			\draw (6.8, 2.3) [partial ellipse = -90:0:0.4 and 0.3];
			\draw (6.5, 2.3) [partial ellipse = 0:90:0.7 and 0.3];
			\draw (6.5, 2.6) -- (4.5, 2.6);
			\draw (4.5, 0.3) [dashed, partial ellipse = 90:180:0.4 and 0.3];
			\draw (4.8, 0.3) [partial ellipse = 180:270:0.7 and 0.3];
			\draw (4.8, 0.3) [partial ellipse = -90:0:0.4 and 0.3];
			\draw (4.5, 0.3) [dashed, partial ellipse = 0:90:0.7 and 0.3];
			\draw (6.8, 0.3) [partial ellipse = -90:0:0.4 and 0.3];
			\draw (6.8, 0.3) [partial ellipse = 270:180:0.7 and 0.3];
			\draw (6.5, 0.3) [dashed, partial ellipse = 0:90:0.7 and 0.3];
			\draw (6.5, 0.3) [dashed, partial ellipse = 180:90:0.4 and 0.3];
			\draw (5.65, 0.3) [partial ellipse = 180:0:0.45 and 0.6];
			\draw (7.2, 2.3) -- (7.2, 0.3);
			\draw (4.1, 2.3) -- (4.1, 0.3);
			\draw (4.6, -0.3) node {$S^1_1$};
			\draw (6.7, -0.3) node {$S^1_1$};
			\draw (5.8, 2.9) node {$S^1_1$};
		\end{tikzpicture}
		\caption{Decomposition of the product.}
		\label{fig:TFTmultiplication}
	\end{figure}
\end{proof}

\begin{prop}
	The coproduct $\bbDelta\colon Z(S^1_1)\rightarrow Z(S^1_0)\otimes Z(S^1_0)$ is equivalent to
	\begin{align*}
		Z(S^1_1)&\cong \ev\circ \ev^\L \\
		&\cong (\ev\otimes \ev)\circ(\id\otimes \coev\circ\coev^\R\otimes \id)\circ(\coev\otimes \coev) \\
		&\xrightarrow{\epsilon} (\ev\circ\coev)\otimes (\ev\circ\coev) \\
		&\cong Z(S^1_0)\otimes Z(S^1_0).
	\end{align*}
	\label{prop:TFTcomultiplication}
\end{prop}
\begin{proof}
	The cobordism representing the coproduct $S^1_1\rightarrow S^1_0\sqcup S^1_0$ can be decomposed as shown in \cref{fig:TFTcomultiplication}, which implies the claim.
	\begin{figure}
		\begin{tikzpicture}[thick]
			\draw (0, 2) -- (2, 2);
			\draw (-0.3, 2.6) -- (1.7, 2.6);
			\draw (-0.2, 1.7) [partial ellipse = 90:180:0.4 and 0.3];
			\draw (0.1, 1.7) [partial ellipse = 180:270:0.7 and 0.3];
			\draw (0.1, 1.4) -- (2.5, 1.4);
			\draw (2.5, 1.7) [partial ellipse = -90:0:0.4 and 0.3];
			\draw (2.2, 1.7) [partial ellipse = 0:90:0.7 and 0.3];
			\draw (-0.6, 1.7) -- (-0.6, -0.3);
			\draw (2.9, 1.7) -- (2.9, -0.3);
			\draw (0.1, -0.3) [partial ellipse = 180:270:0.7 and 0.3];
			\draw (-0.2, -0.3) [dashed, partial ellipse = 90:180:0.4 and 0.3];
			\draw (0, 0.3) [dashed, partial ellipse = -90:0:0.4 and 0.3];
			\draw (-0.3, 0.3) [dashed, partial ellipse = 90:0:0.7 and 0.3];
			\draw (0.1, -0.6) -- (2.5, -0.6);
			\draw (2.5, -0.3) [partial ellipse = -90:0:0.4 and 0.3];
			\draw (2.2, -0.3) [dashed, partial ellipse = 0:90:0.7 and 0.3];
			\draw (2, 0.3) [dashed, partial ellipse = 180:270:0.7 and 0.3];
			\draw (1.7, 0.3) [dashed, partial ellipse = 180:90:0.4 and 0.3];
			\draw (0.85, 0.3) [dashed, partial ellipse = 180:0:0.45 and 0.6];
			\draw (-0.2, 2) -- (-0.2, 1.4);
			\draw[dashed] (-0.2, 1.4) -- (-0.2, 0);
			\draw (0, 2) -- (0, 1.4);
			\draw[dashed] (0, 1.4) -- (0, 0);
			\draw (2.2, 2) -- (2.2, 1.4);
			\draw[dashed] (2.2, 1.4) -- (2.2, 0);
			\draw (2, 2) -- (2, 1.4);
			\draw[dashed] (2, 1.4) -- (2, 0);
			\draw (-0.3, 2.6) -- (-0.3, 2);
			\draw[dashed] (-0.3, 2) -- (-0.3, 0.6);
			\draw (1.7, 2.6) -- (1.7, 2);
			\draw[dashed] (1.7, 2) -- (1.7, 0.6);
			\draw (-0.5, 2.9) [partial ellipse = 270:180:0.7 and 0.3];
			\draw (-0.8, 2.9) [partial ellipse = 180:90:0.4 and 0.3];
			\draw (-0.8, 3.2) -- (1.6, 3.2);
			\draw (1.6, 2.9) [partial ellipse = 90:0:0.7 and 0.3];
			\draw (1.9, 2.9) [partial ellipse = 0:-90:0.4 and 0.3];
			\draw (-1.2, 2.9) -- (-1.2, 0.9);
			\draw (-0.5, 2.6) -- (-0.5, 1.9);
			\draw[dashed] (-0.5, 1.9) -- (-0.5, 0.6);
			\draw (-0.5, 0.9) [partial ellipse = -100:-180:0.7 and 0.3];
			\draw (-0.8, 0.9) [dashed, partial ellipse = 180:90:0.4 and 0.3];
			\draw[dashed] (-0.8, 1.2) -- (1.6, 1.2);
			\draw (2.3, 2.9) -- (2.3, 2);
			\draw[dashed] (2.3, 2) -- (2.3, 0.9);
			\draw (1.9, 0.9) [dashed, partial ellipse = 0:-90:0.4 and 0.3];
			\draw (1.6, 0.9) [dashed, partial ellipse = 90:0:0.7 and 0.3];
			\draw (1.9, 2.6) -- (1.9, 2);
			\draw[dashed] (1.9, 2) -- (1.9, 0.6);
			
			\draw (3.4, 1) node {$=$};
			
			\draw (4.8, 2) -- (6.8, 2);
			\draw (4.5, 2.6) -- (6.5, 2.6);
			\draw (4.8, 1.7) [partial ellipse = 90:180:0.4 and 0.3];
			\draw (5.1, 1.7) [partial ellipse = 180:270:0.7 and 0.3];
			\draw (5.1, 1.4) -- (7.1, 1.4);
			\draw (7.1, 1.7) [partial ellipse = -90:0:0.4 and 0.3];
			\draw (6.8, 1.7) [partial ellipse = 0:90:0.7 and 0.3];
			\draw (4.4, 1.7) -- (4.4, -0.3);
			\draw (7.5, 1.7) -- (7.5, -0.3);
			\draw (5.1, -0.3) [partial ellipse = 180:270:0.7 and 0.3];
			\draw (4.8, -0.3) [dashed, partial ellipse = 90:180:0.4 and 0.3];
			\draw (4.8, 0.3) [dashed, partial ellipse = -90:0:0.4 and 0.3];
			\draw (4.5, 0.3) [dashed, partial ellipse = 90:0:0.7 and 0.3];
			\draw (5.1, -0.6) -- (7.1, -0.6);
			\draw (7.1, -0.3) [partial ellipse = -90:0:0.4 and 0.3];
			\draw (6.8, -0.3) [dashed, partial ellipse = 0:90:0.7 and 0.3];
			\draw (6.8, 0.3) [dashed, partial ellipse = 180:270:0.7 and 0.3];
			\draw (6.5, 0.3) [dashed, partial ellipse = 180:90:0.4 and 0.3];
			\draw (5.65, 0.3) [dashed, partial ellipse = 180:0:0.45 and 0.6];
			\draw (4.5, 2.9) [partial ellipse = 270:180:0.7 and 0.3];
			\draw (4.2, 2.9) [partial ellipse = 180:90:0.4 and 0.3];
			\draw (4.2, 3.2) -- (6.2, 3.2);
			\draw (6.2, 2.9) [partial ellipse = 90:0:0.7 and 0.3];
			\draw (6.5, 2.9) [partial ellipse = 0:-90:0.4 and 0.3];
			\draw (3.8, 2.9) -- (3.8, 0.9);
			\draw (4.5, 0.9) [partial ellipse = -100:-180:0.7 and 0.3];
			\draw (4.2, 0.9) [dashed, partial ellipse = 180:90:0.4 and 0.3];
			\draw[dashed] (4.2, 1.2) -- (6.2, 1.2);
			\draw (6.9, 2.9) -- (6.9, 2);
			\draw[dashed] (6.9, 2) -- (6.9, 0.9);
			\draw (6.5, 0.9) [dashed, partial ellipse = 0:-90:0.4 and 0.3];
			\draw (6.2, 0.9) [dashed, partial ellipse = 90:0:0.7 and 0.3];
			\draw (8, -0.3) node {$S^1_1$};
			\draw (8, 1.7) node {$S^1_0$};
			\draw (7.4, 2.9) node {$S^1_0$};
		\end{tikzpicture}
		\caption{Decomposition of the coproduct.}
		\label{fig:TFTcomultiplication}
	\end{figure}
\end{proof}

\begin{cor}
	The Shklyarov copairing $[\coev]\in Z(S^1_0)\otimes Z(S^1_0)$ is equivalent to $\bbDelta(1)$.
	\label{cor:Mukaicopairingelbow}
\end{cor}
\begin{proof}
	By \cref{def:transfermap} $[\coev]$ is computed by traversing the diagram
	\[
	\xymatrix@C=0.5cm@R=1.5cm{
		& \bu \ar_{\coev_x}[dr] \ar@{=}[rr] & \ar@{}[d]^(.2){}="a"^(.63){}="b" \ar^{\eta}@{=>} "a";"b" & \bu \ar@{=}[dr] & \\
		\bu \ar@{=}[ur] \ar_{(\id\otimes \coev_{x^\vee}\otimes \id)\circ\coev_x}[dr] && x\otimes x^\vee \ar_{\coev_x^\R}[ur] \ar^{\id_{x\otimes x^\vee}\otimes \ev_{x\otimes x^\vee}^\L}[dr] \ar@{}[d]^(.3){}="c"^(.73){}="d" \ar^{\epsilon}@{=>} "c";"d" && \bu \\
		& x\otimes x^\vee\otimes x^{\vee\vee}\otimes x^\vee \ar^-{\id_{x\otimes x^\vee}\otimes \ev_{x^\vee}}[ur] \ar@{=}[rr] && x\otimes x^\vee\otimes x^{\vee\vee}\otimes x^\vee \ar_{\ev_x\circ(\id\otimes \ev_{x^\vee}\otimes \id)}[ur]
	}
	\]
	
	By \cref{prop:TFTcomultiplication} the diagram without $\eta$ is equivalent to the coproduct $\bbDelta\colon Z(S^1_1)\rightarrow Z(S^1_0)\otimes Z(S^1_0)$. By \cref{prop:TFTunit} the unit of $Z(S^1_1)$ is $\eta$.
\end{proof}

Finally, we describe the homotopy $((-)\otimes 1)\wedge \bbDelta(1)\sim \bbDelta(-)$ given by the left part of \cref{fig:TFThomotopy}. It will be convenient to present the 2-bordisms in terms of sequences of horizontal slices, where the shown 2-morphism is understood to be acting on the green dashed area.

\begin{prop}
	The homotopy $((-)\otimes 1)\wedge \bbDelta(1)\sim \bbDelta(-)$ is equivalent to the composite of the following sequence of homotopies.
	\begin{enumerate}
		\item Using \cref{prop:TFTunit,prop:TFTmultiplication,prop:TFTcomultiplication} the 2-morphism $((-)\otimes 1)\wedge \bbDelta(1)$ is given by the composite
		
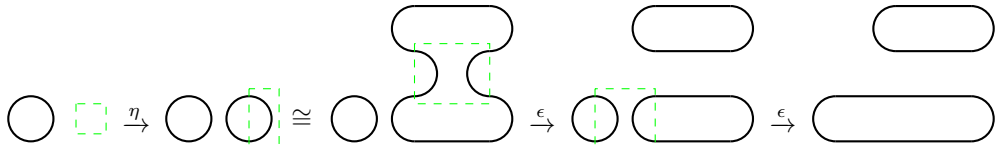
\[
		\begin{tikzpicture}[thick]
			\draw (0, 0) circle (0.3cm);
			\draw[thin, green, dashed] (0.6, 0.2) rectangle (1, -0.2);
			
			\draw (1.4, 0) node {$\xrightarrow{\eta}$};
			
			\draw (2.1, 0) circle (0.3cm);
			\draw (2.9, 0) circle (0.3cm);
			\draw[thin, green, dashed] (2.9, 0.4) rectangle (3.3, -0.4);
			
			\draw (3.6, 0) node {$\cong$};
			
			\draw (4.3, 0) circle (0.3cm);
			\draw (5.1, -0.3) arc (270:90:0.3cm);
			\draw (5.1, 0.3) arc (-90:90:0.3cm);
			\draw (5.1, 0.9) arc (270:90:0.3cm);
			\draw (5.1, 1.5) -- (6.1, 1.5);
			\draw (6.1, 1.5) arc (90:-90:0.3cm);
			\draw (6.1, 0.9) arc (90:270:0.3cm);
			\draw (6.1, 0.3) arc (90:-90:0.3cm);
			\draw (6.1, -0.3) -- (5.1, -0.3);
			\draw[thin, green, dashed] (5.1, 1) rectangle (6.1, 0.2);
			
			\draw (6.8, 0) node {$\xrightarrow{\epsilon}$};
			
			\draw (7.5, 0) circle (0.3cm);
			\draw (8.3, -0.3) arc (270:90:0.3cm);
			\draw (8.3, 0.3) -- (9.3, 0.3);
			\draw (8.3, 0.9) arc (270:90:0.3cm);
			\draw (8.3, 1.5) -- (9.3, 1.5);
			\draw (9.3, 1.5) arc (90:-90:0.3cm);
			\draw (9.3, 0.3) arc (90:-90:0.3cm);
			\draw (9.3, -0.3) -- (8.3, -0.3);
			\draw (8.3, 0.9) -- (9.3, 0.9);
			\draw[thin, green, dashed] (7.5, 0.4) rectangle (8.3, -0.4);
			
			\draw (10, 0) node {$\xrightarrow{\epsilon}$};
			
			\draw (10.7, 0.3) arc (90:270:0.3cm);
			\draw (10.7, 0.3) -- (12.5, 0.3);
			\draw (11.5, 0.9) arc (270:90:0.3cm);
			\draw (11.5, 1.5) -- (12.5, 1.5);
			\draw (12.5, 1.5) arc (90:-90:0.3cm);
			\draw (12.5, 0.3) arc (90:-90:0.3cm);
			\draw (12.5, -0.3) -- (10.7, -0.3);
			\draw (11.5, 0.9) -- (12.5, 0.9);
		\end{tikzpicture}
		\]
		\item
		Exchanging the last $\epsilon$ with the composite of the isomorphism and the first $\epsilon$ we obtain
		
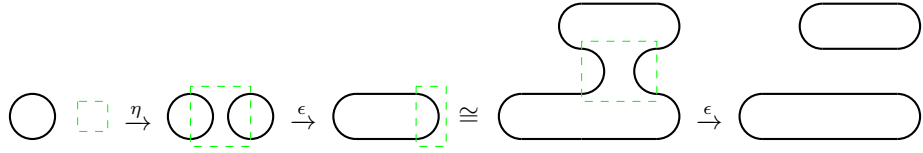
\[
		\begin{tikzpicture}[thick]
			\draw (0, 0) circle (0.3cm);
			\draw[thin, green, dashed] (0.6, 0.2) rectangle (1, -0.2);
			
			\draw (1.4, 0) node {$\xrightarrow{\eta}$};
			
			\draw (2.1, 0) circle (0.3cm);
			\draw (2.9, 0) circle (0.3cm);
			\draw[thin, green, dashed] (2.1, 0.4) rectangle (2.9, -0.4);
			
			\draw (3.6, 0) node {$\xrightarrow{\epsilon}$};
			
			\draw (4.3, 0.3) arc (90:270:0.3cm);
			\draw (4.3, 0.3) -- (5.1, 0.3);
			\draw (5.1, 0.3) arc (90:-90:0.3cm);
			\draw (5.1, -0.3) -- (4.3, -0.3);
			\draw[thin, green, dashed] (5.1, 0.4) rectangle (5.5, -0.4);
			
			\draw (5.8, 0) node {$\cong$};
			
			\draw (6.5, 0.3) arc (90:270:0.3cm);
			\draw (6.5, 0.3) -- (7.3, 0.3);
			\draw (6.5, -0.3) -- (7.3, -0.3);
			\draw (7.3, 0.3) arc (-90:90:0.3cm);
			\draw (7.3, 0.9) arc (270:90:0.3cm);
			\draw (7.3, 1.5) -- (8.3, 1.5);
			\draw (8.3, 1.5) arc (90:-90:0.3cm);
			\draw (8.3, 0.9) arc (90:270:0.3cm);
			\draw (8.3, 0.3) arc (90:-90:0.3cm);
			\draw (8.3, -0.3) -- (7.3, -0.3);
			\draw[thin, green, dashed] (7.3, 1) rectangle (8.3, 0.2);
			
			\draw (9, 0) node {$\xrightarrow{\epsilon}$};
			
			\draw (9.7, 0.3) arc (90:270:0.3cm);
			\draw (9.7, 0.3) -- (11.5, 0.3);
			\draw (10.5, 0.9) arc (270:90:0.3cm);
			\draw (10.5, 1.5) -- (11.5, 1.5);
			\draw (11.5, 1.5) arc (90:-90:0.3cm);
			\draw (11.5, 0.3) arc (90:-90:0.3cm);
			\draw (11.5, -0.3) -- (9.7, -0.3);
			\draw (10.5, 0.9) -- (11.5, 0.9);
		\end{tikzpicture}
		\]
		\item Applying the cusp 3-isomorphism witnessing the adjunction axioms for $\coev\dashv \coev^\R$, we cancel the composite of the first two 2-morphisms and obtain $\bbDelta(-)$:
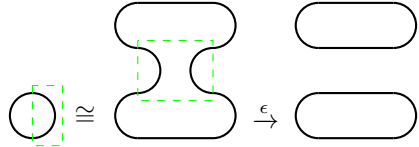
		\[
		\begin{tikzpicture}[thick]
			\draw (0, 0) circle (0.3cm);
			\draw[thin, green, dashed] (0, 0.4) rectangle (0.4, -0.4);
			
			\draw (0.7, 0) node {$\cong$};
			
			\draw (1.4, -0.3) arc (270:90:0.3cm);
			\draw (1.4, 0.3) arc (-90:90:0.3cm);
			\draw (1.4, 0.9) arc (270:90:0.3cm);
			\draw (1.4, 1.5) -- (2.4, 1.5);
			\draw (2.4, 1.5) arc (90:-90:0.3cm);
			\draw (2.4, 0.9) arc (90:270:0.3cm);
			\draw (2.4, 0.3) arc (90:-90:0.3cm);
			\draw (2.4, -0.3) -- (1.4, -0.3);
			\draw[thin, green, dashed] (1.4, 1) rectangle (2.4, 0.2);
			
			\draw (3.1, 0) node {$\xrightarrow{\epsilon}$};
			
			\draw (3.8, -0.3) arc (270:90:0.3cm);
			\draw (3.8, 0.3) -- (4.8, 0.3);
			\draw (3.8, 0.9) arc (270:90:0.3cm);
			\draw (3.8, 1.5) -- (4.8, 1.5);
			\draw (4.8, 1.5) arc (90:-90:0.3cm);
			\draw (3.8, 0.9) -- (4.8, 0.9);
			\draw (4.8, 0.3) arc (90:-90:0.3cm);
			\draw (4.8, -0.3) -- (3.8, -0.3);
		\end{tikzpicture}
		\]
	\end{enumerate}
\end{prop}
\begin{proof}
	The homotopy shown in \cref{fig:TFThomotopy} consists of the following two operations:
	\begin{enumerate}
		\item The black critical point exchanges height with the leftmost orange critical point.
		\item The two orange critical points collide and disappear, passing through a birth-death singularity.
	\end{enumerate}
	The first operation corresponds to an interchange 3-isomorphism. The second operation corresponds to the cusp 3-isomorphism witnessing the adjunction $\coev\dashv \coev^\R$ (see e.g. the proof of \cite[Lemma 8.5]{Pstragowski}).
\end{proof}

Combining the homotopy $\vee'\colon ((-)\otimes 1)\wedge \bbDelta(1)\sim \bbDelta(-)\sim \bbDelta(1)\wedge (1\otimes (-))$ given in \eqref{eq:TFThomotopy} with a nullhomotopy of $\bbDelta(1)$ we obtain a secondary operation $\fZ(x)\rightarrow \dim(x)\otimes \dim(x)[-1]$. In examples we will be interested in the cases when $\bbDelta(1)$ is not trivialized in $\dim(x)$, but only in some quotient. For this, assume that $\End_\cA(\bu)$ is a stable $\infty$-category. Consider a morphism $\widetilde{\dim}(x)\rightarrow \dim(x)$ and assume that $\bbDelta(1)$ lifts along $\widetilde{\dim}(x)\otimes \widetilde{\dim}(x)$. In this case the endpoints of the homotopy $\vee'$ lie in $\dim(x)\otimes \widetilde{\dim}(x)$ and $\widetilde{\dim}(x)\otimes \dim(x)$, respectively. Therefore, $\vee'$ gives rise to a secondary operation in the quotient.

\begin{defn}
Let $x\in\cA$ be a smooth object together with a morphism $\widetilde{\dim}(x)\rightarrow \dim(x)$ whose cofiber in the stable $\infty$-category $\End_\cA(\bu)$ is $\dim(x)\rightarrow \overline{\dim}(x)$. Suppose that the Shklyarov copairing $[\coev]$ admits a lift along the morphism $\widetilde{\dim}(x)\otimes \widetilde{\dim}(x)\rightarrow \dim(x)\otimes \dim(x)$. The \defterm{secondary coproduct}
\[\vee\colon \fZ(x)\longrightarrow \overline{\dim}(x)\otimes \overline{\dim}(x)[-1]\]
is given by composing the homotopy $\vee'$ with the projection $\dim(x)\otimes \dim(x)\rightarrow \overline{\dim}(x)\otimes \overline{\dim}(x)$.
\label{def:secondarycoproduct}
\end{defn}

\section{String topology TQFT}

Let $M$ be a closed oriented $d$-manifold. In this section we denote by the same letter its underlying homotopy type. We use the same notation as in \cref{sect:stringtopology} and replace the smooth category $\Sp^M$ by $\LocSys(M) = \Fun(\Sing(M), \Mod_\Z)$ for concreteness. So, we obtain the TQFT operations described in \cref{sect:TFToperations}. The goal of this section is to show that they are equivalent to the string topology operations from \cref{sect:stringtopology}.

\subsection{Marked spaces}

In this section we introduce a convenient pictorial notation for morphisms of local systems that will be useful in the future sections.

Fix a space $X\in\cS$. Consider the symmetric monoidal $(\infty, 2)$-category $\coCorr(\cS)$ of cocorrespondences in spaces (see \cite[Chapter 7]{GaitsgoryRozenblyum1} for the construction) which has the following informal description:
\begin{itemize}
	\item Its objects are spaces.
	\item Its 1-morphisms from $S$ to $T$ are cocorrespondences $S\rightarrow C\leftarrow T$ of spaces.
	\item Its 2-morphisms from $S\rightarrow C_1\leftarrow T$ to $S\rightarrow C_2\leftarrow T$ are given by commutative diagrams
	
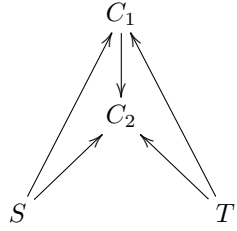
\[
	\xymatrix{
		& C_1 \ar[d]& \\
		& C_2 & \\
		S \ar[uur] \ar[ur] && T \ar[ul] \ar[uul]
	}
	\]
\end{itemize}
The symmetric monoidal structure is given by the disjoint union of spaces. We have a natural symmetric monoidal functor
\[\Map(-, X)\colon \coCorr(\cS)^{2\op}\longrightarrow \Corr(\cS)\]
to the $(\infty, 2)$-category of correspondences, where we recall that the symmetric monoidal structure is given by the Cartesian product. Recall that in \cref{thm:Spbivariant} we have also introduced the symmetric monoidal functor
\[\LocSys_\sharp^*\colon \Corr(\cS)\longrightarrow \PrSt_k.\]
So, we obtain the composite
\begin{equation}
	\coCorr(\cS)^{2\op}\xrightarrow{\Map(-, X)} \Corr(\cS)\xrightarrow{\LocSys_\sharp^*} \PrSt_\Z.
	\label{eq:coCorrcomposite}
\end{equation}
By construction
\[\Hom_{\coCorr(\cS)}(\varnothing, \varnothing) \cong \cS,\qquad \Hom_{\PrSt_\Z}(\Mod_\Z, \Mod_\Z)\cong \Mod_\Z\]
and the induced composite on units is
\begin{equation}
	\cS^{\op}\xrightarrow{\Map(-, X)}\cS\xrightarrow{\C_\bullet(-)} \Mod_\Z.
	\label{eq:coCorrSpcomposite}
\end{equation}

It is easy to see that in $\coCorr(\cS)^{2\op}$ we have an adjunction
\begin{equation}
	(\pt\rightarrow \pt\leftarrow \varnothing)\dashv(\varnothing\rightarrow \pt\leftarrow \pt)
	\label{eq:coCorrleftadjoint}
\end{equation}
which after the application of the composite \eqref{eq:coCorrcomposite} corresponds to the adjunction
\[p_\sharp\dashv p^*.\]
We will later see that working with pictures in $\cS^{\op}$ is much easier than working the corresponding spectra obtained via \eqref{eq:coCorrSpcomposite}.

Let us now assume that $X$ is finitely dominated. By \cref{prop:finitetypeproper} the pullback functor $p^*\colon \Mod_\Z\rightarrow \LocSys(X)$ admits a colimit-preserving right adjoint $p_*\colon \LocSys(X)\rightarrow \Mod_\Z$ given by $p_*=p_\sharp(\zeta_X\otimes(-))$ for some $\zeta_X\in\LocSys(X)$. It can be interpreted in $\coCorr(\cS)^{2\op}$ by formally adjoining a \emph{right} adjoint to $(\varnothing\rightarrow \pt\leftarrow \pt)$ which we denote by $(\pt\rightarrow \begin{tikzpicture}\draw[red, fill=red] (0, 0) circle (2pt); \end{tikzpicture}\leftarrow \varnothing)$. In other words, we allow marking points on the space. For instance, the picture on \cref{fig:markedspaceexample} represents $\zeta_X\otimes \Delta^*\Delta_\sharp\Z_X\cong \Delta^*\Delta_\sharp\zeta_X$.

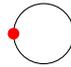
\begin{figure}[ht]
	\begin{tikzpicture}
		\draw(0, 0) circle (0.4cm);
		\draw[red, fill=red] (-0.4, 0) circle (2pt);
	\end{tikzpicture}
	\caption{Example of a marked space.}
	\label{fig:markedspaceexample}
\end{figure}

The adjunction data for $\Delta_\sharp\dashv \Delta^*$ and $p^*\dashv p_\sharp(\zeta_\sharp\otimes(-))$ has the following pictorial representation (see \cref{fig:markedadjunctions}, where dashed lines denote arbitrary spaces):
\begin{itemize}
	\item The unit $\eta_\Delta\colon \cF\rightarrow \Delta^*\Delta_\sharp\cF$ of the $\Delta$ adjunction attaches a loop at a given point.
	
	\item The counit $\epsilon_\Delta\colon \Delta_\sharp\Delta^*\cE\rightarrow \cE$ of the $\Delta$ adjunction breaks an edge.
	
	\item The unit $\eta_p\colon \Z\rightarrow p_\sharp \zeta_X$ of the $p$ adjunction attaches a disjoint basepoint marked by $\zeta_X$.
	
	\item The counit $\epsilon_p\colon \Z_X\otimes p_\sharp(\zeta_X\otimes \cF)\rightarrow \cF$ of the $p$ adjunction connects a marked point to an unmarked point.
\end{itemize}

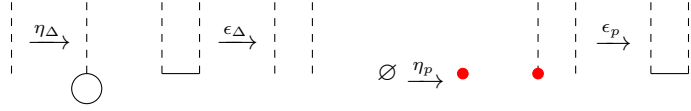
\begin{figure}[ht]
	\begin{tikzpicture}
		\draw[dashed] (0, 0) -- (0, 1);
		\draw (0.5, 0.5) node {$\xrightarrow{\eta_\Delta}$};
		\draw[dashed] (1, 0) -- (1, 1);
		\draw (1, -0.2) circle (0.2cm);
		
		\draw[dashed] (2, 0) -- (2, 1);
		\draw (2, 0) -- (2.5, 0);
		\draw[dashed] (2.5, 0) -- (2.5, 1);
		\draw (3, 0.5) node {$\xrightarrow{\epsilon_\Delta}$};
		\draw[dashed] (3.5, 0) -- (3.5, 1);
		\draw[dashed] (4, 0) -- (4, 1);
		
		\draw (5, 0) node {$\varnothing$};
		\draw (5.5, 0) node {$\xrightarrow{\eta_p}$};
		\draw[red, fill=red] (6, 0) circle (2pt);
		
		\draw[dashed] (7, 0) -- (7, 1);
		\draw[red, fill=red] (7, 0) circle (2pt);
		\draw[dashed] (7.5, 0) -- (7.5, 1);
		\draw (8, 0.5) node {$\xrightarrow{\epsilon_p}$};
		\draw[dashed] (8.5, 0) -- (8.5, 1);
		\draw (8.5, 0) -- (9, 0);
		\draw[dashed] (9, 0) -- (9, 1);
	\end{tikzpicture}
	\caption{Units and counits of $\Delta_\sharp\dashv \Delta^*$ and $p^*\dashv p_\sharp(\zeta_X\otimes(-))$.}
	\label{fig:markedadjunctions}
\end{figure}

\subsection{TQFT from local systems}

Let us explain how string topology operations form a part of a TQFT (this perspective goes back to \cite{CohenGodin}). Throughout this section we return to the case of a closed oriented manifold $M$. As before, we denote by $\Delta\colon M\rightarrow M\times M$ the diagonal map and $p\colon M\rightarrow \pt$ the projection. Denote by $\Delta_{12}\colon M\times M\rightarrow M\times M\times M$ the map $(x, y)\mapsto (x, x, y)$ and similarly for $\Delta_{23}$. Denote by $p_i\colon M\times M\rightarrow M$ the projections on each factor.

Recall that by \cref{cor:finitelydominatedsmooth} the $\infty$-category $\LocSys(M)$ is smooth with the duality data given by
\[\coev(\Z) = \Delta_\sharp \Z_M\in\LocSys(M\times M).\]
Let $\eta_p\colon \Z\rightarrow p_\sharp\Z_M$ and $\epsilon_p\colon \Z_{M\times M}[-d]\rightarrow \Delta_\sharp\Z_M$ be the coevaluation and evaluation of the relative duality $(\Z_M,\Z_M[-d])$. The following immediately follows from \cref{prop:zetapushforward}.

\begin{prop}
The right adjoint to coevaluation $\coev^\R\colon \LocSys(M\times M)\rightarrow \Mod_\Z$ is given by
\[\coev^\R(\cE) = p_\sharp\Delta^*\cE[-d].\]
The unit of the adjunction is
\[\Z\xrightarrow{\eta_p} p_\sharp\Z_M[-d]\xrightarrow{\eta_\Delta} p_\sharp\Delta^*\Delta_\sharp\Z_M[-d].\]
The counit of the adjunction is
\[\Delta_\sharp\Z_M\otimes p_\sharp\Delta^*\cE[-d]\cong \Delta_\sharp (p_1)_\sharp(\Z_M\boxtimes \Delta^*\cE)[-d]\xrightarrow{\epsilon_p} \Delta_\sharp (p_1)_\sharp(\Delta_\sharp\Z_M\otimes (\Z_M\boxtimes \Delta^*\cE))\cong \Delta_\sharp\Delta^*\cE\xrightarrow{\epsilon_\Delta} \cE.\]
\label{prop:LocSyscoevadjoint}
\end{prop}

Recall the inverse Serre morphism from \cref{def:inverseSerre}.

\begin{cor}
The inverse Serre functor $T\colon \LocSys(M)\rightarrow \LocSys(M)$ is given by $\cL\mapsto \cL[-d]$. In particular, it is invertible.
\label{cor:Serreinvertible}
\end{cor}
\begin{proof}
The inverse Serre functor is
\[T(\cL) = (p_2)_\sharp\Delta_{12}^*(\cL\boxtimes \Delta_\sharp\Z_M)[-d].\]
Consider the composition of correspondences
	\[
	\xymatrix{
		& M\times M \ar_{p_1}[dl] \ar^{\Delta_{23}}[dr] && M\times M \ar^{\id}[dr] \ar_{\Delta_{12}}[dl] \\
		M && M\times M\times M && M\times M
	}
	\]
	Then $\cL\mapsto \Delta_{12}^*(\cL\boxtimes \Delta_\sharp \Z_M)$ is given by the composition of $\sharp$-pushforward and $*$-pullback functors along these correspondences. But the composite of the correspondences is $(M\xleftarrow{\id} M\xrightarrow{\Delta} M\times M)$, so we obtain a natural isomorphism
	\[\Delta_{12}^*(\cL\boxtimes \Delta_\sharp \Z_M)\cong \Delta_\sharp \cL.\]
\end{proof}

Since the Serre functor is invertible, by the cobordism hypothesis (see \cref{thm:noncompactcobordismhyp}) we obtain a positive-boundary framed 2-dimensional TQFT
\[Z\colon \Bord_2^{\fr, \nc}\longrightarrow \PrSt_\Z\]
such that $Z(\pt) = \LocSys(M)$.

\begin{remark}
In fact, $\LocSys(M)$ has a left Calabi--Yau structure of dimension $d$ by \cite[Theorem 5.4]{BravDyckerhoff1} and \cite[Section 3.4, Theorem 5]{CohenGanatra}. As sketched in \cite[Remark 4.2.17]{LurieCobordism}, in this case $Z$ descends to a symmetric monoidal functor defined on a twisted version of the $\infty$-category of \emph{oriented} positive-boundary extended bordisms.
\end{remark}

\begin{prop}
For every $n\in\Z$ there is an equivalence
\[Z(S^1_n)\cong \C_\bullet(LM)[-nd].\]
\label{prop:stringS1nvalue}
\end{prop}
\begin{proof}
We have
\begin{align*}
Z(S^1_n)&\cong \ev\circ\sigma\circ (\id\otimes T^n)\circ \coev\\
&\cong p_\sharp \Delta^*\Delta_\sharp \Z_M[-dn]
\end{align*}
where the first equivalence is given by \cref{prop:S1nvalue} and the second equivalence follows from the self-duality of $\LocSys(M)$ constructed in \cref{prop:Spselfdual} and the description of the inverse Serre functor from \cref{cor:Serreinvertible}. The result follows from the base change property associated to the Cartesian diagram
\[
\xymatrix{
LM \ar[r] \ar[d] & M \ar[d] \\
M \ar[r] & M\times M
}
\]
\end{proof}

\begin{remark}
The proof shows that $Z(S^1_n) = \C_\bullet(LM^{-{\T M^{\oplus n}}})$ and the orientation of $M$ is used so get a Thom isomorphism. In particular, without assuming that $M$ is oriented, the various pair-of-pants products $Z(S^1_n) \otimes Z(S^1_m) \to Z(S^1_{n+m-1})$ give operations
    \[
    LM^{-\T M^{\oplus n}} \otimes LM^{-\T M^{\oplus m}} \to LM^{-\T M^{\oplus n+m-1}},
    \]
    that differ by twisting with $\T M$. The natural TQFT secondary coproduct $\vee \colon Z(S^1_1)[1] \to Z(S^1_0) \otimes Z(S^1_0)$
    is then a map
    \[
    \Sigma LM^{-TM} \to \Sigma^\infty LM/M \otimes \Sigma^\infty LM/M.
    \]
\end{remark}

\begin{prop}
The element $\bbDelta(1)\in Z(S^1_0)\otimes Z(S^1_0)$ is equivalent to $(L\Delta) [\Z_M]\in\C_\bullet(LM\times LM)$, where $[\Z_M]\in\C_\bullet(LM)$ is the $\HH$ Euler characteristic and $L\Delta\colon \C_\bullet(LM)\rightarrow \C_\bullet(LM\times LM)$ is the diagonal map on the loop space.
\label{prop:stringelbowvalue}
\end{prop}
\begin{proof}
By \cref{cor:Mukaicopairingelbow} $\bbDelta(1)\in Z(S^1_0)\otimes Z(S^1_0)$ is equivalent to $[\coev]\in\dim(\LocSys(M))\otimes \dim(\LocSys(M))$. Since $\coev(\Z)=\Delta_\sharp \Z_M$, we have $[\coev] = (L\Delta) [\Z_M]$ by \cref{prop:LocSysTHH}.
\end{proof}

\begin{thm}
Under the identification $Z(S^1_n)\cong \C_\bullet(LM)[-nd]$ given by \cref{prop:stringS1nvalue} the loop product $\H_\bullet(LM)\otimes \H_\bullet(LM)\longrightarrow \H_{\bullet-d}(LM)$ is equal on homology to the TQFT product $\wedge\colon Z(S^1_0)\otimes Z(S^1_0)\rightarrow Z(S^1_{-1})$.
\label{thm:TFTstringproduct}
\end{thm}
\begin{proof}
	According to \cref{prop:TFTmultiplication,prop:LocSyscoevadjoint} the TQFT product is given by the composite
	\begin{align*}
		p_\sharp\Delta^*\Delta_\sharp\Z_M\otimes p_\sharp\Delta^*\Delta_\sharp\Z_M &\cong p_\sharp\Delta^*\Delta_\sharp(p_1)_\sharp(\Z_M\boxtimes\Delta^*\Delta_\sharp\Z_M) \\
		&\xrightarrow{\epsilon_p} p_\sharp\Delta^*\Delta_\sharp(p_1)_\sharp(\Delta_\sharp\Z_M\otimes (\Z_M\boxtimes \Delta^*\Delta_\sharp\Z_M))[d] \\
		&\cong p_\sharp\Delta^*\Delta_\sharp(p_1)_\sharp\Delta_\sharp\Delta^*\Delta_\sharp\Z_M[d]
		\\
		&\cong p_\sharp\Delta^*\Delta_\sharp\Delta^*\Delta_\sharp\Z_M[d] \\
		&\xrightarrow{\epsilon_\Delta} p_\sharp\Delta^*\Delta_\sharp \Z_M[d].
	\end{align*}
	
	In terms of marked spaces, it can be represented by maps
	\begin{center}
		\begin{tikzpicture}
			\draw (0, 0) circle (0.4cm);
			\draw (1, 0) circle (0.4cm);
			\draw[red, fill=red] (0.6, 0) circle (2pt);
			\draw (1.8, 0) node {$\xrightarrow{\epsilon_p}$};
			\draw (2.6, 0) circle (0.4cm);
			\draw (3.4, 0) circle (0.4cm);
			\draw (4.2, 0) node {$\xrightarrow{\epsilon_\Delta}$};
			\draw (5, 0) circle (0.4cm);
		\end{tikzpicture}
	\end{center}
	
	The loop product is given by the composite
	\begin{align*}
		p_\sharp\ev_\sharp\Z_{LM}\otimes p_\sharp\ev_\sharp\Z_{LM} &\xrightarrow{\epsilon_p} (p_\sharp\boxtimes p_\sharp)(\Delta_\sharp\Z_M\otimes(\ev_\sharp\Z_{LM}\boxtimes \ev_\sharp\Z_{LM}))[d] \\
		&\cong (p_\sharp\boxtimes p_\sharp)\Delta_\sharp(\ev_\sharp\Z_{LM}\otimes \ev_\sharp\Z_{LM})[d] \\
		&\cong p_\sharp(\ev_\sharp\Z_{LM}\otimes \ev_\sharp\Z_{LM})[d] \\
		&\xrightarrow{m} p_\sharp\ev_\sharp\Z_{LM}[d].
	\end{align*}

	In terms of marked spaces, it can be represented similarly:
	\begin{center}
		\begin{tikzpicture}
			\draw (0, 0) circle (0.4cm);
			\draw (1, 0) circle (0.4cm);
			\draw[red, fill=red] (0.6, 0) circle (2pt);
			\draw (1.8, 0) node {$\xrightarrow{\epsilon_p}$};
			\draw (2.6, 0) circle (0.4cm);
			\draw (3.4, 0) circle (0.4cm);
			\draw (4.2, 0) node {$\xrightarrow{m}$};
			\draw (5, 0) circle (0.4cm);
		\end{tikzpicture}
	\end{center}

	So, we are left with identifying $m$ with $\epsilon_\Delta$ in the second map. Reading the 1-morphisms horizontally, we can represent the counit $\epsilon_\Delta\colon \Delta_\sharp\Delta^*\Rightarrow \id$ and $\Delta^*\Delta_\sharp\Delta^*\Delta_\sharp\Z_M\xrightarrow{\epsilon_\Delta} \Delta^*\Delta_\sharp\Z_M$ as
	\begin{center}
		\begin{tikzpicture}
			\draw (0, 0) arc (-90:90:0.4cm);
			\draw (0.8, 0) arc (270:90:0.4cm);
			\draw (1.4, 0.4) node {$\xrightarrow{\epsilon_\Delta}$};
			\draw (2, 0.8) -- (2.8, 0.8);
			\draw (2, 0) -- (2.8, 0);
			
			\draw (5.2, 0.4) circle (0.4cm);
			\draw (6, 0.4) circle (0.4cm);
			\draw (7, 0.4) node {$\xrightarrow{\epsilon_\Delta}$};
			\draw (8, 0) arc (270:90:0.4cm) -- (8.8, 0.8) arc (90:-90:0.4cm) -- (8, 0);
		\end{tikzpicture}
	\end{center}

	There is a unique homotopy class of based maps on the left, and the induced map on the right is easily seen to be homotopic to $m$.
\end{proof}

Next, let us describe the TQFT origin of the loop coproduct.

\begin{prop}
The coproduct $\bbDelta\colon \C_\bullet(LM)[-d]\rightarrow \C_\bullet(LM)\otimes \C_\bullet(LM)$ is homotopic to the composite
\[\C_\bullet(LM)[-d]\longrightarrow \C_\bullet(\FE)\longrightarrow \C_\bullet(LM)\otimes \C_\bullet(LM),\]
where the first map is given by the intersection product associated to $(\ev_0, \ev_{1/2})\colon LM\rightarrow M\times M$ and the second map is induced by the projection $\FE\rightarrow LM\times LM$.
\label{prop:TFTstringcoproduct}
\end{prop}
\begin{proof}
The claim follows from the description of the coproduct $\bbDelta$ from \cref{prop:TFTcomultiplication} as well as the description of the counit $\coev\circ \coev^\R\rightarrow \id$ given in \cref{prop:LocSyscoevadjoint}.
\end{proof}

In terms of marked spaces, the coproduct $\bbDelta$ can be described as the composite
\begin{center}
	\begin{tikzpicture}
            \draw (-0.8, 0) node {$\bbDelta\colon$};
		\draw (0, 0) circle (0.4cm);
		\draw[red, fill=red] (0.4, 0) circle (2pt);
		\draw (0.8, 0) node {$\xrightarrow{\epsilon_p}$};
		\draw (1.6, 0) circle (0.4cm);
		\draw (2.4, 0) circle (0.4cm);
		\draw (3.2, 0) node {$\xrightarrow{\epsilon_\Delta}$};
		\draw (4, 0) circle (0.4cm);
	\end{tikzpicture}
\end{center}

By \cref{prop:S1nvalue} and \cref{prop:stringS1nvalue} we have
\[\fZ(\LocSys(M))\cong Z(S^1_1)\cong \C_\bullet(LM)[-d],\qquad \dim(\LocSys(M))\cong \C_\bullet(LM).\]
Let $\widetilde{\dim}(\LocSys(M))\cong \C_\bullet(M)$ with $\widetilde{\dim}(\LocSys(M))\rightarrow \dim(\LocSys(M))$ given by the inclusion of constant loops $i\colon \C_\bullet(M)\rightarrow \C_\bullet(LM)$. Its cofiber is given by the complex of relative chains $\overline{\dim}(\LocSys(M)) \cong \C_\bullet(LM, M)$.

By \cref{prop:stringelbowvalue} the class $[\coev]\in \C_\bullet(LM\times LM)$ is equivalent to $(L\Delta)[\Z_M]$. Further, using the Pontryagin--Thom lift of the $\HH$ Euler characteristic, the class $[\Z_M]$ is equivalent to $i(e(M))$. Using the commutative diagram
\[
\xymatrix{
M \ar^{\Delta}[r] \ar^{i}[d] & M\times M \ar^{i\times i}[d] \\
LM \ar^-{L\Delta}[r] & LM\times LM
}
\]
we see that $[\coev]$ admits a lift along $(i\times i)\colon \C_\bullet(M)\otimes \C_\bullet(M)\rightarrow \C_\bullet(LM)\otimes \C_\bullet(LM)$. Thus, we can apply the construction of the secondary coproduct from \cref{def:secondarycoproduct} to obtain an operation
\begin{equation}
\vee\colon \H_{\bullet+d-1}(LM)\longrightarrow \H_\bullet(\C_\bullet(LM, M)\otimes \C_\bullet(LM, M))\cong \H_\bullet(LM\times LM, M\times LM\cup LM\times M).
\label{eq:stringsecondarycoproduct}
\end{equation}

\begin{thm}
The secondary coproduct \eqref{eq:stringsecondarycoproduct} coincides with the loop coproduct from \cref{def:stringcoproduct}.
\label{thm:TFTHWequivalence}
\end{thm}

\subsection{Proof of \texorpdfstring{\cref{thm:TFTHWequivalence}}{theorem \ref{thm:TFTHWequivalence}}}

We begin by slightly rephrasing the definition of the loop coproduct. Split diagram \eqref{eq:F8intersection} into two commutative diagrams
\begin{equation}
	\xymatrix{
		LM \ar^{J_i}[r] \ar^{\ev}[d] & LM \ar^{(\ev_0, \ev_{1/2})}[d] \\
		M \ar^-{\Delta}[r] & M\times M
	}
	\label{eq:F8halfintersection}
\end{equation}
for $i=0, 1$ for each $LM$ term in the upper left corner.

Consider the diagram
\begin{equation}
\xymatrix{
	\C_\bullet(LM)[-d] \ar^-{e(M)}[r] \ar^{J_0}[d] & \C_\bullet(LM) \ar^{\iota_0}[d] \ar^-{\id\times \ev}[r] & \C_\bullet(LM\times M) \ar^{\id\times i}[d] \\
	\C_\bullet(LM)[-d] \ar[r] & \C_\bullet(\FE) \ar[r] & \C_\bullet(LM\times LM) \\
	\C_\bullet(LM)[-d] \ar^-{e(M)}[r] \ar_{J_1}[u] & \C_\bullet(LM) \ar_{\iota_1}[u] \ar^-{\ev\times\id}[r] & \C_\bullet(M\times LM) \ar_{i\times \id}[u]
}
\label{eq:stringcoproducttwohalves}
\end{equation}
where the two squares on the left are given by the relative intersection squares \eqref{eq:relintersectionhomology} for \eqref{eq:F8halfintersection} for $i=0,1$. The middle row describes the coproduct $\bbDelta$ by \cref{prop:TFTstringcoproduct} and the map
\[\H_{\bullet+d-1}(LM)\longrightarrow \H_\bullet(LM\times LM, LM\times M\cup M\times LM)\]
in the definition of the loop coproduct is obtained by summing over the top and bottom rows and taking the cofiber of the resulting map into the middle row and postcomposing with the map induced by $(LM \times LM, LM \times M \sqcup M \times LM) \to (LM \times LM, LM \times M \cup M \times LM)$.

We will next make explicit the secondary coproduct. It is given by a homotopy
\begin{equation}
	(a\otimes 1)\wedge ((i\times i)\circ\Delta) e(M)\sim \bbDelta(a)\sim ((i\times i)\circ\Delta) e(M)\wedge (1\otimes a)
	\label{eq:TFTcompletehomotopy}
\end{equation}
constructed from the following pieces:
\begin{enumerate}
	\item The homotopy \eqref{eq:TFThomotopy} given by \cref{fig:TFThomotopy} is
	\[(a\otimes 1)\wedge \bbDelta(1)\sim \bbDelta(a)\sim \bbDelta(1)\wedge (1\otimes a).\]
	\item Using \cref{prop:stringelbowvalue} we obtain homotopies
	\[(a\otimes 1)\wedge (L\Delta)[\Z_M]\sim (a\otimes 1)\wedge \bbDelta(1)\]
	and
	\[(L\Delta)[\Z_M]\wedge (1\otimes a)\sim \bbDelta(1)\wedge (1\otimes a).\]
	\item Applying the lift of the $\HH$ Euler characteristic we obtain homotopies
	\[(a\otimes 1)\wedge (L\Delta) i(e(M))\sim (a\otimes 1)\wedge (L\Delta)[\Z_M]\]
	and
	\[(L\Delta)i(e(M))\wedge (1\otimes a)\sim (L\Delta)[\Z_M]\wedge (1\otimes a).\]
	\item Using the commutative diagram
	\[
	\xymatrix{
		M \ar^{i}[d] \ar^-{\Delta}[r] & M\times M \ar^{i\times i}[d] \\
		LM \ar^-{L\Delta}[r] & LM\times LM
	}
	\]
	we obtain homotopies
	\[(a\otimes 1)\wedge ((i\times i)\circ \Delta)(e(M))\sim (a\otimes 1)\wedge (L\Delta) i(e(M)),\]
        where the left-hand side is given by a morphism $\C_\bullet(LM)[-d]\rightarrow \C_\bullet(LM)\otimes \C_\bullet(M)$, and
	\[((i\times i)\circ \Delta)(e(M))\wedge (1\otimes a)\sim (L\Delta)i(e(M))\wedge (1\otimes a),\]
        where the left-hand side is given by a morphism $\C_\bullet(LM)[-d]\rightarrow \C_\bullet(M)\otimes \C_\bullet(LM)$.
\end{enumerate}

Thus, we see that both the loop coproduct and the secondary coproduct are given by a homotopy commutative diagram of the form
\[
\xymatrix{
 & \C_\bullet(LM)\otimes \C_\bullet(M) \ar^{\id\otimes i}[dr] & \\
\C_\bullet(LM)[-d] \ar[ur] \ar[dr] \ar^{\bbDelta}[rr] && \C_\bullet(LM)\otimes \C_\bullet(LM) \\
& \C_\bullet(M)\otimes \C_\bullet(LM) \ar_{i\otimes \id}[ur] &
}
\]
where the top and the bottom 2-morphisms are related by permuting the factors. So, to prove \cref{thm:TFTHWequivalence} it is enough to identify the relevant two-morphisms appearing at the top of the loop coproduct and the secondary coproduct.

It will be convenient to use the following notation. Consider a morphism $\phi\colon \Z_M[-d]\rightarrow \ev_\sharp \Z_{LM} \otimes \C_\bullet(LM)$. By adjunction it gives rise to a morphism $\phi(1)\colon \Z\rightarrow \C_\bullet(LM)\otimes \C_\bullet(LM)$. It also gives rise to an operation
\[\phi(a)\colon \C_\bullet(LM)[-d]\xrightarrow{\id\otimes \phi}\C_\bullet(\FE)\otimes \C_\bullet(LM)\xrightarrow{m\otimes\id} \C_\bullet(LM)\otimes \C_\bullet(LM)\]
as well as the morphism
\begin{align*}
		(a\otimes 1)\wedge\phi(1)\colon &\C_\bullet(LM)[-d] \xrightarrow{\id\otimes\phi(1)}\\
		&\C_\bullet(LM)[-d] \otimes \C_\bullet(LM)\otimes \C_\bullet(LM) \xrightarrow{\epsilon_p}\\
		&\C_\bullet(\FE)\otimes \C_\bullet(LM) \xrightarrow{m\otimes\id} \\
		& \C_\bullet(LM)\otimes \C_\bullet(LM).
\end{align*}

There is a homotopy $\phi(a)\sim (a\otimes 1)\wedge\phi(1)$ given pictorially by
{
	\importpiclib
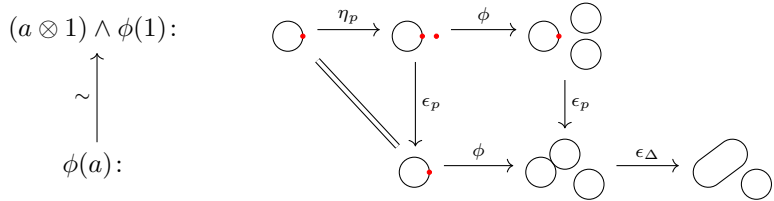
	\[
	\begin{tikzcd}
		(a \otimes 1)\wedge\phi(1) \colon & \phhc \ar[r, "\eta_p"]\ar[dr, equal]& \phhc \ \pzeta \ar[r, "\phi"]\ar[d, "\epsilon_p"] & \phhc \ \pohoh \ar[d, "\epsilon_p"] & \\
		\phi(a)\colon \ar[u, "\sim"] && \phhc \ar[r, "\phi"] & \pcircircircsep \ar[r, "\epsilon_\Delta"] & \povalcirc
	\end{tikzcd}
	\]
}
Traversing this picture along the top part we obtain $(a\otimes 1)\wedge\phi(1)$ and traversing this picture along the bottom part we obtain $\phi(a)$. The left two-cell is given by the cusp isomorphism witnessing the adjunction $p^*(-)\dashv p_\sharp(\zeta_M\otimes(-))$ and the middle two-cell is given by the interchange two-cell.

The strategy of the proof is as follows:
\begin{itemize}
	\item \textbf{Step 1}. We will construct four homotopic morphisms
	\[\phi_1,\phi_2,\phi_3,\phi_4\colon \Z_M[-d]\rightarrow \ev_\sharp \Z_{LM} \otimes \C_\bullet(LM),\]
        where $\phi_4$ factors through $\Z_M[-d]\rightarrow \Z_M\otimes \C_\bullet(LM)\xrightarrow{i\otimes \id} \ev_\sharp \Z_{LM} \otimes \C_\bullet(LM)$. The homotopy $\phi(a)\sim (a\otimes 1)\wedge\phi(1)$ gives rise to a homotopy commuting square
	\begin{equation}\label{diag:squarehomotopy}
		\begin{tikzcd}
			(a\otimes 1)\wedge\phi_1(1) \ar[r, "\sim"]\ar[d, "\sim"] & (a\otimes 1)\wedge\phi_2(1) \ar[r, "\sim"]\ar[d, "\sim"] & (a\otimes 1)\wedge\phi_3(1) \ar[r, "\sim"]\ar[d, "\sim"] & (a\otimes 1)\wedge\phi_4(1) \ar[d, "\sim"] \\
			\phi_1(a) \ar[r, "\sim"] &\phi_2(a) \ar[r, "\sim"] &\phi_3(a) \ar[r, "\sim"] &\phi_4(a).
		\end{tikzcd}
	\end{equation}
	in $\Hom_{\Mod_\Z}(\C_\bullet(LM)[-d], \C_\bullet(LM)\otimes \C_\bullet(LM))$.

	\item \textbf{Step 2}. We will construct a homotopy $\bbDelta(a)\sim \phi_1(a)$.
	
	\item \textbf{Step 3}. We will identify the homotopy
	\[\bbDelta(a)\sim\phi_1(a)\sim (a\otimes 1)\wedge\phi_1(1)\sim (a\otimes 1)\wedge\phi_2(1)\sim (a\otimes 1)\wedge\phi_3(1)\sim (a\otimes 1)\wedge\phi_4(1)\]
	with a half of the homotopy \eqref{eq:TFTcompletehomotopy}.
	
	\item \textbf{Step 4}. We will identify the homotopy
	\[\bbDelta(a)\sim \phi_1(a)\sim \phi_2(a)\sim \phi_3(a)\sim \phi_4(a)\]
	with the bottom half of diagram in \eqref{eq:stringcoproducttwohalves} describing the loop coproduct.
\end{itemize}

\subsubsection{Step 1}

We now define the homotopic elements $\phi_1, \phi_2, \phi_3, \phi_4 \in \Hom_{\LocSys(M)}(\Z_M[-d], \ev_\sharp \Z_{LM} \otimes \C_\bullet(LM))$ by
\[
\phi_1 \colon \Z_M[-d] \xrightarrow{\eta_\Delta} \ev_\sharp \Z_{LM}[-d] \xrightarrow{\epsilon_p} (\ev_\sharp \Z_{LM})^{\otimes 2} \xrightarrow{\epsilon_\Delta} \ev_\sharp \Z_{LM} \otimes \C_\bullet(LM)
\]
\[
\phi_2 \colon \Z_M[-d] \xrightarrow{\epsilon_p} \ev_\sharp \Z_{LM} \xrightarrow{\eta_\Delta} (\ev_\sharp \Z_{LM})^{\otimes 2} \xrightarrow{\epsilon_\Delta} \ev_\sharp \Z_{LM} \otimes \C_\bullet(LM)
\]
\[
\phi_3 \colon \Z_M[-d] \xrightarrow{e(M)} \Z_M \xrightarrow{\eta_\Delta\sqcup \eta_\Delta} (\ev_\sharp \Z_{LM})^{\otimes 2} \xrightarrow{\epsilon_\Delta} \ev_\sharp \Z_{LM} \otimes \C_\bullet(LM)
\]
\[
\phi_4 \colon \Z_M[-d] \xrightarrow{e(M)} \Z_M \xrightarrow{\epsilon_\Delta} \Z_M \otimes \C_\bullet(M) \xrightarrow{\eta_\Delta\sqcup \eta_\Delta} \ev_\sharp \Z_{LM} \otimes \C_\bullet(LM)
\]
Note that with the previous notation we have
\[
\phi_1(1) = \bbDelta(1) = [\coev] = [\Delta_\sharp \circ p^*]
\]
\[
\phi_2(1) = [\Delta_\sharp] \circ [p^*],
\]
\[
\phi_3(1) = [\Delta_\sharp] \circ i \circ e(M),
\]
\[
\phi_4(1) = ((i\times i)\circ\Delta)e(M)
\]
where $[\Delta_\sharp] = L\Delta$ by \cref{prop:LocSysTHH}. Pictorially, these morphisms are given as follows:
{
	\importpiclib
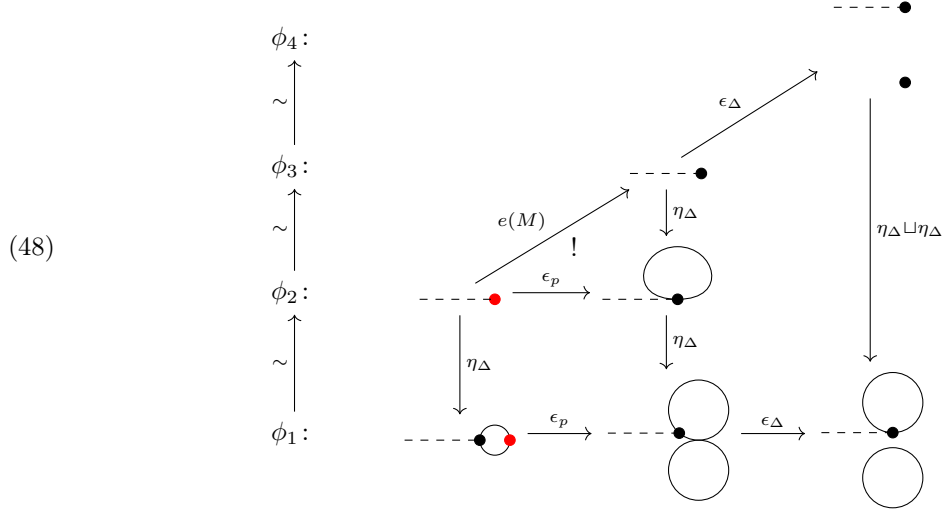
	\begin{equation}
		\begin{tikzcd}\label{diag:phi1 to phi4}
			\phi_4 \colon & & & \ppointmun \ar[ddd, "\eta_\Delta\sqcup\eta_\Delta"] & \\
			\phi_3 \ar[u, "\sim"] \colon & & \ppointm \ar[ur, "\epsilon_\Delta"] \ar[d, "\eta_\Delta"] && \\
			\phi_2 \ar[u, "\sim"] \colon & \pzetam \ar[d, "\eta_\Delta"] \ar[ur, "e(M)"] \ar[phantom, ur, ""{name=T}] \ar[r, "\epsilon_p"] & \phhm \ar[phantom, to=T, "!"] \ar[d, "\eta_\Delta"] && \\
			\phi_1 \ar[u, "\sim"] \colon & \phhcm \ar[r, "\epsilon_p"]& \pfigeightbm \ar[r, "\epsilon_\Delta"] & \ptwocircm
		\end{tikzcd}
	\end{equation}
}

The triangle labeled by ! is provided by the lift of the $\HH$ Euler characteristic. The other two unlabeled two-cells are the interchange two-cells.

\subsubsection{Step 2}

The homotopy $\bbDelta(a)\sim \phi_1(a)$ is given pictorially by
{
	\importpiclib
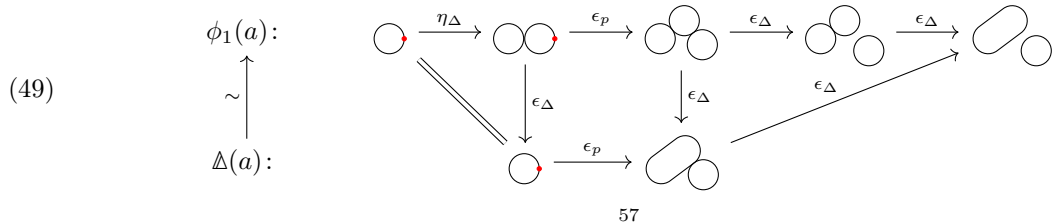
	\begin{equation}
		\begin{tikzcd}\label{diag:ship hull}
			\phi_1(a)\colon & \phhc \ar[r, "\eta_\Delta"] \ar[dr, equal] & \pcirc\phhc \ar[r, "\epsilon_p"]\ar[d, "\epsilon_\Delta"] & \pcircircirc \ar[r, "\epsilon_\Delta"]\ar[d, "\epsilon_\Delta"] & \pcircircircsep \ar[r, "\epsilon_\Delta"] & \povalcirc \\
			\bbDelta(a) \colon \ar[u, "\sim"] && \phhc \ar[r, "\epsilon_p"] & \povalcircnsep \ar[urr, "\epsilon_\Delta"] &&
		\end{tikzcd}
	\end{equation}
}

\subsubsection{Step 3}

The composite
\[\bbDelta(a)\sim \phi_1(a)\sim (a\otimes 1)\wedge \phi_1(1)\sim (a\otimes 1)\wedge \phi_2(1)\sim (a\otimes 1)\wedge \phi_3(1)\sim (a\otimes 1)\wedge \phi_4(1)\]
in \cref{diag:squarehomotopy} is given by
{
	\importpiclib
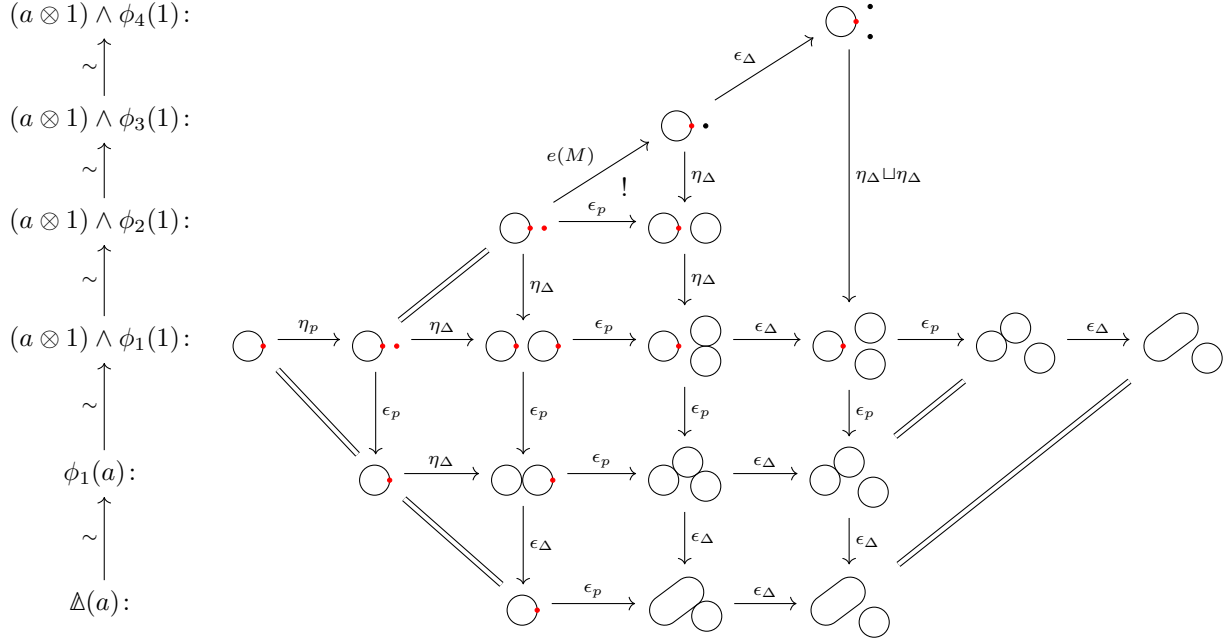
	\[
	\begin{tikzcd}
		(a \otimes 1)\wedge\phi_4(1) \colon &[-2em]&&&& \phhc \ \pdpoint  \ar[ddd, "\eta_\Delta\sqcup\eta_\Delta"]&&\\ 
		(a \otimes 1)\wedge\phi_3(1) \colon \ar[u, "\sim"] &     &                       &  &              \phhc \ \ppoint \ar[d, "\eta_\Delta"] \ar[ur, "\epsilon_\Delta"] &  &                  & \\
		(a \otimes 1)\wedge\phi_2(1)\colon \ar[u, "\sim"] &       &                       & \phhc \ \pzeta \ar[ur, "e(M)"{name=P}] \ar[r, "\epsilon_p"]\ar[d, "\eta_\Delta"]\ar[dl, equal]& \phhc \ \pcirc \ar[d, "\eta_\Delta"]  \ar[phantom, to=P, "!"] &  &                        &\\
		(a \otimes 1)\wedge\phi_1(1) \colon \ar[u, "\sim"] & \phhc \ar[r, "\eta_p"]\ar[dr, equal]& \phhc \ \pzeta \ar[r, "\eta_\Delta"]\ar[d, "\epsilon_p"] & \phhc \ \phhc \ar[r, "\epsilon_p"]\ar[d, "\epsilon_p"]         & \phhc \ \pfigeight \ar[r, "\epsilon_\Delta"]\ar[d, "\epsilon_p"] & \phhc \ \pohoh \ar[r, "\epsilon_p"] \ar[d, "\epsilon_p"] & \pcircircircsep \ar[r, "\epsilon_\Delta"] & \povalcirc \\
		\phi_1(a)\colon \ar[u, "\sim"]  & & \phhc \ar[r, "\eta_\Delta"] \ar[dr, equal] & \pcirc\phhc \ar[r, "\epsilon_p"]\ar[d, "\epsilon_\Delta"] & \pcircircirc \ar[r, "\epsilon_\Delta"]\ar[d, "\epsilon_\Delta"] & \pcircircircsep \ar[d, "\epsilon_\Delta"] \ar[ur, equal] && \\
		\bbDelta(a) \colon \ar[u, "\sim"] &&& \phhc \ar[r, "\epsilon_p"] & \povalcircnsep \ar[r, "\epsilon_\Delta"] & \povalcirc \ar[uurr, equal]&&
	\end{tikzcd}
	\]
}

The bottom three rows identify with the homotopy $\bbDelta(a)\sim (a\otimes 1)\wedge\bbDelta(1)$, and the top 4 rows identify with the homotopies
\[(a\otimes 1)\wedge\bbDelta(1)\sim(a\otimes 1)\wedge (L\Delta)[\Z_M]\sim (a\otimes 1)\wedge (L\Delta)i(e(M))\sim (a\otimes 1)\wedge ((i\times i)\circ\Delta)e(M)\]
in \eqref{eq:TFTcompletehomotopy}.

\subsubsection{Step 4}
Let us first spell out the homotopy
\[\bbDelta(a)\sim \phi_1(a)\sim \phi_2(a)\sim \phi_3(a)\sim \phi_4(a).\]
It is obtained by gluing in an extra circle at the dashed line in \eqref{diag:phi1 to phi4} and attaching the resulting diagram to \eqref{diag:ship hull}. We thus obtain
{
	\importpiclib
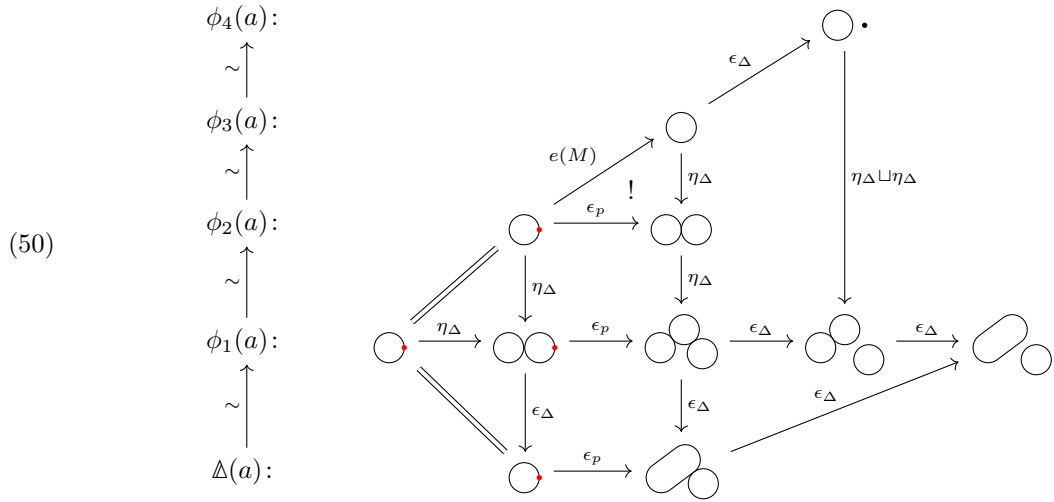
	\begin{equation}
		\begin{tikzcd}\label{diag:tftcopdiagsimplified}
			\phi_4(a) \colon & & & & \pcirc \ \ppoint \ar[ddd, "\eta_\Delta\sqcup\eta_\Delta"] & \\
			\phi_3(a) \ar[u, "\sim"] \colon & & & \pcirc \ar[ur, "\epsilon_\Delta"] \ar[d, "\eta_\Delta"] && \\
			\phi_2(a) \ar[u, "\sim"] \colon & & \phhc \ar[d, "\eta_\Delta"] \ar[ur, "e(M)"] \ar[phantom, ur, ""{name=T}] \ar[r, "\epsilon_p"] & \pcircirc \ar[phantom, to=T, "!"] \ar[d, "\eta_\Delta"] && \\
			\phi_1(a) \ar[u, "\sim"] \colon & \phhc \ar[r, "\eta_\Delta"] \ar[dr, equal] \ar[ur, equal] & \pcirc\phhc \ar[r, "\epsilon_p"]\ar[d, "\epsilon_\Delta"] & \pcircircirc \ar[r, "\epsilon_\Delta"]\ar[d, "\epsilon_\Delta"] & \pcircircircsep \ar[r, "\epsilon_\Delta"] & \povalcirc \\
			\bbDelta(a) \colon \ar[u, "\sim"] && \phhc \ar[r, "\epsilon_p"] & \povalcircnsep \ar[urr, "\epsilon_\Delta"] &&
		\end{tikzcd}
	\end{equation}
}

To compare this to the definition of the loop coproduct as in \eqref{eq:stringcoproducttwohalves} (more precisely, the lower two rows) let us first spell out the relative intersection square. Specializing the definition of the relative intersection square \eqref{eqn:relintersectionpreliminary} we obtain
{\importpiclib
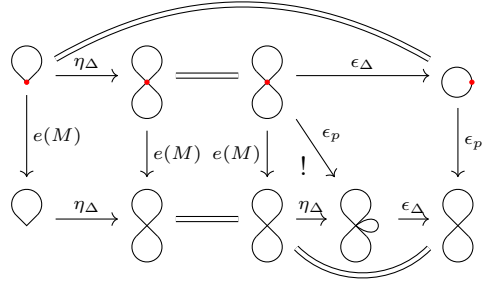
	\[
	\begin{tikzcd}[column sep=small]
		\ptopeightm \ar[rr, "\eta_\Delta"] \ar[rrrrrr, equal, bend left] \ar[d, "e(M)"] && \pxfigeightm \ar[rr, equal] \ar[d, "e(M)"] &&  \pxfigeightm \ar[rr, "\epsilon_\Delta"] \ar[d, "e(M)"'] \ar[dr, "\epsilon_p"] \ar[dr, phantom, ""{name=T}]  && \phhc \ar[d, "\epsilon_p"] \\
		\ptopeight \ar[rr, "\eta_\Delta"] && \pxfigeight \ar[rr, equal]  && \pxfigeight \ar[r, "\eta_\Delta"] \ar[phantom, to=T, "!"] \ar[rr, equal, bend right = 45] & \pxfigeightex \ar[r, "\epsilon_\Delta"] & \pxfigeight
	\end{tikzcd}
	\]
}
Here we used that the diagram \eqref{eqn:echarliftdiag} determining a lift of the $\HH$ Euler characteristic is equivalent to the triangle $!$ using the $\Delta_\sharp \dashv \Delta^*$ adjunction. Namely, \eqref{eqn:echarliftdiag} is given by
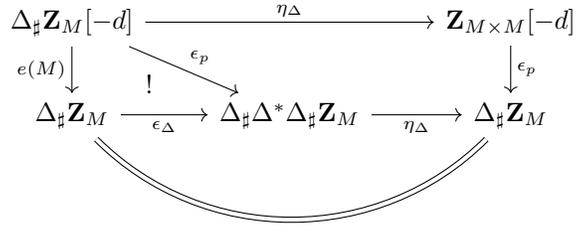
\[
\begin{tikzcd}
	\Delta_\sharp \Z_M[-d] \ar[d, "e(M)"'] \ar[rr, "\eta_\Delta"] \ar[dr, "\epsilon_p"] \ar[dr, phantom, ""{name=T}]& & \Z_{M\times M}[-d] \ar[d, "\epsilon_p"]\\
	\Delta_\sharp \Z_M \ar[r, "\epsilon_\Delta"'] \ar[phantom, to=T, "!"] \ar[rr, bend right = 45, equal] & \Delta_\sharp\Delta^*\Delta_\sharp\Z_M \ar[r, "\eta_\Delta"'] & \Delta_\sharp \Z_M
\end{tikzcd}
\]
Adding the lower right square of \eqref{eq:stringcoproducttwohalves} we obtain
{\importpiclib
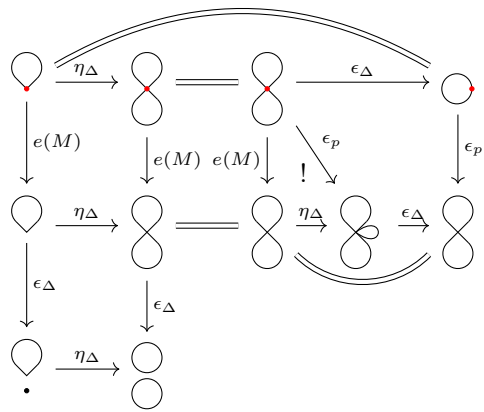
	\[
	\begin{tikzcd}[column sep=small]
		\ptopeightm \ar[rr, "\eta_\Delta"] \ar[rrrrrr, equal, bend left] \ar[d, "e(M)"] && \pxfigeightm \ar[rr, equal] \ar[d, "e(M)"] &&  \pxfigeightm \ar[rr, "\epsilon_\Delta"] \ar[d, "e(M)"'] \ar[dr, "\epsilon_p"] \ar[dr, phantom, ""{name=T}]  && \phhc \ar[d, "\epsilon_p"] \\
		\ptopeight \ar[rr, "\eta_\Delta"] \ar[d, "\epsilon_\Delta"] && \pxfigeight \ar[rr, equal]  \ar[d, "\epsilon_\Delta"]  && \pxfigeight \ar[r, "\eta_\Delta"] \ar[phantom, to=T, "!"]  \ar[rr, equal, bend right = 45] & \pxfigeightex \ar[r, "\epsilon_\Delta"] & \pxfigeight \\
		\ptopeightpoint \ar[rr, "\eta_\Delta"] && \pohoh &&&&
	\end{tikzcd}
	\]
}
Simplifying and rearranging the diagram slightly we obtain
{\importpiclib
	\begin{equation}
		\begin{tikzcd}[column sep=small, row sep = small]
			\ptopeightm \ar[rr, "\eta_\Delta"] \ar[rrrr, equal, bend left] \ar[dd, "e(M)"'] && \pxfigeightm \ar[rr, "\epsilon_\Delta"] \ar[dr, "\epsilon_p"] \ar[dd, "e(M)"'{name=C}] && \phhc \ar[d, "\epsilon_p"] \\
			&&& \pxfigeightex  \ar[phantom, to=C, "!"] \ar[r, "\epsilon_\Delta"] & \pxfigeight \\
			\ptopeight \ar[rr, "\eta_\Delta"] \ar[d, "\epsilon_\Delta"] && \pxfigeight \ar[ur, "\eta_\Delta"'] \ar[urr, equal, bend right] \ar[d, "\epsilon_\Delta"] &&\\
			\ptopeightpoint \ar[rr, "\eta_\Delta"] && \pohoh &&&&&&
		\end{tikzcd}
	\end{equation}
}

Tensoring the $!$-triangle with the morphism $\eta_\Delta\colon \ev_\sharp \Z_M\rightarrow \ev_\sharp\Z_M\otimes \ev_\sharp \Z_M$ we obtain a commutative prism. Two faces of the prism are given by the square as well as the $!$-triangle appearing above. We can hence replace this composite with the other three faces of the prism to obtain
{\importpiclib
	\begin{equation}
		\begin{tikzcd}[column sep=small, row sep = small]
			\ptopeightm \ar[rr, "\eta_\Delta"] \ar[rrrr, equal, bend left] \ar[dd, "e(M)"'{name=T}] \ar[dr, "\epsilon_p"] && \pxfigeightm \ar[rr, "\epsilon_\Delta"] \ar[dr, "\epsilon_p"] && \phhc \ar[d, "\epsilon_p"] \\
			& \ptopeightex \ar[to=T, phantom, "!"] \ar[rr, "\eta_\Delta"] && \pxfigeightex \ar[r, "\epsilon_\Delta"] & \pxfigeight \\
			\ptopeight \ar[rr, "\eta_\Delta"] \ar[ur, "\eta_\Delta"] \ar[d, "\epsilon_\Delta"] && \pxfigeight \ar[ur, "\eta_\Delta"] \ar[urr, equal, bend right] \ar[d, "\epsilon_\Delta"] &&\\
			\ptopeightpoint \ar[rr, "\eta_\Delta"] && \pohoh &&&&&&
		\end{tikzcd}
	\end{equation}
}
which finally coincides with diagram \eqref{diag:tftcopdiagsimplified}.

\printbibliography

\end{document}